\documentclass[10pt,headings=standardclasses,DIV=15]{scrartcl}
\usepackage[utf8]{inputenc}
\usepackage[english]{babel}
\usepackage{amsmath}
\usepackage{amsfonts,amssymb,mathrsfs,textcomp,amsthm}
\usepackage[mathscr]{eucal}
\usepackage{tikz}
\usepackage{tikz-3dplot}
\usepackage{pgfplots}
\usepackage{subfig}
\usepackage{hyperref}
\usepackage[all]{hypcap}
\usepackage[capitalize,nameinlink]{cleveref}

\crefname{section}{Section}{Sections}
\crefname{subsection}{Subsection}{Subsections}
\Crefname{section}{Section}{Sections}
\Crefname{subsection}{Subsection}{Subsections}
\Crefname{figure}{Figure}{Figures}
\crefformat{equation}{\textup{#2(#1)#3}}
\crefrangeformat{equation}{\textup{#3(#1)#4--#5(#2)#6}}
\crefmultiformat{equation}{\textup{#2(#1)#3}}{ and \textup{#2(#1)#3}}
{, \textup{#2(#1)#3}}{, and \textup{#2(#1)#3}}
\crefrangemultiformat{equation}{\textup{#3(#1)#4--#5(#2)#6}}%
{ and \textup{#3(#1)#4--#5(#2)#6}}{, \textup{#3(#1)#4--#5(#2)#6}}{, and \textup{#3(#1)#4--#5(#2)#6}}
\Crefformat{equation}{#2Equation~\textup{(#1)}#3}
\Crefrangeformat{equation}{Equations~\textup{#3(#1)#4--#5(#2)#6}}
\Crefmultiformat{equation}{Equations~\textup{#2(#1)#3}}{ and \textup{#2(#1)#3}}
{, \textup{#2(#1)#3}}{, and \textup{#2(#1)#3}}
\Crefrangemultiformat{equation}{Equations~\textup{#3(#1)#4--#5(#2)#6}}%
{ and \textup{#3(#1)#4--#5(#2)#6}}{, \textup{#3(#1)#4--#5(#2)#6}}{, and \textup{#3(#1)#4--#5(#2)#6}}
\crefdefaultlabelformat{#2\textup{#1}#3}

\usetikzlibrary{math} 
\usepgfplotslibrary{fillbetween}
\usetikzlibrary{quotes,angles,babel}
\hypersetup{
    bookmarks=true,
    unicode=false,
    pdfmenubar=true,
    pdffitwindow=false,
    pdfstartview={FitH},
    pdftitle={Edge detection with trigonometric polynomial shearlets},
    pdfauthor={Kevin Schober},
    pdfcreator={Kevin Schober},
    pdfnewwindow=true,
	colorlinks=true,
	linkcolor=blue,
	citecolor=blue,
	filecolor=blue,
	urlcolor=blue
}
\newtheorem{theorem}{Theorem}[section]
\newtheorem{lemma}{Lemma}[section]

\theoremstyle{definition}
\newtheorem{definition}{Definition}[section]

\newcommand{\norm}[1]{\left\lVert#1\right\rVert}
\newcommand{\abs}[1]{\left\lvert#1\right\rvert}
\makeatletter
\renewenvironment{thebibliography}[1]
     {\section*{\refname}%
      \@mkboth{\MakeUppercase\refname}{\MakeUppercase\refname}%
      \list{\@biblabel{\@arabic\c@enumiv}}%
           {\settowidth\labelwidth{\@biblabel{#1}}%
            \leftmargin2em
            \itemindent0em
			\itemsep0.3em
            \@openbib@code
            \usecounter{enumiv}%
            \let\p@enumiv\@empty
            \renewcommand\theenumiv{\@arabic\c@enumiv}}%
      \sloppy
      \clubpenalty4000
      \@clubpenalty \clubpenalty
      \widowpenalty4000%
      \sfcode`\.\@m}
     {\def\@noitemerr
       {\@latex@warning{Empty `thebibliography' environment}}%
      \endlist}
\makeatother

\makeatletter
\let\@fnsymbol\@arabic
\makeatother

\title{Edge detection with trigonometric polynomial shearlets}
\date{}
\author{Kevin Schober\thanks{Corresponding author} \textsuperscript{,}\thanks{Institute of Mathematics, University of L\"ubeck, Ratzeburger Allee 160, D-23562 L\"ubeck, Germany.\newline
E-mail: \href{mailto:schober@math.uni-luebeck.de}{schober@math.uni-luebeck.de} (K. Schober), \href{mailto:prestin@math.uni-luebeck.de}{prestin@math.uni-luebeck.de} (J. Prestin)}
 \and J\"urgen Prestin\footnotemark[2]
 \and Serhii A. Stasyuk\thanks{Institute of Mathematics, National Academy of Sciences of Ukraine, 01024, Ukraine, Kyiv, 3, Tereschenkivska street.\newline
E-mail: \href{mailto:stasyuk@imath.kiev.ua}{stasyuk@imath.kiev.ua}}}
 
\begin{document}

\maketitle

\begin{abstract}  In this paper we show that certain trigonometric polynomial shearlets which are special cases of directional de la Vall\'{e}e Poussin type wavelets are able to detect singularities along boundary curves of periodic characteristic functions. Motivated by recent results for discrete shearlets in two dimensions, we provide lower and upper estimates for the magnitude of corresponding inner products. In the proof we use localization properties of trigonometric polynomial shearlets in the time and frequency domain and, among other things, bounds for certain Fresnel integrals. Moreover, we give numerical examples which underline the theoretical results.
\end{abstract} 



\smallskip

\section{Introduction}
 
In many applications in signal or image processing great importance is attached to precise information about the location and order of singularities of signals. In one dimension this corresponds to functions which are smooth apart from pointwise singularities. Many authors discussed this problem when the Fourier coefficients of a periodic function are given, see e.g. \cite{batenkov:fourier,eckhoff:detection,gelb:detection,mhaskar:detection,shi:determination,tadmor:filters}.

Because of their localization properties in the time and frequency domain, wavelet expansions provide a powerful tool for detecting and analyzing point discontinuities in one or more dimensions \cite{jaffard:pointwise,mallat:detection}. The reason is that only very few wavelet coefficients of translates near the location of the singularity are large in magnitude, while all other wavelet coefficients corresponding to translates which are further away from the point discontinuity decay rapidly. A framework for univariate periodic wavelets was investigated by several authors \cite{koh:uni_periodic_wavelets,narcowich:uni_periodic_wavelets,plonka:uni_periodic_wavelets,prestin:uni_periodic_wavelets} and some of these constructions were successfully used for the detection of pointwise singularities of periodic functions \cite{mhaskar:detection}.

In two dimensions the situation is more complex since not only point singularities can occur but also discontinuities along curves. To deal with these types of singularities, along with many other constructions, the theory of the continuous shearlet transform was developed \cite{dahlke:space,guo:construction,kutyniok:book} and defined as the mapping
\begin{equation*}
	f\rightarrow\mathcal{SH}_{\psi}f(a,s,p)=\left\langle f,\psi_{a,s,\mathbf{p}} \right\rangle
\end{equation*} 
with scale parameter $a>0$, orientation parameter $s\in \mathbb{R}$ and translation parameter $\mathbf{p}\in \mathbb{R}^2$. The shearlets $\psi_{a,s,\mathbf{p}}$ are well localized functions in the time and frequency domain and provide directional sensitivity controlled by the parameter $s$. It turned out that continuous shearlets provide a suitable tool to precisely describe different types of discontinuities along curves with asymptotic estimates. In particular, let $T\subset \mathbb{R}^2$ be a set with a smooth boundary $\partial T$. If either $\mathbf{p}\notin\partial T$ or if $s = s_0$ does not correspond to the normal direction of $\partial T$ at $\mathbf{p}$, then 
\begin{equation}\label{eq:lower_bound_cont}
	\lim\limits_{a\rightarrow 0^+}a^{-N}\mathcal{SH}_{\psi}\chi_T(a,s_0,\mathbf{p})=0\qquad\text{for all}\;N>0.
\end{equation}
Otherwise, if $\mathbf{p}\in\partial T$ and $s = s_0$ corresponds to the normal direction of $\partial T$ at $\mathbf{p}$, then
\begin{equation}\label{eq:upper_bound_cont}
	\lim\limits_{a\rightarrow 0^+}a^{-3/4}\mathcal{SH}_{\psi}\chi_T(a,s_0,\mathbf{p})=C>0.
\end{equation}
The results were shown for continuous shearlets, which are compactly supported in the time \cite{kutyniok:edges_compactly} or frequency domain \cite{grohs:resolution,labate:detection_continuous2,labate:detection_continuous,kutyniok:resolution}. 

Based on these theoretical results, practical applications for the detection of edges in images were developed \cite{yi:shear_edge}.
Therefore, discrete frames of shearlets were constructed by sampling the parameters of the continuous shearlet systems in a suitable way \cite{kutyniok:construction}. Based on the result for curvelets \cite{candes:curvelets}, it was possible to show that discrete shearlet systems are essentially optimal for the sparse approximation of so-called cartoon-like functions \cite{labate:sparse}. This result implies the upper estimate 
\begin{equation}\label{eq:upper_bound_discrete}
	\abs{\left\langle f,\psi_{j,\ell,\mathbf{k}} \right\rangle}\leq C\,2^{-3j/2}
\end{equation} 
for some constant $C>0$ independent of the scale parameter $j$. In \cite{labate:detection}, the authors showed the existence of a lower estimate $\abs{\left\langle \chi_T,\psi_{j,\ell,\mathbf{k}_\ell} \right\rangle}\geq C\,2^{-3j/2}$ if the localization and orientation of the discrete shearlet are sufficiently close to the boundary curve and its normal direction. These two estimates are the discrete
analogs of \cref{eq:lower_bound_cont} and \cref{eq:upper_bound_cont} implying that discrete shearlets are able to detect step discontinuities along boundary curves of characteristic functions.

The framework of multivariate periodic wavelets was developed for example in \cite{goh:multi_periodic_wavelets,maksimenko:multi_periodic_wavelets}. In \cite{bergmann:dlVP,langemann:multi_periodic_wavelets} the corresponding wavelet functions were trigonometric polynomials of Dirichlet and de la Vall\'{e}e Poussin-type, which can be well localized in the time and frequency domain. The construction allows for fast decomposition algorithms \cite{bergmann:FFTtorus} with many different dilation matrices on each scale, including shearing. This gives rise to directional decompositions of the frequency domain similar to the tilling of the frequency plane in the case of discrete shearlet systems \cite{candes:curvelets,labate:detection}.

In this paper we use the latter construction to prove two main theorems which provide upper and lower bounds similar to \cite{labate:detection}, but this time for a discrete system of periodic de la Vall\'{e}e Poussin-type wavelets that are trigonometric polynomials. The upper estimate in \cref{thm:upper_estimate} refines the estimate \cref{eq:upper_bound_discrete} by including the localization and orientation dependency of the shearlet coefficients in the decay estimate. \cref{thm:lower_estimate} is the analog of the main result in \cite{labate:detection} and implies that the constructed trigonometric polynomial shearlets in this paper are able to detect step discontinuities along boundary curves of periodic functions.

The paper is organized as follows. We start with the construction of a special case of directional de la Vall\'{e}e Poussin wavelets in \cref{sec:trigonometric_polynomial_shearlets} which we will call trigonometric polynomial shearlets and state the two main theorems of this paper in \cref{sec:main_results}. \Cref{sec:numerical_examples} provides a numerical example to illustrate the main results. After some preliminaries, \cref{sec:auxiliary_results} is devoted to formulate and to prove all auxiliary lemmata which are needed for the proof of the main results. In \cref{sec:proof_of_the_main_results}, the proofs for the upper and lower bounds of the corresponding inner products are given. Finally, we consider the extension of the construction to higher dimensions and discuss possible results in the case of corner points and smooth functions.

\section{Trigonometric polynomial shearlets} \label{sec:trigonometric_polynomial_shearlets} 

If a nonnegative and even function $g:\mathbb{R}\rightarrow \mathbb{R}$ with $\mathrm{supp}\,g=\left(-\frac{2}{3},\frac{2}{3}\right)$ satisfies the property
\begin{equation*}
	\sum\limits_{z\in \mathbb{Z}}g(x+z)=1\;\; \text{for all}\;\;x\in \mathbb{R},
\end{equation*}
we call it window function and write $g\in \mathcal{W}$. If additionally $g$ is $q$-times continuously differentiable we use the notation $g\in \mathcal{W}^q$. We remark that a consequence of the properties of a window function is $g(x)=1$ for $x\in\left(-\frac{1}{3},\frac{1}{3}\right)$ and $g$ is monotonically increasing for $x\in\left(-\frac{2}{3},-\frac{1}{3}\right]$ and monotonically decreasing for $x\in\left[\frac{1}{3},\frac{2}{3}\right)$. Further we introduce functions $\widetilde{g}:\mathbb{R}\rightarrow \mathbb{R}$ given by $\widetilde{g}(x)\mathrel{\mathop:}=g\left( \frac{x}{2} \right)-g(x)$.
	 
As an example of a window function we consider
	\begin{equation*}
		r(x)=\begin{cases}
								\mathrm{e}^{-b/x^2}, &\text{for } x>0,\\
								0, &\text{for } x\leq 0,
							\end{cases}
	\end{equation*} 
	where $b>0$ and define $s(x)=r\left(\frac{2}{3}+x\right)\,r\left(\frac{2}{3}-x\right)$. Then for
	\begin{equation}\label{eq:exp_window}
		g_b(x)=\frac{s(x)}{\sum\limits_{k\in \mathbb{Z}}s(x+k)}
	\end{equation}
	we have $g_b\in \mathcal{W}^\infty$ and this function is visualized in \cref{fig:support} for $b=0.025$.\\
	
	We denote two-dimensional vectors by $\mathbf{x}=(x_1,x_2)^{\mathrm{T}}$ with the inner product $\mathbf{x}^{\mathrm{T}}\mathbf{y}\mathrel{\mathop:}= x_1\,y_1+x_2\,y_2$ and the usual Euclidean norm $\abs{\mathbf{x}}_2\mathrel{\mathop:}=\sqrt{\mathbf{x}^{\mathrm{T}}\mathbf{x}}$. Let $C(A)$ denotes the space of all continuous functions on a set $A\subseteq \mathbb{R}^2$ equipped with the norm $\norm{f}_{A,\infty}\mathrel{\mathop:}=\norm{f}_{C(A)}\mathrel{\mathop:}=\sup\limits_{\mathbf{x}\in A}\abs{f(\mathbf{x})}$. For $\mathbf{x}\in \mathbb{R}^2$ and $\mathbf{r}=(r_1,r_2)^{\mathrm{T}}\in \mathbb{N}_0^2$ and a sufficiently smooth function $f$ we use the notation 
	\begin{equation*}
		\partial^{\mathbf{r}}f(\mathbf{x})\mathrel{\mathop:}= \frac{\partial^{r_1+r_2}}{\partial x_1^{r_1}\partial x_2^{r_2}}f(\mathbf{x})
	\end{equation*}
	and the space of all $q$-times continuously differentiable compactly supported functions will be denoted by
\begin{equation*}
	C^q_0(A)\mathrel{\mathop:}=\left\lbrace f:A\rightarrow \mathbb{R}:\partial^\mathbf{r}f\in C(A)\;\text{for all}\;\mathbf{r}\in \mathbb{N}_0^2\; \text{with}\;r_1+r_2\leq q,\,\abs{\mathrm{supp}\,f}<\infty\right\rbrace
\end{equation*}
with the norm
\begin{equation*}
	\norm{f}_{C^q}\mathrel{\mathop:}=\norm{f}_{C^q(A)}\mathrel{\mathop:}=\sup\limits_{r_1+r_2\leq q}\,\sup\limits_{\mathbf{x}\in A}\abs{\partial^\mathbf{r}f(\mathbf{x})}.
\end{equation*}
	 For $i\in\lbrace \mathrm{h},\mathrm{v}\rbrace$ we consider bivariate horizontal (vertical) window functions $\Psi^{(i)}:\mathbb{R}^2\rightarrow \mathbb{R}$ given by
	\begin{equation*}
		\Psi^{(\mathrm{h})}(\mathbf{x})\mathrel{\mathop:}=\widetilde{g}(x_1)\,g(x_2),\qquad\qquad \Psi^{(\mathrm{v})}(\mathbf{x})\mathrel{\mathop:}=g(x_1)\,\widetilde{g}(x_2).
	\end{equation*}	 
	We remark that for $g\in \mathcal{W}^q$ we have $\Psi^{(i)}\in C^q_0(\mathbb{R}^2)$ and in this case use the notation $\Psi^{(i)}\in \mathcal{W}_2^q$. From the support properties of the function $g\in \mathcal{W}$ it follows that
\begin{align*}
	&\mathrm{supp}\,\Psi^{(\mathrm{h})}=\left( \left( -\frac{4}{3},-\frac{1}{3}\right)\cup\left( \frac{1}{3},\frac{4}{3}\right) \right)\times\left( -\frac{2}{3},\frac{2}{3}\right),\\
	&\mathrm{supp}\,\Psi^{(\mathrm{v})}=\left( -\frac{2}{3},\frac{2}{3}\right)\times\left(\left( -\frac{4}{3},-\frac{1}{3}\right)\cup\left( \frac{1}{3},\frac{4}{3}\right)\right).
\end{align*}
	For even $j\in \mathbb{N}_0$ and $\ell\in \mathbb{Z}$ with $\lvert\ell\rvert\leq 2^{j/2}$ we define the matrices
		\begin{equation}\label{eq:N_jl}
			\mathbf{N}_{j,\ell}^{(\mathrm{h})}\mathrel{\mathop:}=\begin{pmatrix}
				2^j & \ell\, 2^{j/2}\\
				0 & 2^{j/2}
			\end{pmatrix},\qquad\qquad\;\mathbf{N}_{j,\ell}^{(\mathrm{v})}\mathrel{\mathop:}=\begin{pmatrix}
				2^{j/2} & 0\\
				\ell\, 2^{j/2} & 2^j
			\end{pmatrix}
		\end{equation}
		and the corresponding discrete angles
		\begin{equation*}
			\theta_{j,\ell}^{(\mathrm{h})}\mathrel{\mathop:}=\arctan\left(\ell\,2^{-j/2}\right),\qquad\qquad \theta_{j,\ell}^{(\mathrm{v})}\mathrel{\mathop:}=\mathrm{arccot}\left(\ell\,2^{-j/2}\right).
		\end{equation*}
		Note that these matrices occur in the construction of discrete shearlet systems, for example in \cite{labate:detection,labate:sparse}. Based on this, we introduce the notation
	\begin{equation}\label{eq:Psi_jl}
		\Psi^{(i)}_{j,\ell}(\cdot)\mathrel{\mathop:}=\Psi^{(i)}\left(\left(\mathbf{N}_{j,\ell}^{(i)}\right)^{-\mathrm{T}}\cdot\right)
	\end{equation} 
	and, since $\det\mathbf{N}_{j,\ell}^{(i)}=2^{3j/2}$, it follows that
		\begin{equation}\label{eq:support_Psi}
			\abs{\mathrm{supp}\,\Psi^{(i)}_{j,\ell}}=\abs{\mathrm{supp}\,\Psi^{(i)}}\det\mathbf{N}_{j,\ell}^{(i)} =\frac{8}{3}\,2^{3j/2}.
		\end{equation}
	
  \pgfplotsset{soldotb/.style={color=black,only marks,mark=*,mark size=0.8}} 
  \pgfplotsset{holdotb/.style={color=black,fill=white,only marks,mark=*,mark size=0.8}}

  \begin{figure}[t]
   \subfloat{
	  	{\includegraphics[width=.48\textwidth]{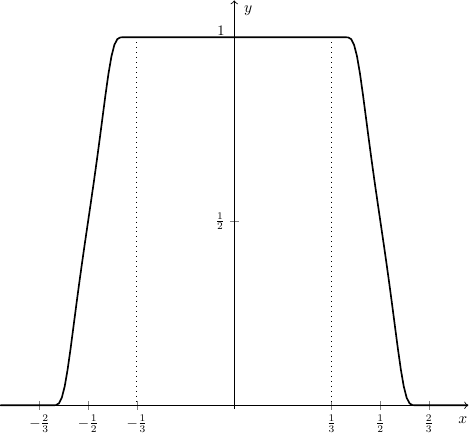}}
			}
	\subfloat{
		{\includegraphics[width=.48\textwidth]{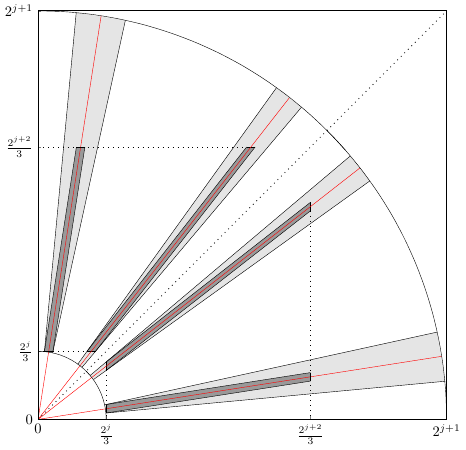}}
		    }
	\caption{Left: The window function $g_{0.025}\in \mathcal{W}^\infty$ (see \cref{eq:exp_window}). Right: Visualization of $\mathrm{supp}\,\Psi_{10,\ell}^{(i)}$ (dark area) and $W_{10,\ell}^{(i)}$ (light area) for $\ell=5,25$ and $i\in\lbrace \mathrm{h},\mathrm{v}\rbrace$. The red lines and the horizontal axis form the angles $\theta_{10,\ell}^{(i)}$.}\label{fig:support}
  \end{figure}
  
		In polar coordinates, we define the sets 
\begin{align*}
	&W_{j,\ell}^{(\mathrm{h})}\mathrel{\mathop:}=\left\lbrace(\rho,\theta)\in \mathbb{R}\times\left[-\frac{\pi}{2},\frac{\pi}{2}\right]:\frac{2^j}{3}<\abs{\rho}< 2^{j+1},\,\theta_{j,\ell-2}^{(\mathrm{h})}<\theta<\theta_{j,\ell+2}^{(\mathrm{h})}\right\rbrace,\\
	&W_{j,\ell}^{(\mathrm{v})}\mathrel{\mathop:}=\left\lbrace(\rho,\theta)\in \mathbb{R}\times\left[0,\pi\right]:\frac{2^j}{3}<\abs{\rho}< 2^{j+1},\,\theta_{j,\ell+2}^{(\mathrm{v})}<\theta<\theta_{j,\ell-2}^{(\mathrm{v})}\right\rbrace
	\end{align*}
	and based on ideas from \cite[Proposition 2.1]{labate:sparse} we show the following lemma, which is visualized on the right side of \cref{fig:support}.
  \begin{lemma}\label{lem:support_Psi}
  	For even $j\geq 10$, $\ell\in \mathbb{Z}$ with $\abs{\ell}\leq 2^{j/2}$ and $i\in\lbrace \mathrm{h},\mathrm{v}\rbrace$ we have
 	\begin{equation*}
 		\mathrm{supp}\,\Psi_{j,\ell}^{(i)}(\rho,\theta)\subset W_{j,\ell}^{(i)}.
 	\end{equation*}
  \end{lemma} 
  \begin{proof}
	We show only the case $i=\mathrm{h}$ since the other one is similar. In \cref{eq:Psi_jl} we defined
    	 \begin{equation*} \Psi_{j,\ell}^{(\mathrm{h})}(\boldsymbol{\xi})=\Psi^{(\mathrm{h})}\left(\left(\mathbf{N}_{j,\ell}^{(\mathrm{h})}\right)^{-\mathrm{T}}\boldsymbol{\xi}\right)=\widetilde{g}(2^{-j}\xi_1)\,g\left(2^{-j}\xi_1\left(2^{j/2}\,\frac{\xi_2}{\xi_1}-\ell\right)\right)
    	 \end{equation*}
 with the support property
 	 \begin{equation*}
 	 	\mathrm{supp}\,\widetilde{g}(2^{-j}\xi_1)=\left\lbrace \xi_1\in \mathbb{R}\,:\, \frac{2^j}{3}<\abs{\xi_1}< \frac{2^{j+2}}{3}\right\rbrace
 	 \end{equation*}
 	 and, assuming that $\xi_1\in\mathrm{supp}\,\widetilde{g}(2^{-j}\cdot)$, we have
 	 \begin{align*}
 	 	\mathrm{supp}\,g\left(2^{-j}\xi_1\left(2^{j/2}\,\frac{\xi_2}{\xi_1}-\ell\right)\right)&=\left\lbrace \xi_2\in \mathbb{R}\,:\,\abs{2^{-j}\xi_1\left(2^{j/2}\,\frac{\xi_2}{\xi_1}-\ell\right)}< \frac{2}{3} \right\rbrace\\
 		&=\left\lbrace\xi_2\in \mathbb{R}\,:\,\abs{\ell-2^{j/2}\,\frac{\xi_2}{\xi_1}}< \frac{2^{j+1}}{3\,\abs{\xi_1}} \right\rbrace\\
 		&\subset\left\lbrace\xi_2\in \mathbb{R}\,:\,\abs{\ell-2^{j/2}\,\frac{\xi_2}{\xi_1}}< 2 \right\rbrace.
 	 \end{align*}
 	 In the following, we introduce polar coordinates with the notation $\boldsymbol{\xi}\mathrel{\mathop:}=\rho\,\boldsymbol{\Theta}(\theta)$, where $\boldsymbol{\Theta}(\theta)\mathrel{\mathop:}=(\cos\theta, \sin\theta)^{\mathrm{T}}$. Recalling the discrete angles $\theta_{j,\ell}^{(\mathrm{h})}=\arctan\left(\ell\,2^{-j/2}\right)$ we have
 	 \begin{align*}
 	 	\mathrm{supp}\,g\left(2^{-j}\xi_1\left(2^{j/2}\,\frac{\xi_2}{\xi_1}-\ell\right)\right)&\subset\left\lbrace \theta\in\left[-\frac{\pi}{2},\frac{\pi}{2}\right]\,:\,\abs{\ell-2^{j/2}\,\tan\theta}< 2 \right\rbrace\\
 		&=\left\lbrace \theta\in\left[-\frac{\pi}{2},\frac{\pi}{2}\right]\,:\,\theta_{j,\ell-2}^{(\mathrm{h})}<\theta<\theta_{j,\ell+2}^{(\mathrm{h})} \right\rbrace.
 	 \end{align*}
	 Since $\rho^2=\xi_1^2\left(1+\tan^2\theta\right)$ and $\abs{\ell}\leq 2^{j/2}$ we can show
 	 \begin{align*}
 	 	\abs{\rho}\leq \frac{2^{j+2}}{3}\Bigl(1+2^{-j}(\abs{\ell}+2)^2\Bigr)^{1/2}\leq\frac{2^{j+2}}{3}\Bigl(2+2^{2-j/2}+2^{2-j}\Bigr)^{1/2}< 2^{j+1},
 	 \end{align*}
 	where the last inequality holds for $j\geq 10$. As a lower bound for the radius $\rho$ we obtain
 	 \begin{equation*}
 	 	\abs{\rho}\geq\frac{2^{j}}{3}\left( 1+2^{-j}(\abs{\ell}+2)^2 \right)^{1/2}>\frac{2^j}{3}.
 	 \end{equation*}
  \end{proof}
  
	The pattern of a regular matrix $\mathbf{M}\in \mathbb{Z}^{2\times 2}$ is defined by $\mathcal{P}(\mathbf{M})\mathrel{\mathop:}=\mathbf{M}^{-1}\mathbb{Z}^2\cap\left[-\frac{1}{2},\frac{1}{2}\right)^2$. As a consequence of \cite[Lemma 2.4]{langemann:multi_periodic_wavelets} the patterns of the matrices in \cref{eq:N_jl} are independent of the parameter $\ell$ and have the tensor product structure
	\begin{align*}
		&\mathcal{P}\left(\mathbf{N}_{j,\ell}^\mathrm{(h)}\right)=\Bigl\lbrace 2^{-j}\,z_1\,:\,z_1=-2^{j-1},\dots,2^{j-1}-1 \Bigr\rbrace\times\Bigl\lbrace 2^{-j/2}\,z_2\,:\,z_2=-2^{j/2-1},\dots,2^{j/2-1}-1\Bigr\rbrace,\\
		&\mathcal{P}\left(\mathbf{N}_{j,\ell}^\mathrm{(v)}\right)=\Bigl\lbrace 2^{-j/2}\,z_1\,:\,z_1=-2^{j/2-1},\dots,2^{j/2-1}-1 \Bigr\rbrace\times\Bigl\lbrace 2^{-j}\,z_2\,:\,z_2=-2^{j-1},\dots,2^{j-1}-1 \Bigr\rbrace.
	\end{align*}
	
For $i\in\lbrace \mathrm{h},\mathrm{v}\rbrace$ and $\Psi^{(i)}\in \mathcal{W}^q_2$ the translates of the de la Vall\'{e}e Poussin wavelet functions (see \cite{bergmann:dlVP}) on the pattern points $\mathbf{y}\in \mathcal{P}(\mathbf{N}_{j,\ell}^{(i)})$ are trigonometric polynomials given by
	\begin{equation*}
		\psi_{j,\ell,\mathbf{y}}^{(i)}(\mathbf{x})\mathrel{\mathop:}=\sum_{\mathbf{k}\in \mathbb{Z}^2}\Psi^{(i)}_{j,\ell}(\mathbf{k})\,\mathrm{e}^{\mathrm{i}\mathbf{k}^{\mathrm{T}}(\mathbf{x}-2\pi\widetilde{\mathbf{y}})},
	\end{equation*}
	where
	\begin{equation*}
		\widetilde{\mathbf{y}}\mathrel{\mathop:}=
		\begin{cases}
			\mathbf{y}-(2^{-j-1},0)^{\mathrm{T}}, &\text{for\; } \mathbf{y}\in \mathcal{P}\left(\mathbf{\mathbf{N}}^\mathrm{(h)}_{j,\ell}\right),\\
			\mathbf{y}-(0,2^{-j-1})^{\mathrm{T}}, &\text{for\; } \mathbf{y}\in \mathcal{P}\left(\mathbf{\mathbf{N}}^\mathrm{(v)}_{j,\ell}\right).
		\end{cases}
		\end{equation*}
		In the following we call the functions $\psi_{j,\ell,\mathbf{y}}^{(i)}$ trigonometric polynomial shearlets.
			
\section{Main results}\label{sec:main_results}
	Let $\rho(t):[0,2\pi)\rightarrow[0,\pi)$ fulfilling
	\begin{equation*}
		\sup\limits_{0\leq t<2\pi}\abs{\rho''(t)}\leq \kappa<\infty
	\end{equation*}
	and let $\boldsymbol{\gamma}:[0,2\pi)\rightarrow(-\pi,\pi)^2$ be a closed curve with
	\begin{equation*}
		\boldsymbol{\gamma}(t)\mathrel{\mathop:}=\rho(t)
		\begin{pmatrix}
			\cos{t}\\
			\sin{t}
		\end{pmatrix},\qquad t\in[0,2\pi),
	\end{equation*}
	 which is a parametrization of the boundary of a set $T\subset(-\pi,\pi)^2$. The space $\mathcal{C}^u(\kappa)$ is defined as the collection of all functions of the form
	\begin{equation}\label{eq:cartoon_like_function}
		f=f_0+f_1\chi_T,
	\end{equation}
	where $f_0,f_1\in C^u([-\pi,\pi]^2),\,u\geq 2$.
	
	Following the ideas from \cite{candes:curvelets,labate:sparse} let $\mathcal{Q}_j,\,j\in \mathbb{N}_0,$ be the set of dyadic squares $Q\subseteq[-\pi,\pi)^2$ of the form
\begin{equation}\label{eq:dyadic_squares}
	Q=\left[2\pi n_1\,2^{-j/2}-\pi,2\pi (n_1+1)\,2^{-j/2}-\pi\right)\times\left[2\pi n_2\,2^{-j/2}-\pi,2\pi (n_2+1)\,2^{-j/2}-\pi\right)
\end{equation} 
with $n_1,n_2=0,\dots ,2^{j/2}-1$. Let $Q\in\mathcal{Q}_j^1\subseteq\mathcal{Q}_j$ if $\partial T\cap Q\neq\emptyset$ and for the non-intersecting squares we define $\mathcal{Q}_j^0\mathrel{\mathop:}=\mathcal{Q}_j\setminus\mathcal{Q}_j^1$. We remark that $\abs{\mathcal{Q}_j^0}\leq C\,2^{j}$ and $\abs{\mathcal{Q}_j^1}\leq C_2\,2^{j/2} $ (see \cite{candes:curvelets,labate:sparse}).

  \begin{figure}[t]
		\subfloat{
		{\includegraphics[width=.48\textwidth]{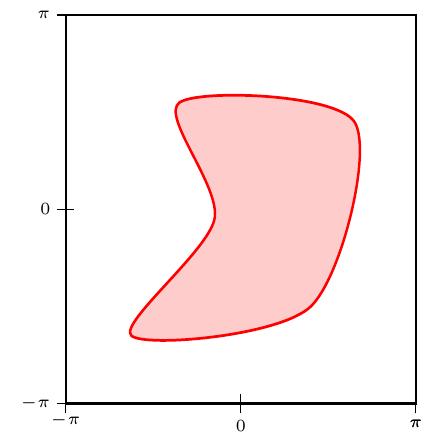}}
		}
		\subfloat{
		{\includegraphics[width=.48\textwidth]{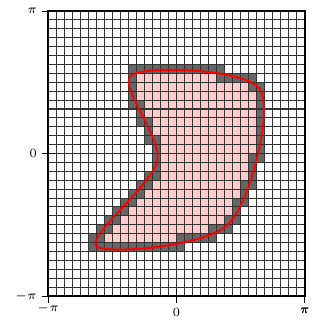}}
		}
	  	\caption{Left: Characteristic function of a set $T\subset(-\pi,\pi)^2$ with boundary $\partial T$. Right: Decomposition into dyadic squares for $j=10$, where  $Q\in \mathcal{Q}_j^0$ are colored white and $Q\in \mathcal{Q}_j^1$ along $\partial T$ are colored dark.}\label{fig:dyadic_squares}
	    \end{figure}
			  	
	For Lebesgue measurable sets $A\subseteq \mathbb{R}^2$ and functions $f:A\rightarrow \mathbb{R}$ define
 \begin{equation*}
 	\norm{f}_{A,p}\mathrel{\mathop:}=\left( \int_{A}\abs{f(\mathbf{x})}^p\,\mathrm{d}\mathbf{x} \right)^{1/p},\qquad 1\leq p<\infty,
 \end{equation*}
and let $L_p(A)$ denote the collection of functions satisfying $\norm{f}_{A,p}<\infty$. In particular, two-dimensional $2\pi$-periodic functions $f:\mathbb{T}^2\rightarrow \mathbb{R}$ are defined on the torus $\mathbb{T}^2\mathrel{\mathop:}=\mathbb{R}^2\setminus2\pi\,\mathbb{Z}^2$. Recall that the usual inner product of the Hilbert space $L_2(\mathbb{T}^2)$ is given by
\begin{equation*}
	\langle f,g \rangle_2 \mathrel{\mathop:}=(2\pi)^{-2}\int_{\mathbb{T}^2}f(\mathbf{x})\overline{g(\mathbf{x})}\,\mathrm{d}\mathbf{x},\qquad\qquad f,g\in L_2(\mathbb{T}^2),
\end{equation*}
and for $f\in L_1(\mathbb{R}^2)$ we call
	\begin{equation}\label{eq:periodization}
		f^{2\pi}\mathrel{\mathop:}=\sum_{\mathbf{n}\in \mathbb{Z}^2}f(\cdot+2\pi\mathbf{n}) 
	\end{equation}
	the $2\pi$-periodization of $f$.\\
	
		The main results of this paper are stated in the following two theorems.
	\begin{theorem}\label{thm:upper_estimate}
		Let $f\in \mathcal{C}^2(\kappa)$ and $\Psi^{(i)}\in \mathcal{W}^{2q}_2,\,i\in\lbrace \mathrm{h},\mathrm{v}\rbrace$ for $q\geq 2$. Moreover for $Q\in \mathcal{Q}_j^1$ let $\mathbf{x}_0\mathrel{\mathop:}=\mathbf{x}_0(Q)\in\partial T\cap Q$ and $\gamma\mathrel{\mathop:}=\gamma(\mathbf{x}_0)$ such that $(\cos{\gamma},\sin{\gamma})^{\mathrm{T}}$ is the normal direction  of the boundary curve $\partial T$ in $\mathbf{x}_0$.  Then we have
		\begin{equation*}
			\abs{\left\langle f^{2\pi},\psi_{j,\ell,\mathbf{y}}^{(i)}\right\rangle_2}\leq C(q)\,\sum_{Q\in \mathcal{Q}_j^1}\left(1+2^j\abs{\mathbf{x}_0-2\pi\widetilde{\mathbf{y}}}_2^2\right)^{-q}\left( 1+2^{j/2}\abs{\sin(\theta_{j,\ell}^{(i)}-\gamma)}\right)^{-5/2}.
		\end{equation*}
		\end{theorem}
		If $\mathbf{y}\in\mathcal{P}\left(\mathbf{N}_{j,\ell}^{(i)}\right)$ is sufficiently far away from the boundary curve, \cref{thm:upper_estimate} implies 
		\begin{equation*}
			\abs{\left\langle f^{2\pi},\psi_{j,\ell,\mathbf{y}}^{(i)}\right\rangle_2}\leq C(q) 2^{-j(q-1/2)}.
		\end{equation*}
		
		For the special case $f_0=0$ and $f_1=1$ in \cref{eq:cartoon_like_function} we define $\mathcal{T}=\chi_T$ and denote by $\mathcal{T}^{2\pi}$ the $2\pi$-periodization of $\mathcal{T}$. 
		\begin{theorem}\label{thm:lower_estimate}
		Let $\Psi^{(i)}\in \mathcal{W}^{2q}_2$ for sufficiently large $q\in \mathbb{N}$ and $\mathbf{y}\in\mathcal{P}\left(\mathbf{N}_{j,\ell}^{(i)}\right)$ for large $j$. If there exists $\mathbf{x}_0\in\partial T$ with the normal direction $(\cos{\gamma},\sin{\gamma})^{\mathrm{T}}$ and curvature $A_0$ in that point, fulfilling $\abs{\mathbf{x}_0-2\pi\widetilde{\mathbf{y}}}_2\leq C\,2^{-j/2}$ and $\theta_{j,\ell}^{(i)}\leq\gamma\leq\theta_{j,\ell+1}^{(i)}$ for $i\in\lbrace \mathrm{h},\mathrm{v}\rbrace$, then there is a constant $C(q,A_0)>0$ such that
		\begin{equation*}
			\abs{\left\langle \mathcal{T}^{2\pi},\psi_{j,\ell,\mathbf{y}}^{(i)}\right\rangle_2}\geq C(q,A_0).
		\end{equation*}
	\end{theorem}
	
\section{Numerical examples}
	\label{sec:numerical_examples}

	In this section we give numerical examples to underline the main results of this paper by computing the shearlet coefficients of a characteristic function of a rotated ellipse. In order to do that, we need to compute the Fourier transform of the characteristic function of a disc, given by
	\begin{equation*}
		D(\mathbf{x})\mathrel{\mathop:}=\begin{cases}
			1 &\text{for } \abs{\mathbf{x}}_2\leq 1,\\
			0 &\text{else.}
		\end{cases}
	\end{equation*}
	We transform $\boldsymbol{\xi}=\rho\,\boldsymbol{\Theta}(\theta)$ and $\mathbf{x}=r\,\boldsymbol{\Theta}(\phi)$ into polar coordinates and use $\boldsymbol{\xi}^{\mathrm{T}}\mathbf{x}=r\rho\cos\left(\theta-\phi\right)$ to obtain
	\begin{align*}
		\mathcal{F}[D](\boldsymbol{\xi})&=\frac{1}{(2\pi)^2}\int\limits_{\mathbb{R}^2}D(\mathbf{x})\,\mathrm{e}^{-\mathrm{i}\boldsymbol{\xi}^{\mathrm{T}}\mathbf{x}}\,\mathrm{d}\mathbf{x}\\
		&=\frac{1}{(2\pi)^2}\int\limits_{0}^{1}\int\limits_{0}^{2\pi}\,\mathrm{e}^{-\mathrm{i}r\rho\cos\left(\theta-\phi\right)}\,r\,\mathrm{d}\phi\,\mathrm{d}r=\frac{1}{2\pi}\int\limits_{0}^{1} r\,J_0(r\rho)\,\mathrm{d}r,
	\end{align*}
	where $J_0$ is the Bessel function of the first kind and zero order. The integral identity
	\begin{equation*}
		\int\limits_{0}^{u}t\,J_0(t)\,\mathrm{d}t=u\,J_1(u)
	\end{equation*}
	together with the change of variable $\lambda=r\rho$ leads to
	\begin{equation*}
		\mathcal{F}[D](\boldsymbol{\xi})=\frac{1}{2\pi}\int\limits_{0}^{1} r\,J_0(r\rho)\,\mathrm{d}r=\frac{1}{2\pi\,\rho^2}\int\limits_{0}^{\rho} \lambda\,J_0(\lambda)\,\mathrm{d}\lambda=\frac{J_1(\rho)}{2\pi\rho}=\frac{J_1\left(\abs{\boldsymbol{\xi}}_2\right)}{2\pi\abs{\boldsymbol{\xi}}_2}.
	\end{equation*}
	
	\begin{figure}[t]
	\subfloat{
	{\includegraphics[width=.51\textwidth]{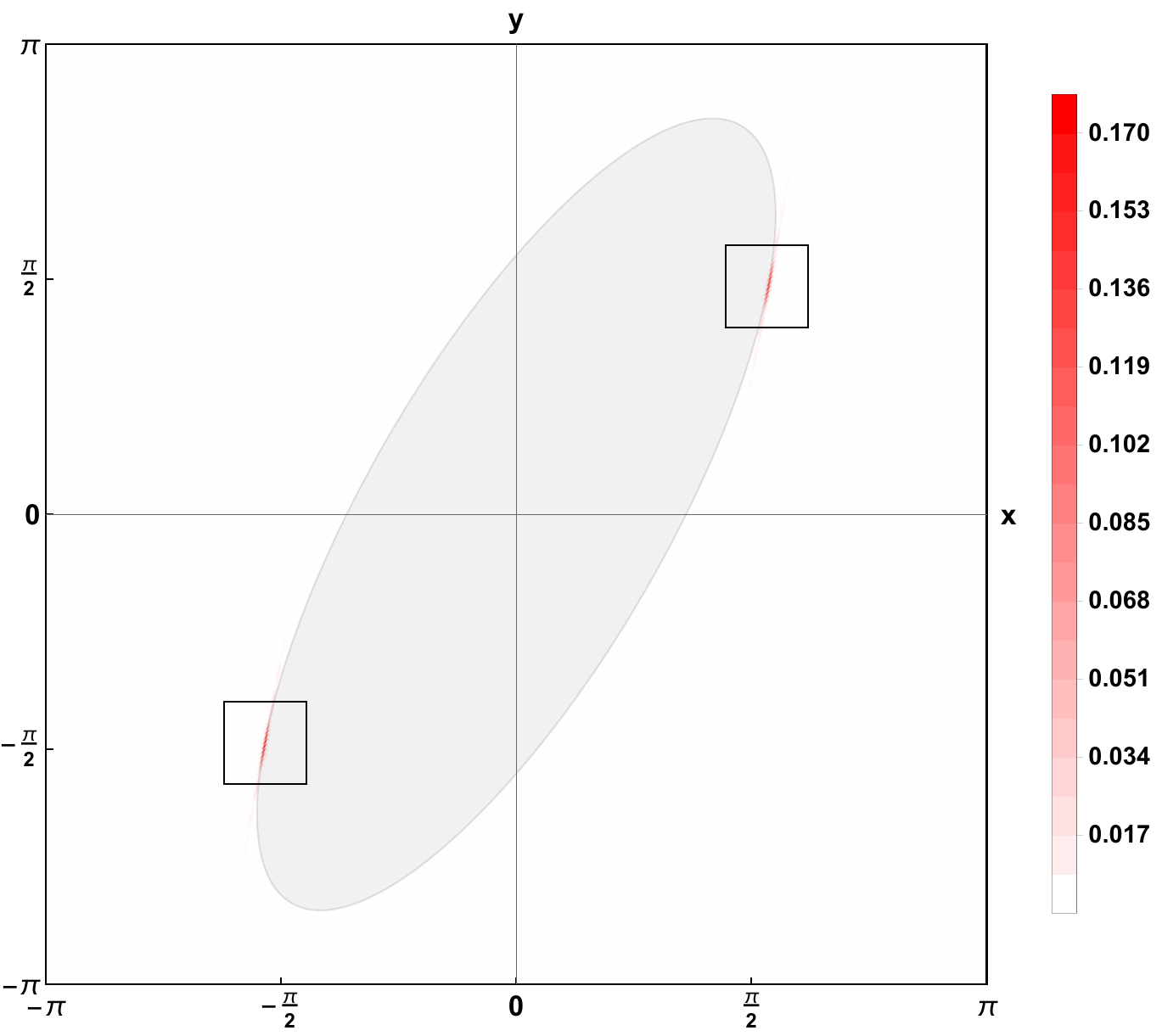}}
	}
	\subfloat{
	{\includegraphics[width=.44\textwidth]{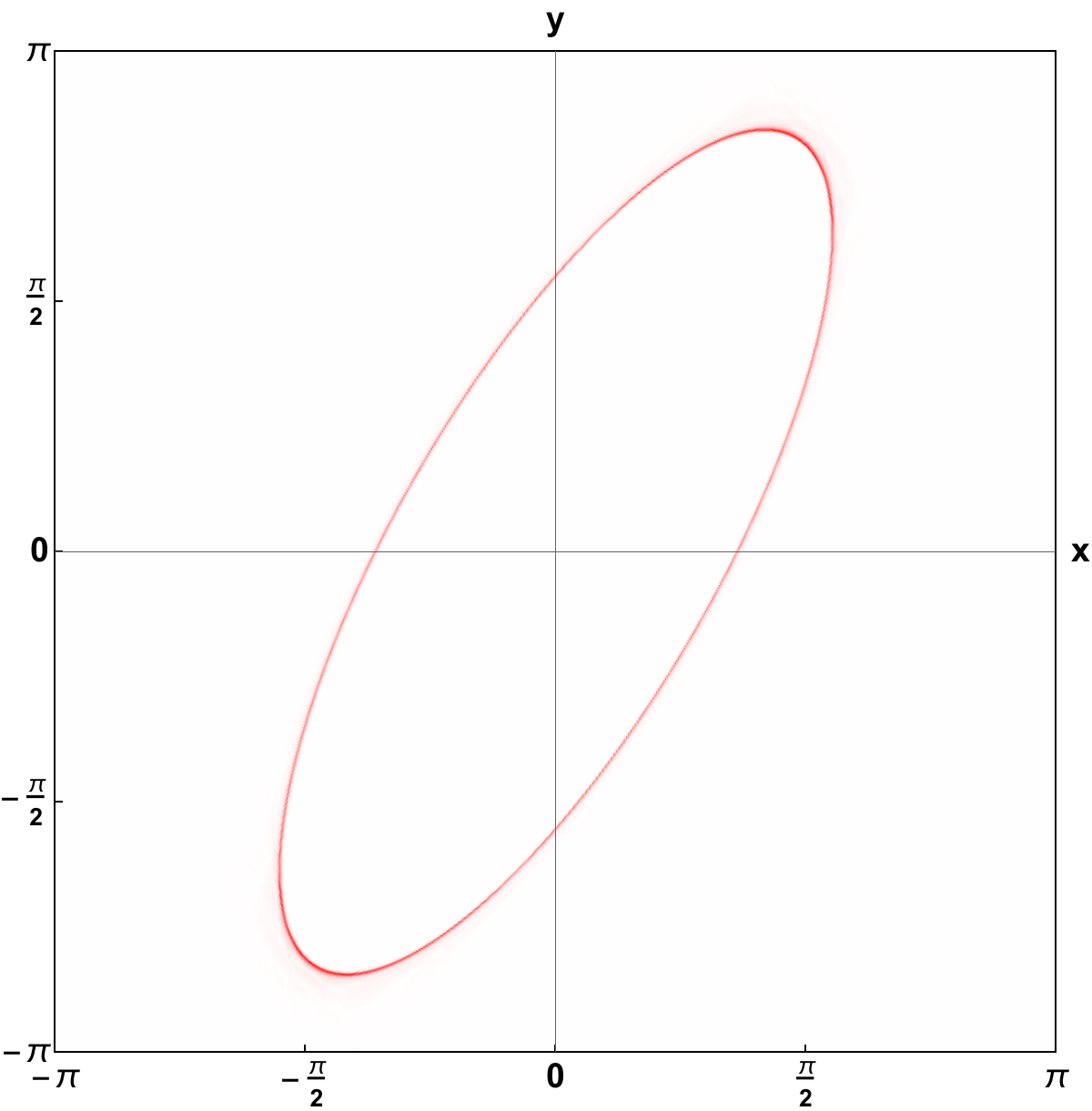}}
	}
	\caption{Left: Characteristic function $D_{1,3,\frac{\pi}{6}}^{2\pi}$ (gray) and magnitude of the inner product $\abs{\left\langle D_{1,3,\frac{\pi}{6}}^{2\pi},\psi_{10,-3,\mathbf{y}}^{(\mathrm{h})}\right\rangle_2}$ (red) for every $\mathbf{y}\in \mathcal{P}(\mathbf{M}_{10})$. Right: Magnitude of $\sum\limits_{\ell=-2^{j/2}+1}^{2^{j/2}-1}\abs{\left\langle D_{1,3,\frac{\pi}{6}}^{2\pi},\psi_{10,\ell,\mathbf{y}}^{(i)}\right\rangle_2}$ for every $\mathbf{y}\in \mathcal{P}(\mathbf{M}_{10})$ and $i\in \left\lbrace \mathrm{h},\mathrm{v}\right\rbrace$.}
	\label{fig:pic}
	\end{figure}

	For $a,b>0$ we define $D_{a,b}(\mathbf{x})\mathrel{\mathop:}=D(a^{-1} x_1,b^{-1}x_2)$ and convert the circle into a characteristic function of an ellipse with major semi-axis of length $a$ and minor semi-axis of length $b$. By the scaling property of the Fourier transform we have
	\begin{equation*}
		\mathcal{F}\left[D_{a,b}\right](\boldsymbol{\xi})=\frac{ab\,J_1\bigl(\abs{(a\xi_1,b\xi_2)}_2\bigr)}{2\pi\abs{(a\xi_1,b\xi_2)}_2}.
	\end{equation*}

	 If we further rotate the function $D_{a,b}$ by an angle $\gamma\in[0,2\pi)$ we obtain a rotated ellipse $D_{a,b,\gamma}(\mathbf{x})\mathrel{\mathop:}=D_{a,b}\left(\mathbf{R}_\gamma\,\mathbf{x}\right)$ and its Fourier transform is given by $\mathcal{F}\left[D_{a,b,\gamma}\right](\boldsymbol{\xi})=\mathcal{F}\left[D_{a,b}\right](\mathbf{R}^{\mathrm{T}}\boldsymbol{\xi})$.\\
 
	  In order to calculate the shearlet coefficients of a rotated ellipse we consider the $2\pi$-periodized function $D_{a,b,\gamma}^{2\pi}(\mathbf{x})$. We use \cref{eq:poisson_fourier} to see that the Fourier coefficients of this function are given by 
	  \begin{equation*}
	  	c_\mathbf{k}(D_{a,b,\gamma}^{2\pi})=\mathcal{F}\left[D_{a,b,\gamma}\right](\mathbf{k}),\;\mathbf{k}\in \mathbb{Z}^2,
	  \end{equation*}
	  and Parseval's identity finally gives
	  \begin{equation}\label{eq:coeffsTest}
	  	\left\langle D_{a,b,\gamma}^{2\pi},\psi^{(i)}_{j,\ell,\mathbf{y}}\right\rangle_2=\sum_{\mathbf{k}\in \mathbb{Z}^2}\mathcal{F}\left[D_{a,b,\gamma}\right](\mathbf{k})\,\Psi^{(i)}_{j,\ell}(\mathbf{k})\,\mathrm{e}^{2\pi\mathrm{i}\mathbf{k}^{\mathrm{T}}\widetilde{\mathbf{y}}}.
	  \end{equation}

In our numerical example we calculate the inner product \cref{eq:coeffsTest} with \textit{Mathematica 12}. We fix the characteristic function of the rotated ellipse $D_{1,3,\frac{\pi}{6}}^{2\pi}$, which is depicted on the left side in \cref{fig:pic} (gray area). For the one-dimensional window function we use the smooth function $g_{0.025}\in \mathcal{W}^\infty$ constructed in \cref{eq:exp_window}. We fix the scale $j=10$ and for a better visualization we consider the matrix $\mathbf{M}_{10}=2^{10}\,\mathbf{I}_2$, where $\mathbf{I}_2$ is the two-dimensional identity matrix and compute the shearlet coefficients in \cref{eq:coeffsTest} on the pattern $\mathcal{P}(\mathbf{M}_{10})$, which corresponds to a two-dimensional equidistant grid of $1024\times1024$ points. Thus the images in this example are of size $1024\times1024$, where every pixel corresponds to the magnitude of the inner product with a translate of the trigonometric polynomial shearlet $\psi_{j,\ell,\mathbf{y}}^{(i)},\;\mathbf{y}\in\mathcal{P}(\mathbf{M}_{10})$. 
\begin{figure}[t]
	  \subfloat{
	  {\includegraphics[width=.49\textwidth]{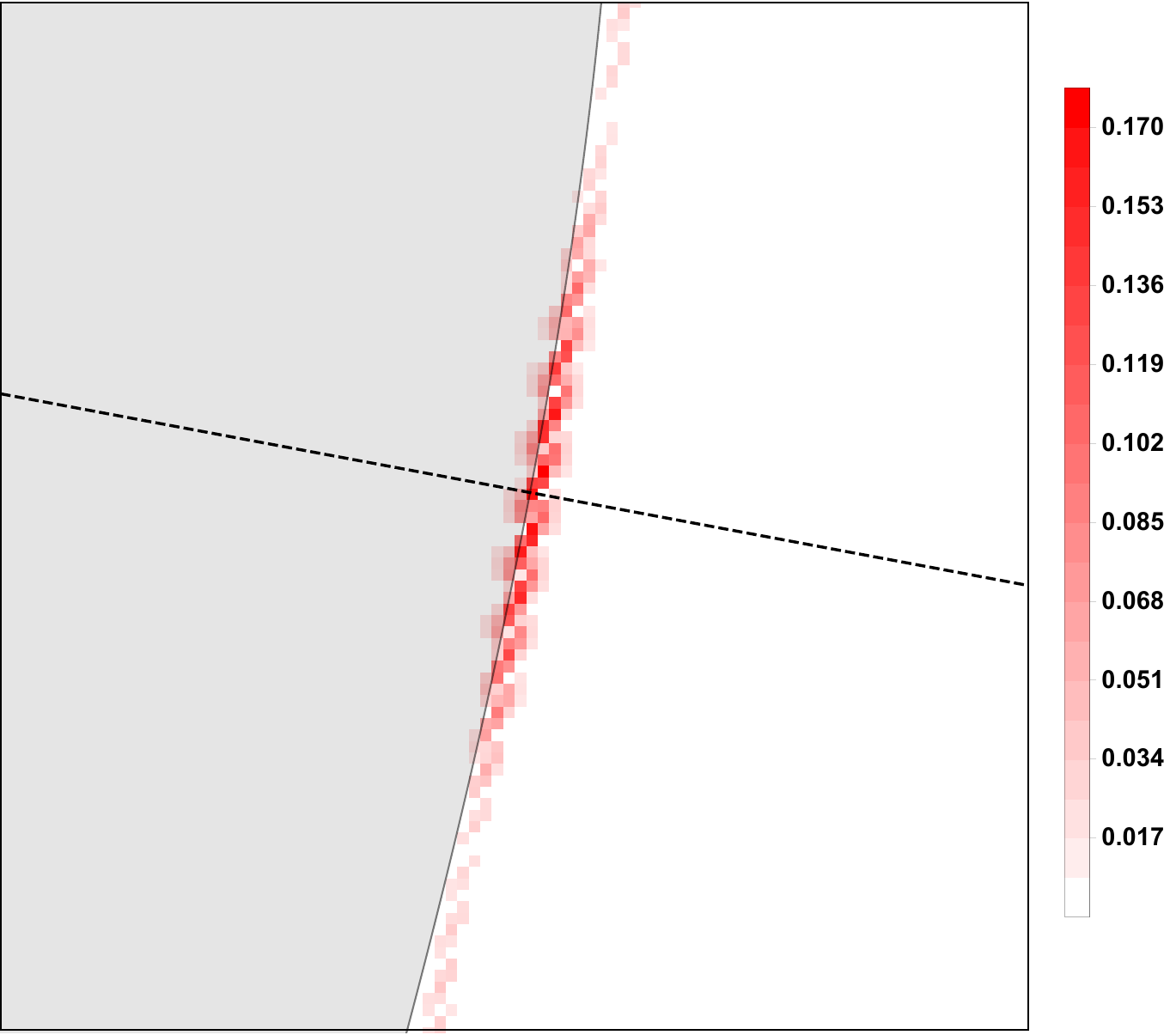}}
	  }
	  \subfloat{
	  {\includegraphics[width=.46\textwidth]{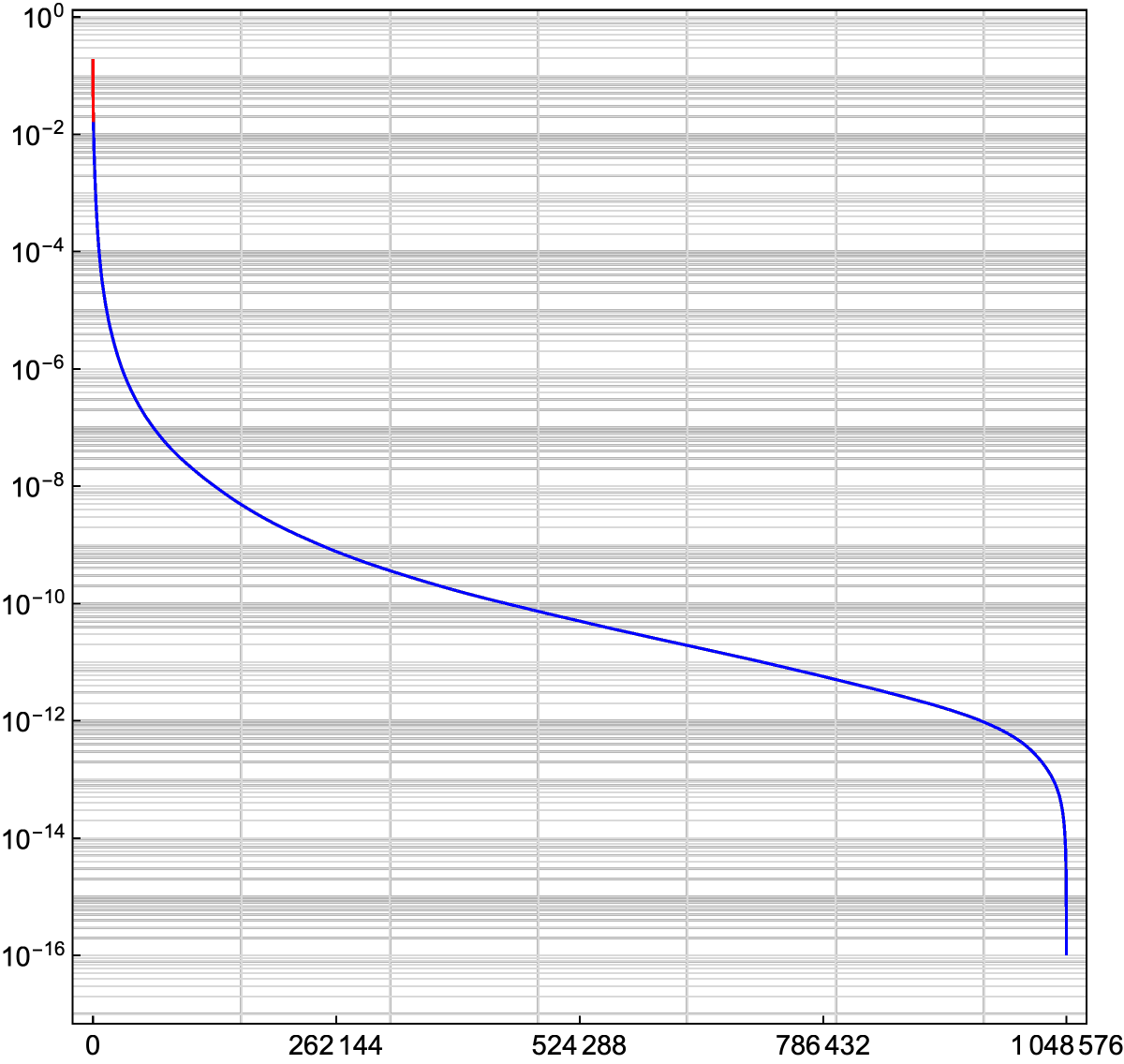}}
	  }
	  \caption{Left: Zoom into the upper right window from the left image of \cref{fig:pic}. Left: Logarithmic plot of the wavelet coefficients $\abs{\left\langle D_{1,3,\frac{\pi}{6}}^{2\pi},\psi_{10,-3,\mathbf{y}}^{(\mathrm{h})}\right\rangle_2}$ for every $\mathbf{y}\in \mathcal{P}(\mathbf{M}_{10})$ in descending order.}
	  \label{fig:pic2}
	  \end{figure}
	On the left side of \cref{fig:pic} we set the directional parameter to $\ell=-3$ and observe that the magnitude of the coefficients
	\begin{equation*}
		\abs{\left\langle D_{1,3,\frac{\pi}{6}}^{2\pi},\psi_{10,-3,\mathbf{y}}^{(\mathrm{h})}\right\rangle_2}
	\end{equation*}
	are very close to zero except for the pattern points $\mathbf{y}\in\mathcal{P}(\mathbf{M}_{10})$, for which the function $\psi_{10,-3,\mathbf{y}}^{(\mathrm{h})}$ is close to points $\mathbf{x}\in\partial D_{1,3,\frac{\pi}{6}}^{2\pi}$ on the boundary with normal direction almost parallel to the direction induced by the angle $\theta_{10,-3}^{(\mathrm{h})}$. 

To make this more clear, the left image of \cref{fig:pic2} zooms into the upper right black square. The dotted line is parallel to the line, which forms the angle $\theta_{10,-3}^{(\mathrm{h})}$ with the horizontal axis. One can observe that the only significant shearlet coefficients are close to the boundary and nearly orthogonal to this line. In other words, only if the trigonometric shearlet $\psi_{10,-3,\mathbf{y}}^{(\mathrm{h})}$ is almost aligned with the boundary $\partial D_{1,3,\frac{\pi}{6}}^{2\pi}$, the inner product \cref{eq:coeffsTest} yields large values. The right graph in \cref{fig:pic2} is a logarithmic plot of the shearlet coefficients $\abs{\left\langle D_{1,3,\frac{\pi}{6}}^{2\pi},\psi_{10,-3,\mathbf{y}}^{(\mathrm{h})}\right\rangle_2}$ for every $\mathbf{y}\in \mathcal{P}(\mathbf{M}_{10})$ in descending order. The red part of the line corresponds to the coefficients visible in the left picture of \cref{fig:pic}. The right image illustrates the capability of the trigonometric polynomial shearlets to detect step discontinuities along the boundary of characteristic functions. For $i\in \left\lbrace \mathrm{h},\mathrm{v}\right\rbrace$ and $\ell=-2^{j/2}+1,\hdots,2^{j/2}-1$ we compute all the pictures of the shearlet coefficients similar to the left image in \cref{fig:pic} and add them component-wise to get the final result. Thus every pixel of the image is given by the sum
	\begin{equation*}
		\sum\limits_{\ell=-2^{j/2}+1}^{2^{j/2}-1}\abs{\left\langle D_{1,3,\frac{\pi}{6}}^{2\pi},\psi_{10,\ell,\mathbf{y}}^{(i)}\right\rangle_2},\qquad i\in \left\lbrace \mathrm{h},\mathrm{v}\right\rbrace,\,\,\mathbf{y}\in \mathcal{P}(\mathbf{M}_{10}),
	\end{equation*}
	and one can clearly see the only significant coefficients for all the directions are exact on the boundary of $D_{1,3,\frac{\pi}{6}}^{2\pi}$.
	\begin{figure}[t]
		\centering
		\subfloat{
	\includegraphics[width=.33\textwidth]{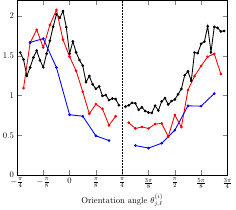}}
	  \subfloat{
	  \includegraphics[width=.3\textwidth]{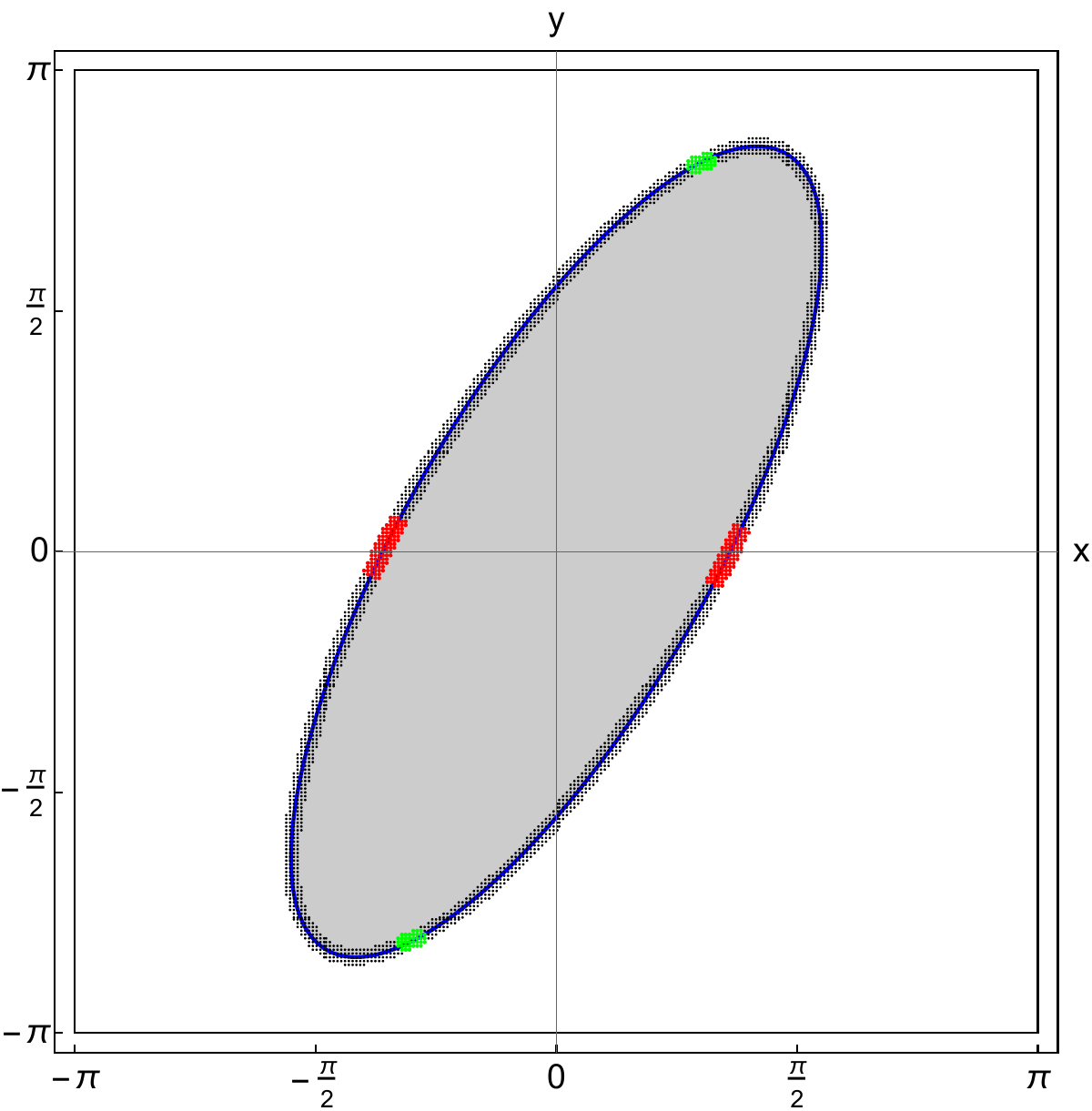}}
		\subfloat{ 
			\includegraphics[width=.33\textwidth]{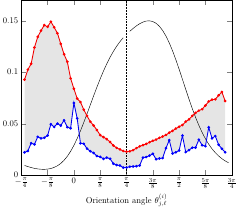}}
	\caption{Left: $U^{(i)}_{j,\ell}$ for $j=4$ (blue line), $j=6$ (red line) and $j=8$ (black line). Middle: $L^{(i)}_{j,\ell}$ for $j=8$ and $\ell\in \mathbb{Z}$ with $\abs{\ell}<2^{j/2}$ (dark points), $L^{(\mathrm{h})}_{8,-8}$ (red points) and $L^{(\mathrm{v})}_{8,-8}$ (green points). Right: $L^{(i),\mathrm{max}}_{j,\ell}$ (red line), $L^{(i),\mathrm{min}}_{j,\ell}$ (blue line) and the curvature values $\frac{1}{20}\kappa(x)$ (black line).}
	\label{abb:upper_lower_bound} 
	\end{figure}
Besides the visual representations of the detection of step discontinuities with trigonometric polynomial shearlets, we want to illustrate the upper and lower estimates given in the two main theorems. In order to do so for the upper bound, we compute the quantity
	\begin{equation*}
		U^{(i)}_{j,\ell}\mathrel{\mathop:}=\max\limits_{\mathbf{y}\in \mathcal{P}(\mathbf{M}_j)}\frac{\abs{\left\langle D_{1,3,\frac{\pi}{6}}^{2\pi},\psi_{j,\ell,\mathbf{y}}^{(i)}\right\rangle_2}}{\sum\limits_{Q\in \mathcal{Q}_j^1}\left(1+2^j\abs{\mathbf{x}_0-2\pi\widetilde{\mathbf{y}}}_2^2\right)^{-q}\left( 1+2^{j/2}\abs{\sin(\theta_{j,\ell}^{(i)}-\gamma)}\right)^{-5/2}}.
	\end{equation*}
	In the left graph of \cref{abb:upper_lower_bound} the values $U^{(i)}_{j,\ell}$ are plotted for different orientation angles $\theta_{j,\ell}^{(i)}$. One can see that the quotient $U^{(i)}_{j,\ell}$ is bounded from above by a moderate constant for every $j$ and $\ell$ which confirms that the estimate in \cref{thm:upper_estimate} provides a valid upper bound.
	
	For the lower bound, we collect all pattern points $\mathbf{y}\in \mathcal{P}(\mathbf{M}_j)$ for which there exists $\mathbf{x}_0\in\partial T$ with the normal direction $(\cos{\gamma},\sin{\gamma})^{\mathrm{T}}$ fulfilling $\abs{\mathbf{x}_0-2\pi\widetilde{\mathbf{y}}}_2\leq C\,2^{-j/2}$ and $\theta_{j,\ell}^{(i)}\leq\gamma\leq\theta_{j,\ell+1}^{(i)}$ and call this set $L^{(i)}_{j,\ell}$. As an example, the set $L^{(i)}_{j,\ell}$ is shown in the middle of \cref{abb:upper_lower_bound} for $j=8$ and all $\ell\in \mathbb{Z}$ with $\abs{\ell}<2^{j/2}$ together with $L^{(\mathrm{h})}_{8,-8}$ and $L^{(\mathrm{v})}_{8,-8}$. The latter two sets include pattern points which are close to boundary points $\mathbf{x}_0\in\partial T$ with normal direction $(\cos{\gamma},\sin{\gamma})^{\mathrm{T}}$ fulfilling $\theta_{8,-8}^{(i)}\leq\gamma\leq\theta_{8,-7}^{(i)}$ for $i\in \left\lbrace \mathrm{h},\mathrm{v} \right\rbrace$. By \cref{thm:lower_estimate}, we expect the values of the shearlet coefficients corresponding to the shearlets $\psi^{(i)}_{8,-8,\mathbf{y}}$ for $\mathbf{y}\in L^{(i)}_{8,-8}$ to be bounded from below by a constant. Therefore, we compute the values 
	\begin{equation*}
		L^{(i),\mathrm{max}}_{j,\ell}\mathrel{\mathop:}=\max\limits_{\mathbf{y}\in L^{(i)}_{j,\ell}}\abs{\left\langle D_{1,3,\frac{\pi}{6}}^{2\pi},\psi_{j,\ell,\mathbf{y}}^{(i)}\right\rangle_2},\quad L^{(i),\mathrm{min}}_{j,\ell}\mathrel{\mathop:}=\min\limits_{\mathbf{y}\in L^{(i)}_{j,\ell}}\abs{\left\langle D_{1,3,\frac{\pi}{6}}^{2\pi},\psi_{j,\ell,\mathbf{y}}^{(i)}\right\rangle_2} 
	\end{equation*} 
	and show them in the right graph of \cref{abb:upper_lower_bound} as functions of the orientation angles $\theta_{j,\ell}^{(i)}$. One can clearly see that the minimal values $L^{(i),\mathrm{min}}_{j,\ell}$ are bounded from below, which confirms the result of \cref{thm:lower_estimate}. 
	In our numerical example a parametrization of the boundary $\partial D_{1,3,\frac{\pi}{6}}$ is given by $\boldsymbol{\gamma}(x)=\frac{1}{2}\bigl(\sqrt{3}\cos x+3\sin x,3\sqrt{3}\sin x-\cos x\bigr)^{\mathrm{T}}$ and the curvature in each point is given by $\kappa(x)=3\left( 5+4\cos(2x) \right)^{-3/2}$.
	On the right side of \cref{abb:upper_lower_bound} the values of the curvature $\kappa(x)$ are shown in the points $\mathbf{x}_0\in\partial T$, where the normal direction of $\mathbf{x}_0$ is orthogonal to $\theta^{(i)}_{j,\ell}$ for $\abs{\ell}< 2^{j/2}$. As anticipated in \cref{thm:lower_estimate}, one can see that the magnitude of the coefficients $L^{(i),\mathrm{max}}_{j,\ell},L^{(i),\mathrm{min}}_{j,\ell}$ varies as the curvature of the ellipse changes. If the curvature is small, which corresponds to the 'stretched' part of the boundary, the values become larger. Intuitively, this makes sense since in that case a large part of the boundary is aligned with the corresponding shearlet.

\section{Auxiliary results}\label{sec:auxiliary_results}

For two-dimensional vector norms we use the notation
			\begin{equation*}
				\abs{\mathbf{x}}_p\mathrel{\mathop:}=
				\begin{cases}
					\left( \abs{x_1}^p+\abs{x_2}^p \right)^{1/p}, &\text{if } 1\leq p<\infty,\\
					\max\left\lbrace\abs{x_1},\abs{x_2}\right\rbrace, &\text{if } p=\infty
				\end{cases}
			\end{equation*}
and for binary relations and exponentials of vectors we write $\mathbf{x}\leq \mathbf{y}$ if $x_1\leq y_1$ and $x_2\leq y_2$, $\mathbf{x}^\mathbf{y}\mathrel{\mathop:}=x_1^{y_1}\,x_2^{y_2}$ and $\mathbf{x}^\beta\mathrel{\mathop:}=\mathbf{x}^{\beta \mathbf{1}}=x_1^\beta\,x_2^\beta$ for $\beta\in \mathbb{R}$. Moreover for $\mathbf{k},\mathbf{n}\in \mathbb{N}_0^2$ with $\mathbf{k}\leq \mathbf{n}$ and $n\in \mathbb{N}_0$ with $\mathbf{k}\leq n\mathbf{1}$ we define $\mathbf{k}!\mathrel{\mathop:}=k_1!\,k_2!$ and 
	\begin{equation*}
		\binom{\mathbf{n}}{\mathbf{k}}\mathrel{\mathop:}=\frac{\mathbf{n}!}{\mathbf{k}!(\mathbf{n}-\mathbf{k})!}=\binom{n_1}{k_1}\,\binom{n_2}{k_2},\quad\binom{n}{\mathbf{k}}\mathrel{\mathop:}=\frac{n!}{\mathbf{k}!(n-\abs{\mathbf{k}}_1)!}.
	\end{equation*}
The Fourier coefficients of a function $f\in L_1(\mathbb{T}^2)$ are given by
\begin{equation*}
	c_{\mathbf{k}}(f)\mathrel{\mathop:}=(2\pi)^{-2}\int_{\mathbb{T}^2}f(\mathbf{x})\,\mathrm{e}^{-\mathrm{i}\mathbf{k}^{\mathrm{T}}\mathbf{x}}\,\mathrm{d}\mathbf{x},\qquad\mathbf{k}\in \mathbb{Z}^2.
\end{equation*}
The Fourier transform of $f\in L_1(\mathbb{R}^2)$ is defined as
	\begin{equation*}
		\mathcal{F}[f](\mathbf{x})\mathrel{\mathop:}=\mathcal{F}f(\mathbf{x})\mathrel{\mathop:}=(2\pi)^{-2}\int_{\mathbb{R}^2}f(\boldsymbol{\xi})\,\mathrm{e}^{-\mathrm{i}\boldsymbol{\xi}^{\mathrm{T}}\mathbf{x}}\,\mathrm{d}\boldsymbol{\xi},\qquad\mathbf{x}\in \mathbb{R}^2,
	\end{equation*}
	and we have the operator
	\begin{equation*}
		\mathcal{F}^{-1}[f](\mathbf{x})\mathrel{\mathop:}=\mathcal{F}^{-1}f(\mathbf{x})\mathrel{\mathop:}=\int_{\mathbb{R}^2}f(\boldsymbol{\xi})\,\mathrm{e}^{\mathrm{i}\boldsymbol{\xi}^{\mathrm{T}}\mathbf{x}}\,\mathrm{d}\boldsymbol{\xi},\qquad\mathbf{x}\in \mathbb{R}^2. 
	\end{equation*}
	
	For $f\in L_1(\mathbb{R}^2)$ and $\mathcal{F}f\in L_1(\mathbb{R}^2)$ the inversion formula $f(\mathbf{x})=\mathcal{F}\mathcal{F}^{-1}f(\mathbf{x})=\mathcal{F}^{-1}\mathcal{F}f(\mathbf{x})$ holds for all $\mathbf{x}\in \mathbb{R}^2$. We recall some basic results about the Fourier transform and its connection to Fourier series via the Poisson summation formula. Let $q\in \mathbb{N}_0$ and $\mathbf{r}\in \mathbb{N}_0^2$ with $\abs{\mathbf{r}}_1\leq q$. If $f\in L_1(\mathbb{R}^2)$ and $(\mathrm{i}\,\mathbf{x})^q\,f\in L_1(\mathbb{R}^2)$, then $\mathcal{F}f\in C^q(\mathbb{R}^2)$ and
		\begin{equation}\label{eq:properties_fourier2}
			\partial^\mathbf{r}\mathcal{F}f(\boldsymbol{\xi})=\mathcal{F}\left[(\mathrm{i}\,\mathbf{x})^\mathbf{r}\,f(\mathbf{x})\right](\boldsymbol{\xi}).
		\end{equation}
		Moreover for $f\in C^q(\mathbb{R}^2)$ and $\partial^\mathbf{r}f\in L_1(\mathbb{R}^2)$ we have
		\begin{equation}\label{eq:properties_fourier3}
			\mathcal{F}\left[ \partial^\mathbf{r}f \right](\boldsymbol{\xi})=(\mathrm{i}\,\boldsymbol{\xi})^\mathbf{r}\,\mathcal{F}f(\boldsymbol{\xi}).
		\end{equation}
	It is well known that there are constants $C_1(q,f),C_2(q,f)>0$ such that for $f\in C_0^q(\mathbb{R}^2)$ with $q\in \mathbb{N}_0$ and all $\mathbf{x}\in \mathbb{R}^2$ we have
		\begin{equation}\label{eq:decay_fourier_transform}
			\lvert\mathcal{F}f(\mathbf{x})\rvert\leq\frac{C_1(q,f)}{\left(1+\lvert\mathbf{x}\rvert_2\right)^q},\qquad\qquad\lvert\mathcal{F}^{-1}f(\mathbf{x})\rvert\leq\frac{C_2(q,f)}{\left(1+\lvert\mathbf{x}\rvert_2\right)^q}.
		\end{equation}
		
 The sum in \cref{eq:periodization} converges for almost every $\mathbf{x}\in \mathbb{T}^2$ and $f^{2\pi}\in L_1(\mathbb{T}^2)$. For the Fourier coefficients we have
	   		\begin{equation}\label{eq:poisson_fourier}
	   			c_{\mathbf{k}}(f^{2\pi})=\mathcal{F}f(\mathbf{k}),\qquad\mathbf{k}\in \mathbb{Z}^2.
	   		\end{equation}
It is a consequence of \cref{eq:decay_fourier_transform} and \cite[Corollary VII.2.6]{stein:fourier} that for a function $f\in C_0^q(\mathbb{R}^2)$ with $q>2$ the Poisson summation formula 
		\begin{equation}\label{eq:poisson_summation}
			\sum_{\mathbf{k}\in \mathbb{Z}^2}\mathcal{F}f(\mathbf{k})\,\mathrm{e}^{\mathrm{i}\mathbf{k}^{\mathrm{T}}\mathbf{x}}=\sum_{\mathbf{n}\in \mathbb{Z}^2} f(\mathbf{x}+2\pi \mathbf{n})=f^{2\pi}(\mathbf{x})
		\end{equation}
		holds true for all $\mathbf{x}\in \mathbb{R}^2$.\\
		
In the following we prepare the proof of \cref{thm:upper_estimate} with several auxiliary lemmata. Note that in the proofs we only show the case $i=\mathrm{h}$ since the other case can be handled similarly.
			\begin{lemma}
				\label{lem:partial_derivative_psi}
				For $i\in\lbrace \mathrm{h},\mathrm{v}\rbrace$ and $q\in \mathbb{N}_0$ let $\Psi^{(i)}\in \mathcal{W}^q_2$ be given. Then for $\mathbf{r}\in \mathbb{N}_0^2$ with $\abs{\mathbf{r}}_1\leq q$ and a rotation matrix $\mathbf{R}_\gamma$ with $\gamma\in[0,2\pi)$ we have
				\begin{equation*}
					\abs{\partial^\mathbf{r}\Psi_{j,\ell}^{(i)}(\mathbf{R}_\gamma\,\boldsymbol{\xi})}\leq C(q)\,2^{-j\,\abs{\mathbf{r}}_1}\left(1+2^{(j+1)/2}\abs{\sin\left(\theta_{j,\ell}^{(i)}-\gamma\right)}\right)^{r_1}\left(1+2^{(j+1)/2}\abs{\cos\left(\theta_{j,\ell}^{(i)}-\gamma\right)}\right)^{r_2}.
				\end{equation*}
			\end{lemma}
			\begin{proof}
				We have $\mathbf{R}_\gamma=
					\begin{pmatrix}
						\cos\gamma & -\sin\gamma\\
						\sin\gamma & \cos\gamma
					\end{pmatrix}$ and use \cref{eq:Psi_jl} to see
				\begin{equation*}
					\Psi_{j,\ell}^{(\mathrm{h})}(\mathbf{R}_\gamma\,\boldsymbol{\xi})=g\left(2^{-j/2}(\xi_1\sin{\gamma}+\xi_2\cos{\gamma})-\ell\,2^{-j}(\xi_1\cos{\gamma}-\xi_2\sin{\gamma})\right)\widetilde{g}\Bigl(2^{-j}(\xi_1\cos{\gamma-}\xi_2\sin{\gamma})\Bigr).
				\end{equation*} 
				In this proof we will omit the long arguments of the function of the last line and simply write $g$ and $\widetilde{g}$. For $\mathbf{m}=(m_1,m_2)^{\mathrm{T}}$ with $\abs{\mathbf{m}}_1\leq q$ we use the chain rule to get
				\begin{equation*}
					\abs{\partial^\mathbf{m}\widetilde{g}}=\norm{\widetilde{g}}_{C^q}\,2^{-j\abs{\mathbf{m}}_1}\abs{\cos{\gamma}}^{m_1}\abs{\sin{\gamma}}^{m_2}\leq C(q)\,2^{-j\abs{\mathbf{m}}_1}
				\end{equation*}
				and, since $\ell=2^{j/2}\tan\left(\theta_{j,\ell}^{(\mathrm{h})}\right)$, we have
				\begin{align*}
					\abs{\partial^\mathbf{m}g}&=\norm{g}_{C^q}\abs{2^{-j/2}\sin{\gamma}-\ell\,2^{-j}\cos\gamma}^{m_1}\abs{2^{-j/2}\cos{\gamma}+\ell\,2^{-j}\sin{\gamma}}^{m_2}\\
					&=C_2(q)\,2^{-j\abs{\mathbf{m}}_1}\,\left(2^{j/2}\frac{\abs{\sin\left(\theta_{j,\ell}^{(\mathrm{h})}-\gamma\right)}}{\abs{\cos\left(\theta_{j,\ell}^{(\mathrm{h})}\right)}}\right)^{m_1}\left(2^{j/2}\frac{\abs{\cos\left(\theta_{j,\ell}^{(\mathrm{h})}-\gamma\right)}}{\abs{\cos\left(\theta_{j,\ell}^{(\mathrm{h})}\right)}}\right)^{m_2}.
				\end{align*}
				
				For sufficiently smooth functions $f,g:\mathbb{R}^2\rightarrow \mathbb{R}$ we employ the multivariate Leibniz rule
		\begin{equation}\label{eq:leibniz_rule} 
			\partial^{\mathbf{r}}(fg)=\sum_{\mathbf{0}\leq\mathbf{s}\leq \mathbf{r}}\binom{\mathbf{r}}{\mathbf{s}}\partial^{\mathbf{s}}f\partial^{\mathbf{r}-\mathbf{s}}g,
		\end{equation} 
		which together with the triangle inequality and the binomial theorem implies
				\begin{align*}
				\abs{\partial^\mathbf{r}\Psi_{j,\ell}^{(\mathrm{h})}(\mathbf{R}_\gamma\,\boldsymbol{\xi})}&\leq C_3(q)\,2^{-j\abs{\mathbf{r}}_1}\sum_{\mathbf{0}\leq \mathbf{s}\leq\mathbf{r}}\binom{\mathbf{r}}{\mathbf{s}}\left(2^{j/2}\frac{\abs{\sin\left(\theta_{j,\ell}^{(\mathrm{h})}-\gamma\right)}}{\abs{\cos\left(\theta_{j,\ell}^{(\mathrm{h})}\right)}}\right)^{s_1}\left(2^{j/2}\frac{\abs{\cos\left(\theta_{j,\ell}^{(\mathrm{h})}-\gamma\right)}}{\abs{\cos\left(\theta_{j,\ell}^{(\mathrm{h})}\right)}}\right)^{s_2}\\
					&\leq C_4(q)\,2^{-j\abs{\mathbf{r}}_1}\left(1+2^{(j+1)/2}\abs{\sin\left(\theta_{j,\ell}^{(i)}-\gamma\right)}\right)^{r_1}\left(1+2^{(j+1)/2}\abs{\cos\left(\theta_{j,\ell}^{(i)}-\gamma\right)}\right)^{r_2},
				\end{align*}
				since $2^{-1/2}\leq\abs{\cos\left(\theta_{j,\ell}^{(\mathrm{h})}\right)}\leq 1$.
			\end{proof}
In the following, we use notations and ideas from \cite{candes:curvelets,labate:sparse} and fix a function $\phi\in C_0^\infty([-\pi,\pi]^2)$. Denote $\phi_j(\mathbf{x})\mathrel{\mathop:}=\phi\left(2^{j/2}\,\mathbf{x}\right)$ and for $Q\in \mathcal{Q}_j$ given by \cref{eq:dyadic_squares} we define
\begin{equation*}
 \phi_Q(\mathbf{x})\mathrel{\mathop:}=\phi\left(2^{j/2}(x_1+\pi)-\pi(2k_1-1),2^{j/2}(x_2+\pi)-\pi(2k_2-1)\right)
\end{equation*} 
for $k_1,k_2=1,\hdots,2^{j/2}$ and assume that $\phi$ defines a smooth partition of unity
\begin{equation}\label{eq:phi_Q}
	\sum_{Q\in \mathcal{Q}_j}\phi_Q(\mathbf{x})=1,\qquad \mathbf{x}\in [-\pi,\pi)^2.
\end{equation}

The ideas of the proof of the next lemma can be found in \cite{candes:curvelets,labate:sparse}.
\begin{lemma}\label{lem:int_FT_Q0}
		For $u\in \mathbb{N}$ let $f\in C^u(\mathbb{R}^2)$ and $f_j\mathrel{\mathop:}=f\phi_j$. Then for $i\in\lbrace \mathrm{h},\mathrm{v}\rbrace$ and any $\mathbf{r}\in \mathbb{N}_0^2$ we have
		\begin{equation*}
			\int_{\mathrm{supp}\,\Psi_{j,\ell}^{(i)}}\abs{\partial^\mathbf{r}\left[\mathcal{F}f_j\right](\boldsymbol{\xi})}^2 \mathrm{d}\boldsymbol{\xi}\leq C(u,\mathbf{r})\,2^{-j(2u+1+\abs{\mathbf{r}}_1)}.
		\end{equation*}
	\end{lemma}
	\begin{proof}
		Since $\phi_j\in C^\infty_0(\mathbb{R}^2)$ we have $f_j\in C^u(\mathbb{R}^2)$ and using \cref{eq:leibniz_rule} we get
		\begin{equation*}
			\partial^{(u,0)} f_j=\sum_{s=0}^{u}\binom{u}{s}\partial^{(s,0)} \phi_j\,\partial^{(u-s,0)}f=\sum_{s=0}^{u} \eta_s,
		\end{equation*}
		where $\eta_s\mathrel{\mathop:}=\binom{u}{s}\,\partial^{(s,0)} \phi_j\,\partial^{(u-s,0)}f$. The function $\eta_s$ is $s$-times continuously differentiable with respect to the variable $\xi_1$. For $0\leq t\leq s$ we can estimate
		\begin{equation*}
			\norm{\partial^{(s+t,0)}\phi_j}_{\mathbb{R}^2,\infty}=\norm{2^{j(s+t)/2}\,\frac{\partial^{s+t}\phi}{\partial\xi_1^{s+t}}\left(2^{j/2}\cdot\right)}_{\mathbb{R}^2,\infty}\leq C_1\,2^{j(s+t)/2}\leq C_1\,2^{js},
		\end{equation*} 
		which leads to
		\begin{equation*}
			\norm{\partial^{(s,0)}\eta_s}_{\mathbb{R}^2,\infty}=\norm{\binom{u}{s}\sum_{t=0}^s\binom{s}{t}\partial^{(s+t,0)} \phi_j\,\partial^{(u-t,0)}f}_{\mathbb{R}^2,\infty}\leq C_2(u,s)\,2^{js}.
		\end{equation*}
	   By definition of the function $\phi_j$ we have $\abs{\mathrm{supp}\,\phi_j}\leq 2^{-j}$ and with property \cref{eq:properties_fourier3} and the Plancherel theorem we get
		\begin{equation*}
			\int_{\mathbb{R}^2}\abs{(2\pi)(\mathrm{i}\,\xi_1)^s\, \mathcal{F}\eta_s(\boldsymbol{\xi})}^2\mathrm{d}\boldsymbol{\xi}=\int_{\mathbb{R}^2}\abs{\partial^{(s,0)}\eta_s(\mathbf{x})}^2 \mathrm{d}\mathbf{x}\leq C_2(u)\,2^{j(2s-1)}.
		\end{equation*}
		
		For the first variable in $\mathrm{supp}\,\Psi_{j,\ell}^{(i)}$ we have $2^{j-1}\leq\xi_1\leq 2^{j+1}$ leading to
		\begin{equation*}
			(2\pi)^2(\mathrm{i}\,2^{j-1})^{2s} \int_{W_{j,\ell}^{(i)}}\abs{\mathcal{F}\eta_s(\boldsymbol{\xi})}^2\mathrm{d}\boldsymbol{\xi}\leq \int_{W_{j,\ell}^{(i)}}\abs{(2\pi)(\mathrm{i}\,\xi_1)^s\, \mathcal{F}\eta_s(\boldsymbol{\xi})}^2\mathrm{d}\boldsymbol{\xi}\leq C_2(u)\,2^{j(2s-1)},
		\end{equation*}
		which implies
	    \begin{equation}
	    \label{proof:estimate_f_smooth1}
	    	\int_{W_{j,\ell}^{(i)}}\abs{\mathcal{F}\eta_s(\boldsymbol{\xi})}^2\mathrm{d}\boldsymbol{\xi}\leq C_3(u)\,2^{-j}
	    \end{equation}
		for all $0\leq s\leq u$. Using again \cref{eq:properties_fourier3}, it follows that
		\begin{equation*}
			(\mathrm{i}\,\xi_1)^u\,\mathcal{F}f_j=\mathcal{F}\left[ \partial^{(u,0)}f_j\right]=\sum_{s=0}^{u}\,\mathcal{F}\eta_s,
		\end{equation*}
		which leads together with \cref{proof:estimate_f_smooth1} to
		\begin{align}
		\int_{W_{j,\ell}^{(i)}}\abs{\mathcal{F}f_j(\boldsymbol{\xi})}^2 \mathrm{d}\boldsymbol{\xi}&\leq C_4(u)\,2^{-2ju}\int_{W_{j,\ell}^{(i)}}\abs{(\mathrm{i}\,\xi_1)^{u}\, \mathcal{F}f_j(\boldsymbol{\xi})}^2\mathrm{d}\boldsymbol{\xi}\notag\\
			&\leq C_5(u)\,2^{-2ju}\sum_{s=0}^{u}\int_{W_{j,\ell}^{(i)}}\abs{\mathcal{F}\eta_s(\boldsymbol{\xi})}^2\mathrm{d}\boldsymbol{\xi}\notag\\\label{proof:estimate_f_smooth3}
			&\leq C_6(u)\,2^{-j(2u+1)}.
		\end{align}
		
		Next, we consider the function
		\begin{equation*}
			\mathbf{x}^\mathbf{r}f_j(\mathbf{x})=2^{-j\abs{\mathbf{r}}_1/2}f(\mathbf{x})\,2^{j\abs{\mathbf{r}}_1/2}\,\mathbf{x}^\mathbf{r}\phi_j\left(\mathbf{x}\right)=2^{-j\abs{\mathbf{r}}_1/2}\,f(\mathbf{x})\,\phi_\mathbf{r}\left(2^{j/2}\,\mathbf{x}\right),
		\end{equation*}
		where $\phi_\mathbf{r}(\mathbf{x})\mathrel{\mathop:}=\mathbf{x}^\mathbf{r}\,\phi_j(\mathbf{x})$ and see that $\phi_\mathbf{r}\left(2^{j/2}\cdot\right)\in C^\infty_0(\mathbb{R}^2)$ with $\abs{\mathrm{supp}\,\phi_\mathbf{r}}\leq 2^{-j}$ is fulfilled. Thus, the Fourier transform of the function $f(\mathbf{x})\,\phi_\mathbf{r}\left(2^{j/2}\mathbf{x}\right)$ satisfies \cref{proof:estimate_f_smooth3} with a constant $C_6(u,\mathbf{r})$. We use \cref{eq:properties_fourier2} to deduce
		\begin{equation*}
			\partial^\mathbf{r}\mathcal{F}f_j(\boldsymbol{\xi})=\mathcal{F}\left[( \mathrm{i}\,\mathbf{x})^\mathbf{r}f_j(\mathbf{x})\right](\boldsymbol{\xi})=\mathrm{i}^\mathbf{r}\,2^{-j\abs{\mathbf{r}}_1/2}\,\mathcal{F}\left[f(\mathbf{x})\,\phi_\mathbf{r}\left(2^{j/2}\mathbf{x}\right)\right](\boldsymbol{\xi}),
		\end{equation*}
		which leads to
		\begin{align*}
			\int_{W_{j,\ell}^{(i)}}\abs{\partial^\mathbf{r}\mathcal{F}f_j(\boldsymbol{\xi})}^2 \mathrm{d}\boldsymbol{\xi}&=2^{-j\abs{\mathbf{r}}_1}\int_{W_{j,\ell}^{(i)}}\abs{\mathcal{F}\left[f(\mathbf{x})\,\phi_\mathbf{r}\left(2^{j/2}\mathbf{x}\right)\right](\boldsymbol{\xi})}^2 \mathrm{d}\boldsymbol{\xi}\\
			&\leq C_7(u,\mathbf{r})\,2^{-j(2u+1+\abs{\mathbf{r}}_1)}.
		\end{align*}
	\end{proof}
	\begin{lemma}\label{lem:norm_FT_Q0}
		For $u\in \mathbb{N}$ let $f\in C^u(\mathbb{R}^2)$ and $f_j\mathrel{\mathop:}=f\phi_j$. Moreover for $i\in\lbrace \mathrm{h},\mathrm{v}\rbrace$ and $q\geq 2$ let $\Psi^{(i)}\in \mathcal{W}_2^q$ be given. Then for $Q\in \mathcal{Q}_j^0$ and any $\mathbf{r}\in \mathbb{N}_0^2$ with $\abs{\mathbf{r}}_1\leq q$ we have
		\begin{equation*}
			\norm{\partial^\mathbf{r}\left[ \mathcal{F}[f_j]\,\Psi_{j,\ell}^{(i)} \right]}_{\mathbb{R}^2,2}^2\leq C(u,q)\,2^{-j(2u+1+\abs{\mathbf{r}}_1)}.
		\end{equation*}
	\end{lemma}
	\begin{proof}
		For the partial derivative of the product inside of the norm we use the multivariate Leibniz rule \cref{eq:leibniz_rule} and obtain
		\begin{equation*}
			\norm{\partial^{\mathbf{r}}\left[\mathcal{F}[f_j]\,\Psi_{j,\ell}^{(i)}\right]}^2_{\mathbb{R}^2,2}\leq\sum_{\mathbf{0}\leq\mathbf{s}\leq \mathbf{r}}\binom{\mathbf{r}}{\mathbf{s}}\int_{\mathbb{R}^2}\abs{\partial^{\mathbf{s}}\left[\mathcal{F}f_Q\right](\boldsymbol{\xi})\,\partial^{\mathbf{r}-\mathbf{s}}\left[\Psi_{j,\ell}^{(i)}\right](\boldsymbol{\xi}) }^2 \mathrm{d}\boldsymbol{\xi}.
		\end{equation*}
		\cref{lem:partial_derivative_psi} implies that for all $\boldsymbol{\xi}\in \mathbb{R}^2$ the inequality
		\begin{equation*}
			\abs{\partial^{\mathbf{r}-\mathbf{s}}\left[\Psi_{j,\ell}^{(i)}\right](\boldsymbol{\xi})}^2\leq C_1(q)\,2^{-j(\abs{\mathbf{r}}_1-\abs{\mathbf{s}}_1)}
		\end{equation*}
		holds, independent of the orientation parameter $\ell$. Together with \cref{lem:int_FT_Q0} we have
		\begin{align*}
			\norm{\partial^\mathbf{r}\left( \mathcal{F}[f_j]\,\Psi_{j,\ell}^{(i)} \right)}_{\mathbb{R}^2,2}^2&\leq\sum_{\mathbf{0}\leq\mathbf{s}\leq \mathbf{r}}\binom{\mathbf{r}}{\mathbf{s}}\sup\limits_{\boldsymbol{\xi}\in \mathbb{R}^2}\abs{\partial^{\mathbf{r}-\mathbf{s}}\left[\Psi_{j,\ell}^{(i)}\right](\boldsymbol{\xi})}^2\int_{\mathrm{supp}\,\Psi_{j,\ell}^{(i)}}\abs{\partial^\mathbf{s}\left[\mathcal{F}f_j\right](\boldsymbol{\xi})}^2 \mathrm{d}\boldsymbol{\xi}\\
			&\leq \sum_{\mathbf{0}\leq\mathbf{s}\leq \mathbf{r}}\binom{\mathbf{r}}{\mathbf{s}} C_2(u,q)\,2^{-j(\abs{\mathbf{r}}_1-\abs{\mathbf{s}}_1)}\,2^{-j(2u+1+\abs{\mathbf{s}}_1)}\\
			&=C_3(u,q)\,2^{-j(2u+1+\abs{\mathbf{r}}_1)}.
		\end{align*}
	\end{proof}
	Following the approach from \cite[Chapter 6.1]{candes:curvelets} we assume that for $j\geq j_0$ the edge curve $\partial T$ can be parametrized on the support of $\phi_Q,\,Q\in \mathcal{Q}_j^1,\,$ either as $(x_1,E(x_1))^{\mathrm{T}}$ or $(E(x_2),x_2)^{\mathrm{T}}$.
	\begin{definition}
		For $x_2\in\left[-2^{-j/2},2^{-j/2}\right]$ let $(E(x_2),x_2)^{\mathrm{T}}$ be a parametrization of $\partial T$ with $E(0)=E'(0)=0$. For $f\in C^2(\mathbb{R}^2)$ we call
		\begin{equation*}
			\mathcal{E}_j(\mathbf{x})=f(\mathbf{x})\,\phi_j(\mathbf{x})\,\chi_{\lbrace x_1\geq E(x_2)\rbrace}(\mathbf{x})
		\end{equation*}
		standard edge fragment.
	\end{definition}
	
	Let $\mathcal{E}_{j,\mathbf{x}_0,\gamma}$ be an arbitrary edge fragment, which tangent in the point $\mathbf{x}_0\in\partial T$ is pointing in the direction $(\cos{\gamma},\sin{\gamma})^{\mathrm{T}}$ for $\gamma\in[0,2\pi)$. Clearly $\mathcal{E}_{j,\mathbf{0},0}=\mathcal{E}_j$ is a standard edge fragment. Moreover, in \cite[Corollary 6.7]{candes:curvelets} it is remarked that, although an arbitrary edge fragment $\mathcal{E}_{j,\mathbf{x}_0,\gamma}$ can not be obtained via rotation and translation of a standard edge fragment, we have the connection
	\begin{equation}\label{eq:rotate_edge_fragment}
		\mathcal{F}\mathcal{E}_{j,\mathbf{x}_0,\gamma}(\boldsymbol{\xi})=\mathrm{e}^{-\mathrm{i}\,\mathbf{x}_0^{\mathrm{T}}\boldsymbol{\xi}}\,\mathcal{F}\mathcal{E}_j(\mathbf{R}_\gamma^{\mathrm{T}}\,\boldsymbol{\xi})
	\end{equation}
	of their Fourier transforms. The following lemma is a consequence of \cite[Corollary 6.6]{candes:curvelets}.
	\begin{lemma}\label{lem:int_FT_radius}
	For $j\in \mathbb{N}$ let $I_j=\left[2^{j-1},2^{j+1}\right]$ and $\mathcal{E}_j$ be a standard edge fragment. Then for angles $\theta,\gamma\in[0,2\pi)$ and $\mathbf{r}\in \mathbb{N}_0^2$ we have
		\begin{equation*}
			\int\limits_{\abs{\rho}\in I_j}\abs{\partial^{\mathbf{r}}\left[\mathcal{F}\mathcal{E}_j\right]\Bigr(\rho\,\boldsymbol{\Theta}(\theta-\gamma)\Bigl)}^2 \mathrm{d}\rho\leq C(\mathbf{r})2^{-j(2+\abs{\mathbf{r}}_1)}\,\Bigl(1+2^{j/2}\abs{\sin\left(\theta-\gamma\right)}\Bigr)^{-5}.
	 	\end{equation*}
		\end{lemma}
		We can deduce the following result, which proof uses ideas from \cite[Proposition 2.1]{labate:sparse}.
	    \begin{lemma}\label{lem:int_FT_Q1}
	    	For $i\in\lbrace \mathrm{h},\mathrm{v}\rbrace$ let $\Psi^{(i)}\in \mathcal{W}_2^q$. Then for a standard edge fragment $\mathcal{E}_j$, a rotation matrix $\mathbf{R}_\gamma$ by the angle $\gamma\in[0,2\pi)$ and $\mathbf{r}\in \mathbb{N}_0^2$ we have
	    	\begin{equation*}
	\int\limits_{\mathrm{supp}\Psi_{j,\ell}^{(i)}}\abs{\partial^{\mathbf{r}}\left[\mathcal{F}\mathcal{E}_j\right]\Bigr(\mathbf{R}_\gamma^{\mathrm{T}}\,\boldsymbol{\xi}\Bigl)}^2 \mathrm{d}\boldsymbol{\xi}\leq C(\mathbf{r})\,2^{-j(3/2+\abs{\mathbf{r}}_1)}\,\left(1+2^{j/2}\abs{\sin(\theta_{j,\ell}^{(i)}-\gamma)}\right)^{-5}.
	    	\end{equation*}
	    \end{lemma}
		\begin{proof}
			From \cref{lem:support_Psi} we know $\mathrm{supp}\,\Psi_{j,\ell}^{(i)}\subset W_{j,\ell}^{(i)}$ and we transform the integral into polar coordinates and use \cref{lem:int_FT_radius} to obtain
			\begin{align*}
				\int\limits_{\mathrm{supp}\Psi_{j,\ell}^{(i)}}\abs{\partial^{\mathbf{r}}\left[\mathcal{F}\mathcal{E}_j\right]\Bigr(\mathbf{R}_\gamma^{\mathrm{T}}\,\boldsymbol{\xi}\Bigl)}^2 \mathrm{d}\boldsymbol{\xi}&\leq 2^{j+1}\int\limits_{\theta_{j,\ell-2}^{(i)}}^{\theta_{j,\ell+2}^{(i)}}\int\limits_{\frac{2^j}{3}}^{2^{j+1}}\abs{\partial^{\mathbf{r}}\left[\mathcal{F}\mathcal{E}_j\right]\Bigr(\rho\,\boldsymbol{\Theta}(\theta-\gamma)\Bigl)}^2\,\mathrm{d}\rho\,\mathrm{d}\theta\\
				&\leq C(\mathbf{r})\,2^{-j(1+\abs{\mathbf{r}}_1)}\int\limits_{\theta_{j,\ell-2}^{(i)}}^{\theta_{j,\ell+2}^{(i)}}\left(1+2^{j/2}\abs{\sin(\theta-\gamma)}\right)^{-5}\mathrm{d}\theta\\
				&\leq C_2(\mathbf{r})\,2^{-j(3/2+\abs{\mathbf{r}}_1)}\left(1+2^{j/2}\abs{\sin(\theta_{j,\ell}^{(i)}-\gamma)}\right)^{-5}.
			\end{align*}
		\end{proof}
    \begin{lemma}
    \label{lem:norm_FT_Q1}
    	For $i\in\lbrace \mathrm{h},\mathrm{v}\rbrace$ let $\Psi^{(i)}\in \mathcal{W}_2^q$. Then for a standard edge fragment $\mathcal{E}_j$, a rotation matrix $\mathbf{R}_\gamma$ by the angle $\gamma\in[0,2\pi)$ and $\mathbf{r}\in \mathbb{N}_0^2$ we have
    	\begin{equation*}
    		\norm{\partial^{\mathbf{r}}\left[\mathcal{F}\mathcal{E}_j(\mathbf{R}_\gamma^{\mathrm{T}}\cdot)\,\Psi_{j,\ell}^{(i)}\right]}^2_{\mathbb{R}^2,2}\leq C(q)\,2^{-j(3/2+\abs{\mathbf{r}}_1)}\,\left(1+2^{j/2}\abs{\sin(\theta_{j,\ell}^{(i)}-\gamma)}\right)^{-5}.
    	\end{equation*}
    \end{lemma}
	\begin{proof}
		We repeat the steps of the proof of \cref{lem:norm_FT_Q0} and use \cref{lem:int_FT_Q1} instead of \cref{lem:int_FT_Q0} in the last step.
	\end{proof}
	The Laplace operator is denoted by $\Delta\mathrel{\mathop:}=\partial^{(2,0)}+\partial^{(0,2)}$ and for $q\in \mathbb{N}_0$ we have
	\begin{equation}\label{eq:laplace_exp} 
		\Delta^q=\sum_{\abs{\mathbf{r}}_1=q}\binom{q}{\mathbf{r}}\partial^{2 \mathbf{r}}.
	\end{equation}
	For the next lemma we define the second order differential operator $L\mathrel{\mathop:}=I+2^j\Delta$, which was already used in \cite{candes:curvelets,labate:sparse}. Using \cref{eq:laplace_exp} we have
		\begin{equation}\label{eq:Lq}
			L^q=\left(I+2^j\Delta\right)^q=\sum_{s=0}^q\binom{q}{s}\,2^{js}\,\Delta^s=\sum_{s=0}^q\binom{q}{s}\,2^{js}\sum_{\abs{\mathbf{r}}_1=s}\binom{s}{\mathbf{r}}\partial^{2 \mathbf{r}}.
		\end{equation}
    \begin{lemma}\label{lem:norm_Lq}
		For $u\in \mathbb{N}$ let $f\in C^u(\mathbb{R}^2)$ and $f_j\mathrel{\mathop:}=f\phi_j$. Moreover for $i\in\lbrace \mathrm{h},\mathrm{v}\rbrace$ let $\Psi^{(i)}\in \mathcal{W}_2^{2q}$ with $q\geq2$. Then we have
		\begin{equation*}
			\norm{L^q\left[ \mathcal{F}[h]\,\Psi_{j,\ell}^{(i)} \right]}^2_{\mathbb{R}^2,2}\leq 
			\begin{cases}
				C_1(u,q)\,2^{-j(2u+1)} & h=f_j,\\
				C_2(q)\,2^{-3j/2}\,\left(1+2^{j/2}\abs{\sin\left(\theta_{j,\ell}^{(i)}-\gamma\right)}\right)^{-5} & h=\mathcal{E}_j(\mathbf{R}_\gamma^{\mathrm{T}}\cdot).
			\end{cases}
		\end{equation*}
    \end{lemma}
	\begin{proof}
		By applying the Cauchy-Schwarz inequality twice we obtain
		\begin{equation*}
			\norm{L^q\left[ \mathcal{F}[h]\,\Psi_{j,\ell}^{(i)}\right]}^2_{\mathbb{R}^2,2}\leq q \sum_{s=0}^q\binom{q}{s}^2\,(s+1)\,2^{2js}\,\sum_{\abs{\mathbf{r}}_1=s}\binom{s}{\mathbf{r}}^2\norm{ \partial^{2 \mathbf{r}}\left[ \mathcal{F}[h]\,\Psi_{j,\ell}^{(i)}\right] }^2_{\mathbb{R}^2,2}
		\end{equation*}
		and get the result by inserting the corresponding upper bounds for the norm from \cref{lem:norm_FT_Q0} and \cref{lem:norm_FT_Q1}.
	\end{proof}

The last part of this section consists of lemmata which are needed to proof \cref{thm:lower_estimate}. We start with some important localization properties.

	\begin{lemma}\label{lem:decay_psi}
		For $i\in\lbrace \mathrm{h},\mathrm{v}  \rbrace$ and $q\geq 2$ let $\Psi^{(i)}\in \mathcal{W}_2^{2q}$ be given. Then for all $\mathbf{x}\in[-\pi,\pi)^2$ we have
		\begin{equation*}
			\abs{\psi_{j,\ell,\mathbf{y}}^{(i)}(\mathbf{x})}\leq C(q)\,2^{3j/2}\min\left\lbrace 1,\frac{\left( 1+2^{(j+1)/2}\abs{\sin\left(\theta_{j,\ell}^{(i)}-\gamma\right)}\right)^q}{\left(2^j\abs{\mathbf{x}-2\pi\widetilde{\mathbf{y}}}_2\right)^q}\right\rbrace,
		\end{equation*}
		where $\mathbf{x}-2\pi\widetilde{\mathbf{y}}=\abs{\mathbf{x}-2\pi\widetilde{\mathbf{y}}}_2(\cos{\gamma},\sin{\gamma})^{\mathrm{T}}$ for $\gamma\in[0,2\pi)$.
	\end{lemma}
	\begin{proof}
		The function $\Psi_{j,\ell}^{(\mathrm{h})}$ is nonnegative leading to
		\begin{equation*}
			\abs{\psi_{j,\ell,\mathbf{y}}^{(\mathrm{h})}(\mathbf{x})}\leq\sum_{\mathbf{k}\in \mathbb{Z}^2}\Psi^{(\mathrm{h})}_{j,\ell}(\mathbf{k})=\psi_{j,\ell,\mathbf{y}}^{(\mathrm{h})}(2\pi\widetilde{\mathbf{y}})\leq C_1\,2^{3j/2},
		\end{equation*} 
		where the last estimate follows from \cref{eq:support_Psi}.
		
		Since $\Psi_{j,\ell}^{(\mathrm{h})}\in \mathcal{W}^{2q}_2$, we can use the Poisson summation formula \cref{eq:poisson_summation} to arrive at 
		\begin{align}
			\abs{\psi_{j,\ell,\mathbf{y}}^{(\mathrm{h})}(\mathbf{x})}&=\abs{\sum_{\mathbf{k}\in \mathbb{Z}^2}\Psi_{j,\ell}^{(\mathrm{h})}(\mathbf{k})\,\mathrm{e}^{\mathrm{i}\mathbf{k}^{\mathrm{T}}(\mathbf{x}-2\pi\widetilde{\mathbf{y}})}}\notag\\\label{eq:proof_decay0}
			&=\abs{\sum_{\mathbf{n}\in \mathbb{Z}^2}\mathcal{F}^{-1}\Psi_{j,\ell}^{(\mathrm{h})}\Bigl(\mathbf{x}-2\pi(\widetilde{\mathbf{y}}-\mathbf{n})\Bigr)}\leq\sum_{\mathbf{n}\in \mathbb{Z}^2}\abs{S(\mathbf{n})}, 
			\end{align}
			where
			\begin{equation*}
				S(\mathbf{n})\mathrel{\mathop:}=\int_{\mathbb{R}^2}\Psi_{j,\ell}^{(\mathrm{h})}(\boldsymbol{\xi})\,\mathrm{e}^{\mathrm{i}\,\boldsymbol{\xi}^{\mathrm{T}}(\mathbf{x}-2\pi(\widetilde{\mathbf{y}}-\mathbf{n}))}\mathrm{d}\boldsymbol{\xi}.
			\end{equation*}
			Let $\mathbf{R}_\gamma$ be a rotation matrix by the angle $\gamma$. Then
			\begin{equation*}
				\mathbf{R}^{\mathrm{T}}_\gamma(\mathbf{x}-2\pi\widetilde{\mathbf{y}})=\abs{\mathbf{x}-2\pi\widetilde{\mathbf{y}}}_2\mathbf{R}^{\mathrm{T}}_\gamma\,(\cos{\gamma},\sin{\gamma})^{\mathrm{T}}=\abs{\mathbf{x}-2\pi\widetilde{\mathbf{y}}}_2\,(1,0)^{\mathrm{T}}
			\end{equation*}
			and in the integral $S(\mathbf{0})$ we use this rotation matrix for a change of variable to see
			\begin{equation*}
				S(\mathbf{0})=\int_{\mathbb{R}^2}\Psi_{j,\ell}^{(\mathrm{h})}(\mathbf{R}_\gamma\,\boldsymbol{\xi})\,\mathrm{e}^{\mathrm{i}\,\xi_1\abs{\mathbf{x}-2\pi\widetilde{\mathbf{y}}}_2}\mathrm{d}\boldsymbol{\xi}.
			\end{equation*}
			Since the function $\Psi_{j,\ell}^{(\mathrm{h})}$ is compactly supported, we can use $q$-times partial integration together with \cref{lem:partial_derivative_psi} and \cref{eq:support_Psi} to deduce
			\begin{equation}\label{eq:proof_decay1}
				\abs{S(\mathbf{0})}\leq \frac{\sup\limits_{\boldsymbol{\xi}\in \mathbb{R}^2}\abs{\partial^{(q,0)}\Psi_{j,\ell}^{(\mathrm{h})}(\mathbf{R}_\gamma\cdot)}}{\abs{\mathbf{x}-2\pi\widetilde{\mathbf{y}}}_2^q}\hspace{-0.8cm}\int\limits_{\mathrm{supp}\,\Psi_{j,\ell}^{(\mathrm{h})}(\mathbf{R}_\gamma\cdot)}\hspace{-0.8cm}\mathrm{d}\boldsymbol{\xi}\leq \frac{C(q)\,2^{3j/2}\left(1+2^{(j+1)/2}\abs{\sin\left(\theta_{j,\ell}^{(\mathrm{h})}-\gamma\right)}\right)^{q}}{\left(2^j\abs{\mathbf{x}-2\pi\widetilde{\mathbf{y}}}_2\right)^q}.
			\end{equation}
			
			Using the same idea as before we substitute with the rotation matrices $\mathbf{R}_{\gamma_{\mathbf{n}}}$ in the integrals $S(\mathbf{n})$, where $\gamma_{\mathbf{n}}$ is the direction of the vector $\mathbf{x}-2\pi\left(\widetilde{\mathbf{y}}-\mathbf{n}\right)$. Similar to \cref{eq:proof_decay1} we use $2q$-times integration by parts with respect to the first variable, \cref{lem:partial_derivative_psi} and \cref{eq:support_Psi} to obtain
			\begin{align}
				\abs{S(\mathbf{n})}&\leq\frac{C_2(q)\,2^{3j/2}\left(1+2^{(j+1)/2}\abs{\sin\left(\theta_{j,\ell}^{(\mathrm{h})}-\gamma_\mathbf{n}\right)}\right)^{2q}}{\left(2^j\abs{\mathbf{x}-2\pi(\widetilde{\mathbf{y}}-\mathbf{n})}_2\right)^{2q}}\notag\\\label{eq:proof_decay2}
				&\leq\frac{C_3(q)\,2^{3j/2}}{\left(2^{j/2}\abs{\mathbf{x}-2\pi(\widetilde{\mathbf{y}}-\mathbf{n})}_2\right)^{2q}}.
			\end{align}
			Observe that
			\begin{equation*}
				\pi^2\geq\pi\abs{\mathbf{x}-2\pi\widetilde{\mathbf{y}}}_\infty\geq \frac{\pi}{\sqrt{2}}\abs{\mathbf{x}-2\pi\widetilde{\mathbf{y}}}_2\geq\abs{\mathbf{x}-2\pi\widetilde{\mathbf{y}}}_2
			\end{equation*}
			and with the inverse triangle inequality we can estimate
		\begin{equation*}
			\abs{\mathbf{x}-2\pi(\widetilde{\mathbf{y}}-\mathbf{n})}_2\geq 2\pi\abs{\mathbf{n}}_\infty-\abs{\mathbf{x}-2\pi \widetilde{\mathbf{y}}}_\infty\geq\sqrt{\abs{\mathbf{x}-2\pi\widetilde{\mathbf{y}}}_2}(2\abs{\mathbf{n}}_\infty-1).
		\end{equation*}
		Since $\abs{\left\lbrace \mathbf{n}\in \mathbb{Z}^2\,;\,\abs{\mathbf{n}}_\infty=k,\,k\in \mathbb{N} \right\rbrace}=8k$ we use \cref{eq:proof_decay2} to conclude
		\begin{align}
			\sum_{\mathbf{n}\in \mathbb{Z}^2\setminus\lbrace \mathbf{0} \rbrace}\abs{S(\mathbf{n})}&\leq C_3(q)\,2^{-jq}\,2^{3j/2}\sum_{\mathbf{n}\in \mathbb{Z}^2\setminus\lbrace \mathbf{0} \rbrace}\abs{\mathbf{x}-2\pi(\widetilde{\mathbf{y}}-\mathbf{n})}_2^{-2q}\notag\\\label{eq:proof_decay3}
			&\leq \frac{C_3(q)\,2^{3j/2}}{\left(2^j\abs{\mathbf{x}-2\pi\widetilde{\mathbf{y}}}_2\right)^q}\sum_{k=1}^\infty \frac{8k}{(2k-1)^{2q}}
			\end{align}
			and the infinite sum in the last line converges because $q\geq2$. We finish the proof by making use of \cref{eq:proof_decay1} and \cref{eq:proof_decay3} in \cref{eq:proof_decay0}.
		\end{proof}
		
Let $\boldsymbol{\gamma}:[0,2\pi)\rightarrow\partial T$ be a parametrization of the boundary $\partial T$. We assume there is $M\in \mathbb{N}$ such that for each $x\in[a_k,b_k], k=1,\hdots,M,$ the curve $\boldsymbol{\gamma}$ can either be represented as a horizontal curve $(x,f(x))^{\mathrm{T}}$ or a vertical curve $(f(x),x)^{\mathrm{T}}$. Depending on the choice of the parameter $i\in\lbrace \mathrm{h},\mathrm{v}\rbrace$ we will distinguish if a curve is horizontal or vertical. If $i=\mathrm{h}$ then $(f(x),x)^{\mathrm{T}}$ with $\abs{f'(x)}\leq 1$ is a vertical curve and $(x,f(x))^{\mathrm{T}}$ with $\abs{f'(x)}<1$ is a horizontal curve. On the other hand, if $i=\mathrm{v}$ then $(f(x),x)^{\mathrm{T}}$ with $\abs{f'(x)}<1$ is a vertical curve and $(x,f(x))^{\mathrm{T}}$ with $\abs{f'(x)}\leq 1$ is a horizontal curve.

Let $\mathbf{y}=(y_1,y_2)^{\mathrm{T}}\in\mathcal{P}\left(\mathbf{N}_{j,\ell}^{(\mathrm{h})}\right)$. We assume that the boundary curve can be vertically parametrized by $(f(x),x)^{\mathrm{T}}$ for $\abs{x-2\pi y_2}\leq 2^{-j/2}$. For $m\mathrel{\mathop:}=f'(2\pi y_2)\in[-1,1]$ and $A\mathrel{\mathop:}=\frac{1}{2}\,f''(2\pi y_2)$ let 
\begin{equation}\label{eq:T_approx}
	T_\mathbf{y}(x)=f(2\pi y_2)+m(x-2\pi y_2)+A(x-2\pi y_2)^2
\end{equation}
be the second order Taylor approximation for $f(x)$ in the point $x_0=2\pi y_2$. Denote by $\widehat{\mathcal{T}}_\mathbf{y}^{(\mathrm{h})}$ the modified version of $\mathcal{T}$ by replacing the function $f(x)$ by the approximation $T_\mathbf{y}(x)$ for $\abs{x-2\pi y_2}\leq 2^{-j/2}$ if the corresponding parametrization is a vertical curve and similarly $\widehat{\mathcal{T}}_\mathbf{y}^{(\mathrm{v})}$ as the modified version of $\mathcal{T}$ if the parametrization is a horizontal curve. Although this notation seems to be counterintuitive, it is convenient since by \cref{lem:orientation_lemma} only the interaction of horizontal wavelets with vertical curves and vertical wavelets with horizontal curves contributes to the desired lower bound in \cref{thm:lower_estimate}. The analog of the following lemma can be found in \cite{labate:detection} for the discrete and in \cite{labate:detection_continuous2,labate:detection_continuous} for the continuous setting.
\begin{lemma}\label{lem:T_tilde}
	For $i\in\lbrace\mathrm{h},\mathrm{v}\rbrace$ and large $q\in \mathbb{N}$ let $\Psi^{(i)}\in \mathcal{W}^{2q}_2$ be given. Then for $\mathbf{y}\in\mathcal{P}\left(\mathbf{N}_{j,\ell}^{(i)}\right)$ we have
	\begin{equation*}
		\abs{\left\langle \mathcal{T}^{2\pi}-\left(\widehat{\mathcal{T}}_\mathbf{y}^{(i)}\right)^{2\pi},\psi_{j,\ell,\mathbf{y}}^{(i)} \right\rangle_2}\leq C(q)\,2^{-j/4}.
	\end{equation*}
\end{lemma}
\begin{proof}
	We only show the proof for $i=\mathrm{h}$ since the other case is similar. For this proof we define the set $B_j=\left\lbrace(x_1,x_2)^{\mathrm{T}}\in [-\pi,\pi)^2\,:\,\abs{x_2-2\pi \widetilde{y}_2}\leq 2^{-7j/16}\right\rbrace$ and write
	\begin{align*}
			\abs{\left\langle\mathcal{T}^{2\pi}-\left(\widehat{\mathcal{T}}_\mathbf{y}^{(\mathrm{h})}\right)^{2\pi},\psi_{j,\ell,\mathbf{y}}^{(\mathrm{h})}\right\rangle_2}&\leq\int\limits_{\mathbb{T}^2}\abs{\psi^{(\mathrm{h})}_{j,\ell,\mathbf{y}}(\mathbf{x})}\abs{\chi_{T}(\mathbf{x})-\chi_{\widehat{T}_\mathbf{y}^{(\mathrm{h})}}(\mathbf{x})}\mathrm{d}\mathbf{x}\\
			&=\Biggl(\;\int\limits_{B_j}+\int\limits_{B_j^\mathrm{c}}\Biggr)\abs{\psi^{(\mathrm{h})}_{j,\ell,\mathbf{y}}(\mathbf{x})}\abs{\chi_{T}(\mathbf{x})-\chi_{\widehat{T}_\mathbf{y}^{(\mathrm{h})}}(\mathbf{x})}\mathrm{d}\mathbf{x}\\
			&=\mathrel{\mathop:}\mathcal{I}_1+\mathcal{I}_2.
	\end{align*}
	Using the definition of $T_\mathbf{y}(x)$ in \cref{eq:T_approx} we can estimate
	\begin{equation*}
		\abs{f(x)-T_\mathbf{y}(x)}\leq C\abs{x-2\pi y_2}^3
	\end{equation*}
	for the area between $\mathcal{T}$ and $\widehat{\mathcal{T}}_\mathbf{y}^{(\mathrm{h})}$ if $\abs{x-2\pi y_2}\leq 2^{-j/2}$. From \cref{lem:decay_psi} we can obtain the uniform bound $\abs{\psi_{j,\ell,\mathbf{y}}^{(\mathrm{h})}(\mathbf{x})}\leq C(q)\,2^{3j/2}$ and we can estimate the first integral by
	\begin{equation*}
	 \abs{\mathcal{I}_1}\leq C(q)\,2^{3j/2}\int\limits_{\abs{x-2\pi \widetilde{y}_2}\leq 2^{-7j/16}}\abs{x-2\pi y_2}^3\mathrm{d}x\leq C(q)\, 2^{3j/2}2^{-7j/4}= C(q)\,2^{-j/4}.
	\end{equation*}
In addition we use again \cref{lem:decay_psi} but this time for the decay term in the minimum to arrive at
	\begin{align*}
	 \abs{\mathcal{I}_2}&\leq C(q)\,2^{3j/2}\hspace{-0.7cm}\int\limits_{\abs{x-2\pi\widetilde{y}_2}>2^{-7j/16}}\left(2^{j/2}\abs{x-2\pi\widetilde{y}_2}\right)^{-q}\,\mathrm{d}x\\
	 &\leq C_3(q)\,2^{3j/2}\,2^{-jq/2}\,2^{7(q-1)j/16}=C_3(q)\ 2^{-j(q/16-17/16)}
	\end{align*}
	for the second integral, which shows that the lemma is proved for $q\geq 21$.
\end{proof}

From the divergence theorem one can see that the Fourier transform of a characteristic function $\mathcal{T}=\chi_T$ is given by
		\begin{equation}\label{eq:T_fourier}
		\mathcal{F}\mathcal{T}(\boldsymbol{\xi})=(2\pi)^{-2}\int\limits_{\mathbb{R}^2}\chi_T(\mathbf{x})\,\mathrm{e}^{-\mathrm{i}\mathbf{x}^{\mathrm{T}}\boldsymbol{\xi}}\, \mathrm{d}\mathbf{x}=\frac{\mathrm{i}\left(2\pi\right)^{-2}}{\abs{\boldsymbol{\xi}}_2}\int\limits_{\partial T}\,\mathrm{e}^{-\mathrm{i}\mathbf{x}^{\mathrm{T}}\boldsymbol{\xi}}\,\boldsymbol{\Theta}^{\mathrm{T}}(\theta)\,\mathbf{n}(\mathbf{x})\,\mathrm{d}\sigma(\mathbf{x})
		\end{equation}
		with the outer normal vector $\mathbf{n}(\mathbf{x})$.  We remind the parametrization of $\partial T$ given by $\boldsymbol{\gamma}(x),\,x\in[0,2\pi)$ and use polar coordinates to represent the line integral \cref{eq:T_fourier} as
		\begin{equation*}
		\mathcal{F}\mathcal{T}(\rho,\theta)=\frac{\mathrm{i}}{\left(2\pi\right)^{2}\rho}\int\limits_{0}^{2\pi}\mathrm{e}^{-\mathrm{i}\rho\,\boldsymbol{\Theta}^{\mathrm{T}}(\theta)\,\boldsymbol{\gamma}(x)}\,\boldsymbol{\Theta}^{\mathrm{T}}(\theta)\,\mathbf{n}(\boldsymbol{\gamma}(x))\abs{\boldsymbol{\gamma}'(x)}_2\mathrm{d}x=\frac{\mathrm{i}}{\left(2\pi\right)^{2}\rho}\sum_{k=1}^{M} \mathcal{I}_k(\rho,\theta),
		\end{equation*}
		where
		\begin{equation}\label{eq:I_k}
			\mathcal{I}_k(\rho,\theta)\mathrel{\mathop:}=\int\limits_{a_k}^{b_k}\mathrm{e}^{-\mathrm{i}\rho\,\boldsymbol{\Theta}^{\mathrm{T}}(\theta)\,\boldsymbol{\gamma}_k(x)}\,\boldsymbol{\Theta}^{\mathrm{T}}(\theta)\,\boldsymbol{\beta}_k(x)\,\mathrm{d}x
		\end{equation}
		and $\boldsymbol{\beta}_k(x)\mathrel{\mathop:}=\mathbf{n}(\boldsymbol{\gamma}_k(x))\sqrt{1+\left(f_k'(x)\right)^2}$. Using this we can conclude
		\begin{align}
			\mathcal{F}^{-1}\left[ \mathcal{F}[\mathcal{T}]\Psi^{(i)}_{j,\ell} \right](2\pi\widetilde{\mathbf{y}})&=\int_{\mathbb{R}^2}\mathcal{F}\mathcal{T}(\boldsymbol{\xi})\,\Psi^{(i)}_{j,\ell}(\boldsymbol{\xi})\,\mathrm{e}^{2\pi\mathrm{i}\boldsymbol{\xi}^{\mathrm{T}}\widetilde{\mathbf{y}}}\mathrm{d}\boldsymbol{\xi}\notag\\
		&=\frac{\mathrm{i}}{\left(2\pi\right)^2}\sum_{k=1}^{M}\int\limits_{0}^{\infty} \int\limits_{0}^{2\pi}\Psi_{j,\ell}^{(i)}(\rho,\theta)\,
		\mathrm{e}^{2\pi \mathrm{i} \rho \boldsymbol{\Theta}^{\mathrm{T}}(\theta)\,\widetilde{\mathbf{y}}}\,\mathcal{I}_k(\rho,\theta)\,\mathrm{d}\theta\,\mathrm{d}\rho\notag\\\label{eq:F_inv}
		&=\frac{2^j\,\mathrm{i}}{\left(2\pi\right)^2}\sum_{k=1}^{M}\int\limits_{0}^{\infty} \int\limits_{0}^{2\pi}\Psi_{j,\ell}^{(i)}(2^j\rho,\theta)\,
		\mathrm{e}^{2\pi \mathrm{i} 2^j\rho \boldsymbol{\Theta}^{\mathrm{T}}(\theta)\,\widetilde{\mathbf{y}}}\,\mathcal{I}_k(2^j\rho,\theta)\,\mathrm{d}\theta\,\mathrm{d}\rho,
		\end{align}
		where we again transformed the integral into polar coordinates and the interchange of summation and integration is valid since the function $\Psi_{j,\ell}^{(i)}$ has finite support.
		
Denote by $\mathcal{M}^{(\mathrm{h})}\subset\lbrace 1,\hdots,M\rbrace$ the set of all indices such that for $k\in\mathcal{M}^{(\mathrm{h})}$ the curve $\boldsymbol{\gamma}_{k}$ is horizontal and by $\mathcal{M}^{(\mathrm{v})}\subset\lbrace 1,\hdots,M\rbrace$ the set of all indices such that for $k\in\mathcal{M}^{(\mathrm{v})}$ the curve $\boldsymbol{\gamma}_{k}$ is vertical. Obviously we have $\mathcal{M}^{(\mathrm{h})}\cup\mathcal{M}^{(\mathrm{v})}=\lbrace 1,\hdots,M\rbrace$ and can prove the following lemma, which idea of proof was given in \cite{labate:detection} 
 \begin{lemma}\label{lem:orientation_lemma} 
	For $i\in\lbrace\mathrm{h},\mathrm{v}\rbrace$ and $q\in \mathbb{N}$ let $\Psi^{(i)}\in \mathcal{W}^q_2$ be given. Then for any $n\in \mathbb{N}$ there is a constant $C(n)>0$ such that for every $k\in \mathcal{M}^{(i)}$ we have
	\begin{equation*}
		\abs{\int\limits_{0}^{\infty} \int\limits_{0}^{2\pi}\Psi_{j,\ell}^{(i)}(2^j\rho,\theta)\,\mathrm{e}^{2\pi \mathrm{i} 2^j\rho\, \boldsymbol{\Theta}^{\mathrm{T}}(\theta)\,\widetilde{\mathbf{y}}}\,\mathcal{I}_{k}(2^j\rho,\theta)\,\mathrm{d}\theta\,\mathrm{d}\rho}\leq C(n)\,2^{-jn}.
	\end{equation*}
\end{lemma}
\begin{proof}
	We show the result for $i=\mathrm{h}$ since the other case is similar. Suppose that $k\in\mathcal{M}^{(\mathrm{h})}$ and $\boldsymbol{\gamma}_k(x)=(x,f_k(x))^{\mathrm{T}}$ for $x\in[a_k,b_k]$. Hence the outer normal vector in $x$ is given by $\mathbf{n}(\boldsymbol{\gamma}_k(x))=(f_k'(x),-1)^{\mathrm{T}}$ leading to
	\begin{equation*}
		\mathcal{I}_k(2^j\rho,\theta)=\int\limits_{a_k}^{b_k}\mathrm{e}^{-\mathrm{i}2^j\rho\,\boldsymbol{\Theta}^{\mathrm{T}}(\theta)\,(x,f_k(x))^{\mathrm{T}}}\,(f_k'(x),-1)\,\boldsymbol{\Theta}(\theta)\,\beta_k(x)\,\mathrm{d}x.
	\end{equation*}
	From the assumption on horizontal curves for $i=\mathrm{h}$ we have $\abs{f_k'(x)}<1$ and the support properties of the function $\Psi_{j,\ell}^{(\mathrm{h})}$ given in \cref{lem:support_Psi} imply
	\begin{equation*}
		\abs{\theta}\leq\theta_{j,2^{j/2}+2}^{(\mathrm{h})}=\arctan\left(1+2^{1-j/2}\right)\leq\frac{\pi}{4}+\delta
	\end{equation*}
	for some small $\delta>0$. From that we conclude
	\begin{equation*}
		\abs{\frac{\partial}{\partial x}\Bigl[(x,f_k(x))\,\boldsymbol{\Theta}(\theta)\Bigr]}=\abs{\frac{\partial}{\partial x}\Bigl[\cos\theta(x+f_k(x)\tan\theta)\Bigr]}\geq\abs{\cos\theta}(1-\abs{f_k'(x)\tan\theta})\geq C.
	\end{equation*}
	For $n\in \mathbb{N}$ we do $n$-times integration by parts with respect to the variable $x$ and obtain $\abs{\mathcal{I}_k(2^j\rho,\theta)}\leq 2^{-jN}$, which leads to
	\begin{equation*}
		\abs{\int\limits_{0}^{\infty} \int\limits_{0}^{2\pi}\Psi_{j,\ell}^{(\mathrm{h})}(2^j\rho,\theta)\,\mathrm{e}^{2\pi \mathrm{i} 2^j\rho \boldsymbol{\Theta}^{\mathrm{T}}(\theta)\,\widetilde{\mathbf{y}}}\,\mathcal{I}_{k}(2^j\rho,\theta)\,\mathrm{d}\theta\,\mathrm{d}\rho}\leq C(n)\,2^{-jn} \abs{\mathrm{supp}\,\Psi_{j,\ell}^{(\mathrm{h})}(2^j\cdot)}.
	\end{equation*}
	\cref{lem:support_Psi} implies $\abs{\mathrm{supp}\,\Psi_{j,\ell}^{(\mathrm{h})}(2^j\cdot)}\leq C$ and since $n\in \mathbb{N}$ was arbitrary the lemma is proven.
\end{proof}
The proof of the following lemma was given in \cite{labate:detection_continuous}.
\begin{lemma}\label{lem:localization_lemma}
	For $i\in\lbrace\mathrm{h},\mathrm{v}\rbrace$ and $q\in \mathbb{N}$ let $\Psi^{(i)}\in \mathcal{W}^q_2$ and $\mathbf{y}=(y_1,y_2)^{\mathrm{T}}\in \mathcal{P}\left(\mathbf{N}_{j,\ell}^{(i)}\right)$ be given. Then for any $n\in \mathbb{N}$ there is a constant $C(n)>0$ such that 
	\begin{equation*}
		\abs{\int\limits_{0}^{\infty} \int\limits_{0}^{2\pi}\int\limits_{\abs{x-2\pi y_2}>2^{-j/2}}\hspace{-0.8cm}\Psi_{j,\ell}^{(i)}(\rho,\theta)\,\mathrm{e}^{\mathrm{i} 2^j\rho\boldsymbol{\Theta}^{\mathrm{T}}(\theta)\left(2\pi\widetilde{\mathbf{y}}-\boldsymbol{\gamma}(x)\right)}\,\boldsymbol{\Theta}^{\mathrm{T}}(\theta)\,\boldsymbol{\beta}(x)\,\mathrm{d}x\,\mathrm{d}\theta\,\mathrm{d}\rho}\leq C(n)\,2^{-jn}.
	\end{equation*}
\end{lemma}
The following lemma is a special case of \cite[Proposition 8.3]{stein:harmonic}, called method of stationary phase.
		 \begin{lemma}
		 \label{lem:constant_phase}
		 	Let $\phi$ and $\varphi$ be smooth functions on the real line. Suppose $\phi'(t_0)=0$ and $\phi''(t_0)\neq 0$. If $\varphi$ is supported in a sufficiently small neighborhood of $t_0$, then
		 	\begin{equation*}
		 	\int\limits_{\mathbb{R}}\mathrm{e}^{\mathrm{i}\,\Lambda\,\phi(t)}\,\varphi(t)\,\mathrm{d}t=a_0\;\Lambda^{-1/2}+O(\Lambda^{-1})
		 	\end{equation*}
		 	as $\Lambda\rightarrow\infty$, where $a_0=\left(\frac{2\pi \mathrm{i}}{\abs{\phi''(t_0)}}\right)^{1/2}\varphi(t_0)$.
		 \end{lemma}
		 
		 For $x\in(0,\infty)$ we introduce the so-called Fresnel integrals
		 	\begin{equation*}
		 		Fc(x)\mathrel{\mathop:}=2\int\limits_0^{\sqrt{x}} \cos{\left(v^2\right)}\,\mathrm{d}v=\int\limits_0^{x} \frac{\cos{\left(v\right)}}{\sqrt{v}}\,\mathrm{d}v,\quad Fs(x)\mathrel{\mathop:}=2\int\limits_0^{\sqrt{x}} \sin{\left(t^2\right)}\,\mathrm{d}t=\int\limits_0^{x} \frac{\sin{\left(v\right)}}{\sqrt{v}}\,\mathrm{d}v
		 	\end{equation*}
			and define the functions $F^+(x)\mathrel{\mathop:}=Fc(x)+Fs(x)$ and $F^-(x)\mathrel{\mathop:}=Fc(x)-Fs(x)$ to show the following lemma.
			
		    \begin{figure}[t]
		  		{\includegraphics[width=\textwidth]{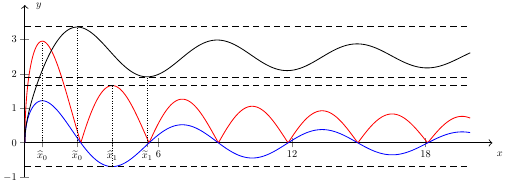}}
		  	  	\caption{Visualization of the functions $F^+(x)$ (black), $F^-(x)$ (blue) and $\left(1+\sqrt{2}\right)\abs{F^-(x)}$ (red) for $x\in[0,20]$ together with upper and lower bounds of these functions (dashed lines) and the local extremal points from \cref{lem:properties_fresnel}.}\label{fig:fresnel}
		  	    \end{figure}
				
			\begin{lemma}
			\label{lem:properties_fresnel}
				We have 
				\begin{align*}
					F^+(x)>F^-(x)>0\qquad&\text{for}\quad 0<x<\frac{3\pi}{4},\\
					F^+(x)>\left(1+\sqrt{2}\right)\bigl\lvert F^-(x)\bigr\rvert \qquad&\text{for}\quad x\geq\frac{3\pi}{4}.
				\end{align*}
			\end{lemma}
			\begin{proof}
				It is clear that $Fc(0)=Fs(0)=0$ and for $x>0$ it is well known that $Fc(x)>0$ and $Fs(x)>0$, which implies 
				\begin{equation}\label{proof:properties_fresnel0}
					F^+(x)>Fc(x)>F^-(x).
				\end{equation}
				Moreover we have
				\begin{equation*}
					\frac{\mathrm{d}}{\mathrm{d}x}F^+(x)=\frac{\cos{x}+\sin{x}}{\sqrt{x}}=0
				\end{equation*}
				for $\widetilde{x}_k\mathrel{\mathop:}=\frac{3\pi}{4}+k\pi,\,k\in \mathbb{N}_0$ and
				\begin{equation*}
					\frac{\mathrm{d}^2}{\mathrm{d}x^2}F^+(x)=\frac{\cos{x}-\sin{x}}{\sqrt{x}}-\frac{\cos{x}+\sin{x}}{2(\sqrt{x})^3}.
				\end{equation*}
				Since 
				\begin{equation*}
					\frac{\cos{\widetilde{x}_k}-\sin{\widetilde{x}_k}}{\sqrt{\widetilde{x}_k}}-\frac{\cos{\widetilde{x}_k}+\sin{\widetilde{x}_k}}{2(\sqrt{\widetilde{x}_k})^3}=\frac{\cos{\widetilde{x}_k}-\sin{\widetilde{x}_k}}{\sqrt{\widetilde{x}_k}}=
					\begin{cases}
						-\sqrt{\frac{2}{\widetilde{x}_k}}<0 &\text{for } k\;\text{even},\\
						\sqrt{\frac{2}{\widetilde{x}_k}}>0 &\text{for } k\;\text{odd},
					\end{cases}
					\end{equation*}
					we see that $\widetilde{x}_k$ is a local maximum point of $F^+$ for $k$ even and a local minimum point for $k$ odd. To get information about global extremal points we define 
					\begin{equation*}
						I_0\mathrel{\mathop:}=\int\limits_0^{\widetilde{x}_0} \frac{\cos{v}+\sin{v}}{\sqrt{v}}\,\mathrm{d}v\qquad\text{and}\qquad I_k\mathrel{\mathop:}=\int\limits_{\widetilde{x}_{k-1}}^{\widetilde{x}_k} \frac{\cos{v}+\sin{v}}{\sqrt{v}}\,\mathrm{d}v\quad\text{for\;}k\geq 1
					\end{equation*}
					and write $F^+(\widetilde{x}_k)=\sum\limits_{j=0}^k I_j$. For $v\in(\widetilde{x}_{k-1},\widetilde{x}_k)$ we have $(\cos{v}+\sin{v})>0$ for $k$ even and $(\cos{v}+\sin{v})<0$ for $k$ odd. It follows that $I_k>0$ for $k$ even and $I_k<0$ for $k$ odd. Additionally, we have $I_0>3.36>1.45>\abs{I_1}$ and since $v^{-1/2}$ is monotonically decreasing we have $\abs{I_k}>\abs{I_{k+1}}$ for $k\geq 1$. 
					
					Putting these observations together, we obtain $(I_k+I_{k+1})>0$ for $k$ even and $(I_k+I_{k+1})<0$ for $k$ odd. From that we deduce that $\widetilde{x}_0=\frac{3\pi}{4}$ is the global maximum point of $F^+$ since for even $k>0$ we have
					\begin{equation}\label{proof:properties_fresnel1}
						F^+(\widetilde{x}_k)=I_0+\sum\limits_{j=1}^k I_j< I_0=F^+(\widetilde{x}_0)<3.37.
					\end{equation}
					Similarly for $k>1$ odd we can write
					\begin{equation}\label{proof:properties_fresnel2}
						F^+(\widetilde{x}_k)=I_0+I_1+\sum\limits_{j=2}^k I_j>I_0+I_1=F^+(\widetilde{x}_1)>1.91,
					\end{equation}
					which shows that smallest local minimum is obtained at the point $\widetilde{x}_1=\frac{7\pi}{4}$ and thus $1.91<F^+(x)<3.37$ for $x\geq\widetilde{x}_0$. For $0\leq x<\widetilde{x}_0$ we clearly have $0\leq F^+(x)<3.37$.
					
					Similarly one can show that $\widehat{x}_k\mathrel{\mathop:}=\frac{\pi}{4}+k\pi,\,k\in \mathbb{N}_0,$ is a local maximum point of $F^-$ for $k$ even and a local minimum point for $k$ odd. Since $\widehat{x}_0<\frac{3\pi}{4}<\widehat{x}_1=\frac{5\pi}{4}$ and $F^-(\frac{3\pi}{4})>0.14>0$ we have $F^-(x)>0$ for $0<x<\frac{3\pi}{4}$, which together with \cref{proof:properties_fresnel0} gives the first statement of the lemma. 
					
					With similar arguments as in \cref{proof:properties_fresnel1} and \cref{proof:properties_fresnel2} we have
					\begin{equation*}
						-0.69<F^-(\widehat{x}_1)\leq F^-(x)\leq F^-(\widehat{x}_2)<0.53
					\end{equation*}
					for $x\geq\frac{3\pi}{4}$. Since $\left(1+\sqrt{2}\right)0.69<1.91$ we obtain $F^+(x)>\left(1+\sqrt{2}\right)\bigl\lvert F^-(x)\bigr\rvert $ for $x\geq\frac{3\pi}{4}$ and the proof is complete.
			\end{proof}
		 
	 	For the next lemma we define the integrals
	 		 \begin{align}\label{def:a}
	 		a(\lambda,p,A)&\mathrel{\mathop:}=\int\limits_{0}^{\infty}\left(g\left(2\sqrt{A\lambda\,v}+ p\,\lambda\right)+g\left(2\sqrt{A\lambda\,v}- p\,\lambda\right)\right)\,\frac{\cos{v}}{\sqrt{v}}\,\mathrm{d}v,\\\label{def:b}
			b(\lambda,p,A)&\mathrel{\mathop:}=\int\limits_{0}^{\infty}\left(g\left(2\sqrt{A\lambda\,v}+ p\,\lambda\right)+g\left(2\sqrt{A\lambda\,v}- p\,\lambda\right)\right)\,\frac{\sin{v}}{\sqrt{v}}\,\mathrm{d}v.
	 		 \end{align}

\begin{lemma}
\label{lem:a_b_minus}\label{lem:a_b}
	For $\lambda\in\left[\frac{1}{3},\frac{4}{3}\right]$, $p\in\left[-\frac{1}{4},\frac{1}{4}\right]$ and $A>0$ we have $a(\lambda,p,A)>0$ and $b(\lambda,p,A)>0$ and at least one of the inequalities
	\begin{equation*}
		a(\lambda,p,A)-b(\lambda,p,A)>0
	\end{equation*}
	or
	\begin{equation*}
		a(\lambda,p,A)+b(\lambda,p,A)>\left(1+\sqrt{2}\right)\bigl\lvert a(\lambda,p,A)-b(\lambda,p,A)\bigr\rvert
	\end{equation*}
	is true.
\end{lemma}
\begin{proof}
We consider only the case $p\in\left[0,\frac{1}{4}\right]$ because $a(\lambda,p,A)$ and $b(\lambda,p,A)$ are symmetric in that variable. Let the function $h^{\pm}:\left[\frac{1}{3},\frac{4}{3}\right]\times\left[0,\frac{1}{4}\right]\times(0,\infty)\times[0,\infty)\rightarrow \mathbb{R}$ be given by
\begin{equation*}
	h^{\pm}(\lambda,p,A,v)\mathrel{\mathop:}=2\sqrt{A\lambda\,v}\pm p\,\lambda.
\end{equation*}
For fixed $(\lambda,p,A)\in\left[\frac{1}{3},\frac{4}{3}\right]\times\left[0,\frac{1}{4}\right]\times(0,\infty)$ we have
\begin{equation}\label{proof:lem:a_b0}
	-\frac{1}{3}\leq -p\lambda=h^{-}(\lambda,p,A,0)\leq h^{+}(\lambda,p,A,0)=p\lambda\leq\frac{1}{3}
\end{equation}
and the functions $h^{\pm}$ are monotonically increasing in the variable $v\geq 0$. Since $g\in \mathcal{W}^q$ this implies that in the variable $v$ the functions $g\bigl(h^+(\lambda,p,A,v)\bigr)$ and $g\bigl(h^-(\lambda,p,A,v)\bigr)$ and hence the sum $\Bigl(g\bigl(h^+(\lambda,p,A,v)\bigr)+g\bigl(h^-(\lambda,p,A,v)\bigr)\Bigr)$ are also monotonically decreasing. Similarly we have
\begin{equation*}
	h^{+}(\lambda,p,A,v)\geq h^{-}(\lambda,p,A,v)\geq \frac{2}{3}\qquad\text{if}\qquad v\geq \frac{p}{3A}+\frac{1}{9A\lambda}+\frac{p^2\lambda}{4A}=\mathrel{\mathop:}r(\lambda,p,A)>0,
\end{equation*} 
which leads to $\Bigl(g\bigl(h^+(\lambda,p,A,v)\bigr)+g\bigl(h^-(\lambda,p,A,v)\bigr)\Bigr)=0$ for $v\geq r(\lambda,p,A)$. Moreover \cref{proof:lem:a_b0} implies
\begin{equation*}
	\lim\limits_{v\rightarrow 0^+}g\left(h^+(\lambda,p,A,v)\right)+g\left(h^-(\lambda,p,A,v)\right)=2
\end{equation*}
and we use the mean value theorem of integration to deduce that there exists $x\in(0,r(\lambda,p,A)]$ such that
\begin{align*}
	a(\lambda,p,A)\pm b(\lambda,p,A)&=\hspace{-0.4cm}\int\limits_{0}^{r(\lambda,p,A)}\hspace{-0.2cm}\left(g\left(2\sqrt{A\lambda v}+ p\lambda\right)+g\left(2\sqrt{A\lambda  v}- p\lambda\right)\right)\frac{\cos{v}\pm\sin{v}}{\sqrt{v}}\mathrm{d}v\\
	&=2\,F^{\pm}(x).
\end{align*}
Using \cref{lem:properties_fresnel} the proof is finished.
\end{proof}
For the last lemma of this section we define the following integrals
\begin{align*}
	P_1(D,p,A)&\mathrel{\mathop:}=\int\limits_{\frac{1}{3}}^{\frac{4}{3}}\widetilde{g}(\lambda)\,\lambda^{-1}\,\Bigl(\Bigl[a(\lambda,p,A)+b(\lambda,p,A)\Bigr]\cos(D\lambda)+\Bigl[a(\lambda,p,A)-b(\lambda,p,A)\Bigr]\sin(D\lambda)\Bigr)\,\mathrm{d}\lambda,
	\end{align*}
	\begin{align*}
	P_2(D,p,A)&\mathrel{\mathop:}=\int\limits_{\frac{1}{3}}^{\frac{4}{3}}\widetilde{g}(\lambda)\,\lambda^{-1}\,\Bigl(\Bigl[a(\lambda,p,A)+b(\lambda,p,A)\Bigr]\sin(D\lambda)-\Bigl[a(\lambda,p,A)-b(\lambda,p,A)\Bigr]\cos(D\lambda)\Bigr)\,\mathrm{d}\lambda.
\end{align*}
\begin{lemma}\label{lem:I11_I21_positive}
There is a constant $C>0$ such that for all $D\in\left[-\frac{3\pi}{4},\frac{3\pi}{4}\right]$, $p\in\left[-\frac{1}{4},\frac{1}{4}\right]$ and $A>0$ at least one of the inequalities
			\begin{equation}\label{proof:P_12}
				\Bigl(\bigl\lvert P_1(D,p,A)\bigr\rvert\geq C\Bigr)\quad \text{or}\quad\Bigl(\bigl\lvert P_2(D,p,A)\bigr\rvert \geq C\Bigr)
			\end{equation}
			is true.
		\end{lemma}
		\begin{proof}
			We define
		\begin{align*}
			P^+(D,p,A)\mathrel{\mathop:}&=P_1(D,p,A)+P_2(D,p,A)\\
			&=2\int\limits_{\frac{1}{3}}^{\frac{4}{3}}\widetilde{g}(\lambda)\,\lambda^{-1}\,\Bigl[b(\lambda,p,A)\cos(D\lambda)+a(\lambda,p,A)\sin(D\lambda)\Bigr]\mathrm{d}\lambda
		\end{align*}
		and show that there exist a constant $C>0$ that either one of the statements in \cref{proof:P_12} or equivalently $\Bigl(\abs{P^+(D,p,A)}\geq C\Bigr)$ is true. For the rest of the proof the variables $p\in\left[-\frac{1}{4},\frac{1}{4}\right]$ and $A>0$ will be arbitrary and fixed. For simplicity we assume $D\in[0,\frac{3\pi}{4}]$ since the proof for negative values of $D$ is similar. We consider different intervals for the variable $D$ and show that at least one of the equivalent propositions see \cref{proof:P_12} or $\Bigl(\abs{P^+(D,p,A)}\geq C\Bigr)$ is true. From the construction of the window function we deduce $\widetilde{g}(\lambda)\,\lambda^{-1}>0$ and from \cref{lem:a_b} we know $a(\lambda,p,A)+b(\lambda,p,A)>0$ for $\lambda\in\left[\frac{1}{3},\frac{4}{3}\right]$. Thus for $D=0$ it holds that
		\begin{equation*}
			P_1(0,p,A)=\int\limits_{\frac{1}{3}}^{\frac{4}{3}}\widetilde{g}(\lambda)\,\lambda^{-1}\Bigl[a(\lambda,p,A)+b(\lambda,p,A)\Bigr]\mathrm{d}\lambda>0.
		\end{equation*}
		
		For $D\in\left(0,\frac{3\pi}{8}\right]$ we have $D\lambda\in \left(0,\frac{\pi}{2}\right)$ leading to $\sin(D\lambda)>0$ and $\cos(D\lambda)>0$ and from \cref{lem:a_b} we know $a(\lambda,p,A)>0$ and $b(\lambda,p,A)>0$ which gives $P^+(D,p,A)>0$.
		
		For $D\in\left(\frac{3\pi}{8},\frac{3\pi}{4}\right]$ we split up $\int\limits_{\frac{1}{3}}^{\frac{4}{3}}=\left(\int\limits_{\frac{1}{3}}^{\frac{\pi}{2D}}+\int\limits_{\frac{\pi}{2D}}^{\frac{\pi}{D}}+\int\limits_{\frac{\pi}{D}}^{\frac{4}{3}}\right)$ and write 
		\begin{align*}
			P_2(D,p,A)&\mathrel{\mathop:}=P_{21}(D,p,A)+P_{22}(D,p,A)+P_{23}(D,p,A),\\
			P^+(D,p,A)&\mathrel{\mathop:}=P^+_{1}(D,p,A)+P^+_{2}(D,p,A)+P^+_{3}(D,p,A).
		\end{align*} 
		Since $\frac{4}{3}\leq\frac{\pi}{D}$ and $\mathrm{supp}\,\widetilde{g}=\left[\frac{1}{3},\frac{4}{3}\right]$, it follows that $P_{23}(D,p,A)=P^+_{3}(D,p,A)=0$. 
		
		We assume $a(\lambda,p,A)>b(\lambda,p,A)>0$ (the first case of \cref{lem:a_b}). In the integral $P^+_{1}$ we have $D\lambda\in\left(\frac{\pi}{8},\frac{\pi}{2}\right)$, hence $\sin(D\lambda)>0$ and $\cos(D\lambda)>0$. This leads directly to $P^+_{1}(D,p,A)>0$. Since $D\lambda\in \left(\frac{\pi}{2},\pi\right)$ implies $\cos(D\lambda)<0$ and $\sin(D\lambda)>0$ in the integral $P^+_{2}$ we can estimate
		\begin{align}\label{proof:Pp_2}
			P^+_{2}(D,p,A)&>\int\limits_{\frac{\pi}{2D}}^{\frac{\pi}{D}}\widetilde{g}(\lambda)\,\lambda^{-1}\,a(\lambda,p,A)\Bigl(\cos(D\lambda)+\sin(D\lambda)\Bigr)\mathrm{d}\lambda\notag\\
			&=\int\limits_{\frac{\pi}{2D}}^{\frac{\pi}{D}}\widehat{h}(\lambda,p,A)\Bigl(\cos(D\lambda)+\sin(D\lambda)\Bigr)\mathrm{d}\lambda
		\end{align}
		with $\widehat{h}(\lambda,p,A)\mathrel{\mathop:}=\widetilde{g}(\lambda)\,\lambda^{-1}\,a(\lambda,p,A)$. This function is monotonically decreasing for $\lambda\in\left[\frac{2}{3},\frac{4}{3}\right]$. Using the substitution $t=D\lambda-\frac{3\pi}{4}$, we obtain
		\begin{align}
			P^+_{2}(D,p,A)&>-\frac{\sqrt{2}}{D}\int\limits_{-\frac{\pi}{4}}^{\frac{\pi}{4}}\widehat{h}\left(\frac{3\pi+4t}{4D},p,A\right)\sin{t}\,\mathrm{d}t\notag\\\label{proof:Pp_2_1}
			&=\frac{\sqrt{2}}{D}\int\limits_{0}^{\frac{\pi}{4}}\left[\widehat{h}\left(\frac{3\pi-4t}{4D},p,A\right)-\widehat{h}\left(\frac{3\pi+4t}{4D},p,A\right)\right]\sin{t}\,\mathrm{d}t>0,
		\end{align}
		where we used the monotonicity of the function $\widehat{h}$ to deduce the last inequality. Overall, for $D\in\left(\frac{3\pi}{8},\frac{3\pi}{4}\right]$ and $a(\lambda,p,A)>b(\lambda,p,A)>0$ we showed $P^+_{1}(D,p,A)>0$, $P^+_{2}(D,p,A)>0$ and $P^+_{3}(D,p,A)=0$, which leads to $P^+(D,p,A)>0$ in that case.
		
		Let us assume $a(\lambda,p,A)+b(\lambda,p,A)>\left(1+\sqrt{2}\right)\bigl\lvert a(\lambda,p,A)-b(\lambda,p,A)\bigr\rvert$ (the second case of \cref{lem:a_b}). In $P_{21}$ we have $D\lambda\in\left(\frac{\pi}{8},\frac{\pi}{2}\right)$, hence $\left(1+\sqrt{2}\right)\sin(D\lambda)>\cos(D\lambda)>0$, which allows for the estimate
				\begin{equation*}
					P_{21}(D,p,A)>\int\limits_{\frac{1}{3}}^{\frac{\pi}{2D}}\widetilde{g}(\lambda)\lambda^{-1}\bigl\lvert a(\lambda,p,A)-b(\lambda,p,A)\bigr\rvert\,\left(\left(1+\sqrt{2}\right)\sin(D\lambda)-\cos(D\lambda)\right)\mathrm{d}\lambda>0.
				\end{equation*}
		To estimate the integral $P_{22}$ we can use exactly the same arguments as in \cref{proof:Pp_2} and \cref{proof:Pp_2_1} but this time with the function $\widetilde{h}(\lambda,p,A)\mathrel{\mathop:}=\widetilde{g}(\lambda)\,\lambda^{-1}\,\bigl(a(\lambda,p,A)+b(\lambda,p,A)\bigr)$ instead of $\widehat{h}$ which gives $P_{22}(D,p,A)>0$ and overall $P_2(D,p,A)>0$.
		\end{proof}
		
		\section{Proof of the main results}\label{sec:proof_of_the_main_results}

We start with the proof of \cref{thm:upper_estimate}.

Recall that we denote the set of dyadic squares $Q\subseteq[-\pi,\pi)^2$ of the form \cref{eq:dyadic_squares} for $j\in \mathbb{N}_0$ by $\mathcal{Q}_j$ and smooth functions $\phi_Q$ with support on these dyadic squares with the property
\begin{equation*}
	\sum_{Q\in \mathcal{Q}_j}\phi_Q(\mathbf{x})=1,\qquad \mathbf{x}\in [-\pi,\pi)^2,
\end{equation*}
are defined in \cref{eq:phi_Q}. Moreover, for $u\in \mathbb{N}$ let $f\in C^u(\mathbb{R}^2)$ and define $f_Q\mathrel{\mathop:}=f\phi_Q$ for $Q\in \mathcal{Q}_j$. We can decompose
\begin{equation}\label{eq:decomposition_T}
	f=\sum_{Q\in \mathcal{Q}_j}f_Q=\sum_{Q\in \mathcal{Q}_j^0}f_Q+\sum_{Q\in \mathcal{Q}_j^1}f_Q,
\end{equation}
where $Q\in\mathcal{Q}_j^1\subseteq\mathcal{Q}_j$ if $\partial T\cap Q\neq\emptyset$. For the non-intersecting squares we define $\mathcal{Q}_j^0\mathrel{\mathop:}=\mathcal{Q}_j\setminus\mathcal{Q}_j^1$.

According to \cref{eq:periodization} we denote by $f_Q^{2\pi}$ the $2\pi$-periodization of $f_Q$. From the finite support of $f_Q$ we deduce $f_Q\in L_1(\mathbb{R}^2)$ and from \cref{eq:poisson_fourier} we get
		\begin{equation*}
			c_\mathbf{k}(f_Q^{2\pi})=\mathcal{F}[f_Q](\mathbf{k}),\qquad \mathbf{k}\in \mathbb{Z}^2.
		\end{equation*}
		Moreover by \cref{eq:properties_fourier2} we have $\mathcal{F}f_Q\in C^q(\mathbb{R}^2)$ for all $q\in \mathbb{N}_0$. The smoothness assumption on the window function $\Psi^{(i)}_{j,\ell}\in \mathcal{W}^{2q}_2$ implies $\mathcal{F}[f_Q]\,\Psi_{j,\ell}^{(i)} \in C_0^q(\mathbb{R}^2)$. Thus, the estimates \cref{eq:decay_fourier_transform} hold for this function and with Parseval's identity and the Poisson summation formula it follows
		\begin{align*}
			\left\langle f_Q^{2\pi},\psi^{(i)}_{j,\ell,\mathbf{y}}\right\rangle_2&=\sum_{\mathbf{k}\in \mathbb{Z}^2}\mathcal{F}[f_Q](\mathbf{k})\,\Psi^{(i)}_{j,\ell}(\mathbf{k})\,\mathrm{e}^{2\pi\mathrm{i}\mathbf{k}^{\mathrm{T}}\widetilde{\mathbf{y}}}\\
			&=\sum_{\mathbf{n}\in \mathbb{Z}^2}\mathcal{F}^{-1}\left[\mathcal{F}[f_Q]\Psi^{(i)}_{j,\ell} \right]\Bigl(2\pi(\widetilde{\mathbf{y}}+\mathbf{n})\Bigr)=\sum_{\mathbf{n}\in \mathbb{Z}^2}S_Q(\mathbf{n}),
		\end{align*}
where
\begin{equation*}
	S_Q(\mathbf{n})\mathrel{\mathop:}=\int\limits_{\mathbb{R}^2}\mathcal{F}[f_Q](\boldsymbol{\xi})\,\Psi^{(i)}_{j,\ell}(\boldsymbol{\xi})\,\mathrm{e}^{2\pi\mathrm{i}\boldsymbol{\xi}^{\mathrm{T}}(\widetilde{\mathbf{y}}+\mathbf{n})}\,\mathrm{d}\boldsymbol{\xi}.
\end{equation*}
For $Q\in \mathcal{Q}_j^0$ we choose $\mathbf{x}_1\in[-\pi,\pi]^2$ such that 
\begin{equation}\label{proof:upper_estimate0}
	1\leq\abs{2\pi \widetilde{\mathbf{y}}-\mathbf{x}_1}_\infty\leq\abs{2\pi \widetilde{\mathbf{y}}-\mathbf{x}_1}_2\leq\pi
\end{equation} 
and define $\widetilde{f}(\mathbf{x})\mathrel{\mathop:}=f_Q(\mathbf{x}-\mathbf{x}_1)$. From \cref{eq:rotate_edge_fragment} we see that $\mathcal{F}[\widetilde{f}](\boldsymbol{\xi})=\mathrm{e}^{\mathrm{i}\,\mathbf{x}_1^{\mathrm{T}}\boldsymbol{\xi}}\,\mathcal{F}[f_Q](\boldsymbol{\xi})$ and since $\mathcal{F}\widetilde{f}\,\Psi_{j,\ell}^{(i)}\in C_0^q(\mathbb{R}^2)$ we can use integration by parts repeatedly in both variables for every $\mathbf{r}\in \mathbb{N}_0^2$ with $\abs{\mathbf{r}}_1\leq q$ to obtain
				\begin{equation*}
					\Bigl(2\pi \mathrm{i}(\widetilde{\mathbf{y}}+\mathbf{n})-\mathbf{x}_1\Bigr)^{\mathbf{r}}\,S_Q(\mathbf{n})=\int_{\mathbb{R}^2}\partial^{\mathbf{r}}\left[\mathcal{F}[\widetilde{f}]\,\Psi^{(i)}_{j,\ell}\right](\boldsymbol{\xi})\,\mathrm{e}^{\mathrm{i}\,\boldsymbol{\xi}^{\mathrm{T}}(2\pi(\widetilde{\mathbf{y}}+\mathbf{n})-\mathbf{x}_1)}\mathrm{d}\boldsymbol{\xi}.
				\end{equation*}
				With the calculation
				\begin{equation*}
					\Bigl(1+2^j\abs{2\pi(\widetilde{\mathbf{y}}+\mathbf{n})-\mathbf{x}_1 }_2^2\Bigr)^q=\sum_{s=0}^q\binom{q}{s}2^{js}\sum_{\abs{\mathbf{r}}_1=s}\binom{s}{\mathbf{r}}\Bigl(2\pi(\widetilde{\mathbf{y}}+\mathbf{n})-\mathbf{x}_1\Bigr)^{2\mathbf{r}}
				\end{equation*}
				and the representation \cref{eq:Lq} of the $q$-th order differential operator $L^q$ we have
				\begin{equation}\label{proof:upper_estimate1}
					\Bigl(1+2^j\abs{2\pi(\widetilde{\mathbf{y}}+\mathbf{n})-\mathbf{x}_1}_2^2\Bigr)^q S_Q(\mathbf{n})=\int_{\mathbb{R}^2}L^q\left[\mathcal{F}[\widetilde{f}]\,\Psi^{(i)}_{j,\ell}\right](\boldsymbol{\xi})\,\mathrm{e}^{\mathrm{i}\,\boldsymbol{\xi}^{\mathrm{T}}(2\pi(\widetilde{\mathbf{y}}+\mathbf{n})-\mathbf{x}_1)}\mathrm{d}\boldsymbol{\xi}.
				\end{equation}
				A consequence from H\"older's inequality for a set $A\subset \mathbb{R}^2$ with finite Lebesgue measure $\abs{A}$, parameters $1\leq p\leq s<\infty$ and a function $f\in L_s(A)$ is the estimate
				\begin{equation}\label{eq:holder}
					\norm{f}_{A,p}\leq\abs{A}^{(s-p)/(p\,s)}\,\norm{f}_{A,s}.
				\end{equation}
				From \cref{lem:norm_Lq} we conclude $L^q\left[\mathcal{F}[\widetilde{f}]\,\Psi^{(i)}_{j,\ell}\right]\in L_2(\mathbb{R}^2)$ and with the estimate \cref{eq:holder} for $p=1$ and $s=2$ together with the upper bound for the support size of $\Psi^{(\mathrm{i})}_{j,\ell}$ given by \cref{eq:support_Psi} we see that 
				\begin{equation}\label{proof:upper_estimate2}
					\norm{L^q\left[\mathcal{F}[\widetilde{f}]\,\Psi^{(i)}_{j,\ell}\right] }_{\mathbb{R}^2,1}\leq 2^{3j/4}\norm{L^q\left[\mathcal{F}[\widetilde{f}]\,\Psi^{(i)}_{j,\ell}\right] }_{\mathbb{R}^2,2}.
				\end{equation}
				Next, \cref{proof:upper_estimate1} and \cref{proof:upper_estimate2} and \cref{lem:norm_Lq} for $u=2$ imply
				\begin{align}
					\abs{\left\langle f_Q^{2\pi},\psi^{(i)}_{j,\ell,\mathbf{y}}\right\rangle_2}&\leq\sum_{\mathbf{n}\in \mathbb{Z}^2}\norm{L^q\left[\mathcal{F}[\widetilde{f}]\,\Psi^{(i)}_{j,\ell}\right] }_{\mathbb{R}^2,1}\Bigl(1+2^j\abs{2\pi(\widetilde{\mathbf{y}}+\mathbf{n})-\mathbf{x}_1 }_2^2\Bigr)^{-q}\notag\\\label{proof:upper_estimate3}
					&\leq C(q)\,2^{-7j/4}\sum_{\mathbf{n}\in \mathbb{Z}^2}\Bigl(1+2^j\abs{2\pi(\widetilde{\mathbf{y}}+\mathbf{n})-\mathbf{x}_1}_2^2\Bigr)^{-q}.
				\end{align}
				We split up the infinite sum in the last line into
				\begin{equation}\label{proof:upper_estimate4}
				\Bigl(1+2^j\abs{2\pi\widetilde{\mathbf{y}}-\mathbf{x}_1}_2^2\Bigr)^{-q}+\sum_{\mathbf{n}\in \mathbb{Z}^2\setminus\{\mathbf{0}\}}\Bigl(1+2^j\abs{2\pi(\widetilde{\mathbf{y}}+\mathbf{n})-\mathbf{x}_1 }_2^2\Bigr)^{-q},
				\end{equation}
				where due to \cref{proof:upper_estimate0} the summand corresponding to $\mathbf{n}=\mathbf{0}$ is bounded from above by $C(q)\,2^{-jq}$. With the monotonicity of finite vector norms and the inverse triangle inequality we get
					\begin{equation*}
						\abs{2\pi(\widetilde{\mathbf{y}}+\mathbf{n})-\mathbf{x}_1}_2\geq\abs{2\pi(\widetilde{\mathbf{y}}+\mathbf{n})-\mathbf{x}_1}_{\infty}\geq \pi\left(2\abs{\mathbf{n}}_\infty-\abs{2\widetilde{\mathbf{y}}-\frac{\mathbf{x}_1}{\pi}}_\infty\right)\geq \pi(2\abs{\mathbf{n}}_\infty-1)
					\end{equation*}
					for $\mathbf{n}\neq \mathbf{0}$, because again with \cref{proof:upper_estimate0} we have $\abs{2\,\widetilde{\mathbf{y}}-\frac{\mathbf{x}_1}{\pi}}_\infty\leq 1$. Moreover the equation $\abs{\left\lbrace \mathbf{n}\in \mathbb{Z}^2\,;\,\abs{\mathbf{n}}_\infty=k,\,k\in \mathbb{N} \right\rbrace}=8k$ holds, leading to
					\begin{align}
				\sum_{\mathbf{n}\in \mathbb{Z}^2\setminus\{\mathbf{0}\}}\Bigl(1+2^j\abs{2\pi(\widetilde{\mathbf{y}}+\mathbf{n})-\mathbf{x}_1 }_2^2\Bigr)^{-q}&\leq C(q)\,2^{-jq}\sum_{k=1}^\infty\sum_{\abs{\mathbf{n}}_{\infty}=k}\left(2\abs{\mathbf{n}}_{\infty}-1\right)^{-2q}\notag\\\label{proof:upper_estimate5}
						&=C(q)\,2^{-jq}\sum_{k=1}^\infty\frac{8k}{(2k-1)^{2q}}\leq C_2(q)\,2^{-jq}.
					\end{align}
					Using the splitting \cref{proof:upper_estimate4} and the corresponding upper bound \cref{proof:upper_estimate5} for the infinite sum in \cref{proof:upper_estimate3} we get
					\begin{equation*}
						\abs{\left\langle f_Q^{2\pi},\psi^{(i)}_{j,\ell,\mathbf{y}}\right\rangle_2}\leq C(q)\,2^{-j(7/4+q)}
					\end{equation*}
					in the case $Q\in \mathcal{Q}_j^0$.\\
					
					For $Q\in \mathcal{Q}_j^1$ we use \cref{eq:rotate_edge_fragment} to write
					\begin{equation*}
						S_Q(\mathbf{n})\mathrel{\mathop:}=\int\limits_{\mathbb{R}^2}\mathcal{F}[\mathcal{E}_j](\mathbf{R}_\gamma^{\mathrm{T}}\,\boldsymbol{\xi})\,\Psi^{(i)}_{j,\ell}(\boldsymbol{\xi})\,\mathrm{e}^{\mathrm{i}\,\boldsymbol{\xi}^{\mathrm{T}}(2\pi(\widetilde{\mathbf{y}}+\mathbf{n})-\mathbf{x}_0)}\,\mathrm{d}\boldsymbol{\xi},
					\end{equation*}
					where $\mathcal{E}_j$ is a standard edge fragment. With the same arguments as in the first case and again \cref{lem:norm_Lq} we can deduce
					\begin{align}
						\abs{\left\langle f_Q^{2\pi},\psi^{(i)}_{j,\ell,\mathbf{y}}\right\rangle_2}&\leq\sum_{\mathbf{n}\in \mathbb{Z}^2}\norm{L^q\left[\mathcal{F}[\mathcal{E}_j](\mathbf{R}_\gamma^{\mathrm{T}}\cdot)\,\Psi^{(i)}_{j,\ell}\right] }_{\mathbb{R}^2,1}\Bigl(1+2^j\abs{2\pi(\widetilde{\mathbf{y}}+\mathbf{n})-\mathbf{x}_0}_2^2\Bigr)^{-q}\notag\\\label{proof:upper_estimate6}
						&\leq C\left(1+2^{j/2}\abs{\sin(\theta_{j,\ell}^{(i)}-\gamma)}\right)^{-5/2}\hspace{-0.1cm}\sum_{\mathbf{n}\in \mathbb{Z}^2}\Bigl(1+2^j\abs{2\pi(\widetilde{\mathbf{y}}+\mathbf{n})-\mathbf{x}_0}_2^2\Bigr)^{-q}
					\end{align}
					and we split up the infinite sum into 
					\begin{equation*}
					\Bigl(1+2^j\abs{2\pi\widetilde{\mathbf{y}}-\mathbf{x}_0}_2^2\Bigr)^{-q}+\sum_{\mathbf{n}\in \mathbb{Z}^2\setminus\{\mathbf{0}\}}\Bigl(1+2^j\abs{2\pi(\widetilde{\mathbf{y}}+\mathbf{n})-\mathbf{x}_0}_2^2\Bigr)^{-q}.
					\end{equation*}
					Using the same arguments, which led to \cref{proof:upper_estimate5}, we see that the infinite sum in the last equation is bounded from above by $C(q)\,2^{-jq}$ implying
					\begin{equation*}
						\abs{\left\langle f_Q^{2\pi},\psi^{(i)}_{j,\ell,\mathbf{y}}\right\rangle_2}\leq C(q)\,\left( 1+2^{j/2}\abs{\sin(\theta_{j,\ell}^{(i)}-\gamma)}\right)^{-5/2}\,\left(1+2^j\abs{\mathbf{x}_0-2\pi\widetilde{\mathbf{y}}}_2^2\right)^{-q}
					\end{equation*}
					in the case $Q\in \mathcal{Q}_j^1$.
					
					To finish the proof we use the decomposition \cref{eq:decomposition_T} and the fact that $\abs{\mathcal{Q}_j^0}\leq C\,2^{j}$ to get
					\begin{align*}
						\abs{\left\langle f^{2\pi},\psi^{(i)}_{j,\ell,\mathbf{y}}\right\rangle_2}&\leq\sum_{Q\in \mathcal{Q}_j^0}\abs{\left\langle f_Q^{2\pi},\psi^{(i)}_{j,\ell,\mathbf{y}}\right\rangle_2}+\sum_{Q\in \mathcal{Q}_j^1}\abs{\left\langle f_Q^{2\pi},\psi^{(i)}_{j,\ell,\mathbf{y}}\right\rangle_2}\\
						&\leq C_3(q)\sum_{Q\in \mathcal{Q}_j^1}\,\left( 1+2^{j/2}\abs{\sin(\theta_{j,\ell}^{(i)}-\gamma)}\right)^{-5/2}\,\left(1+2^j\abs{\mathbf{x}_0-2\pi\widetilde{\mathbf{y}}}_2^2\right)^{-q}.
					\end{align*}
					\qed\ \\
We proceed with the proof of \cref{thm:lower_estimate}. 

For $\mathbf{y}\in \mathcal{P}(\mathbf{N}_{j,\ell}^{(i)})$ let $\widehat{\mathcal{T}}_\mathbf{y}^{(i)}$ be the modified version of $\mathcal{T}^{(i)}$ as explained in the paragraph after \cref{eq:T_approx}. Since $\widehat{\mathcal{T}}_\mathbf{y}^{(i)}\in L_1(\mathbb{R}^2),\,i\in\lbrace \mathrm{h},\mathrm{v}\rbrace\,$ we use \cref{eq:poisson_fourier} to get
		\begin{equation*}
			c_\mathbf{k}((\widehat{\mathcal{T}}_\mathbf{y}^{(i)})^{2\pi})=\mathcal{F}\widehat{\mathcal{T}}_\mathbf{y}^{(i)}(\mathbf{k}),\qquad \mathbf{k}\in \mathbb{Z}^2.
		\end{equation*}
	  From the finite support of $\widehat{\mathcal{T}}_\mathbf{y}^{(i)}$ we deduce $\mathcal{F}\widehat{\mathcal{T}}_\mathbf{y}^{(i)}\in C^{2q}(\mathbb{R}^2)$ for all $q\in \mathbb{N}_0$. The smoothness assumption on the window $\Psi^{(i)}_{j,\ell}\in \mathcal{W}_2^{2q}$ implies $\mathcal{F}[\widehat{\mathcal{T}}_\mathbf{y}^{(i)}]\,\Psi_{j,\ell}^{(i)} \in C_0^{2q}(\mathbb{R}^2)$. Similar as in the proof of \cref{thm:upper_estimate} this product fulfills \cref{eq:decay_fourier_transform} and with Parseval's identity and the Poisson summation formula it follows 
		\begin{equation*}
			\left\langle \left(\widehat{\mathcal{T}}_\mathbf{y}^{(i)}\right)^{2\pi},\psi^{(i)}_{j,\ell,\mathbf{y}}\right\rangle_2=\sum_{\mathbf{n}\in \mathbb{Z}^2}\mathcal{F}^{-1}\left[ \mathcal{F}[\widehat{\mathcal{T}}_\mathbf{y}^{(i)}]\Psi^{(i)}_{j,\ell} \right](2\pi(\widetilde{\mathbf{y}}+\mathbf{n}))=\sum_{\mathbf{n}\in \mathbb{Z}^2}S(\mathbf{n}),
		\end{equation*}
		where
			\begin{equation*}
				S(\mathbf{n})\mathrel{\mathop:}=\int\limits_{\mathbb{R}^2}\mathcal{F}[\widehat{\mathcal{T}}_\mathbf{y}^{(i)}](\boldsymbol{\xi})\,\Psi^{(i)}_{j,\ell}(\boldsymbol{\xi})\,\mathrm{e}^{2\pi\mathrm{i}\boldsymbol{\xi}^{\mathrm{T}}(\widetilde{\mathbf{y}}+\mathbf{n})}\,\mathrm{d}\boldsymbol{\xi}.
			\end{equation*}
			Using again the decomposition \cref{eq:decomposition_T} for $\mathcal{T}$ we can repeat the arguments from the proof of \cref{thm:upper_estimate} to see that
			\begin{equation*}
				\sum_{\mathbf{n}\in \mathbb{Z}^2\setminus\{\mathbf{0}\}}\abs{S(\mathbf{n})}\leq C_1(q)\,2^{-jq}.
			\end{equation*}
		Assume that we can show
				\begin{equation}\label{eq:lower_bound}
					\abs{S(\mathbf{0})}\geq C_2(q).
				\end{equation}
		With the inverse triangle inequality we can deduce
		\begin{equation*}
			\abs{\left\langle \left(\widehat{\mathcal{T}}_\mathbf{y}^{(i)}\right)^{2\pi},\psi^{(i)}_{j,\ell,\mathbf{y}} \right\rangle_2}\geq \abs{S(\mathbf{0})}-\sum_{\mathbf{n}\in \mathbb{Z}^2\setminus\{\mathbf{0}\}}\abs{S(\mathbf{n})}\geq C_3(q)
		\end{equation*}
		and again with the inverse triangle inequality and \cref{lem:T_tilde} we finally get 
		\begin{equation*}
			\abs{\left\langle \mathcal{T}^{2\pi},\psi^{(i)}_{j,\ell,\mathbf{y}} \right\rangle_2}\geq\abs{\left\langle \left(\widehat{\mathcal{T}}_\mathbf{y}^{(i)}\right)^{2\pi}\hspace{-0.3cm},\psi^{(i)}_{j,\ell,\mathbf{y}} \right\rangle_2}-\abs{\left\langle \mathcal{T}^{2\pi}-\left(\widehat{\mathcal{T}}_\mathbf{y}^{(i)}\right)^{2\pi}\hspace{-0.3cm},\psi_{j,\ell,\mathbf{y}}^{(i)} \right\rangle_2}\geq C_4(A_0,q).
		\end{equation*}
		Thus it is left to show the existence of a constant $C_2(q)>0$ such that \cref{eq:lower_bound} is fulfilled.\\
		From \cref{eq:I_k,eq:F_inv} we recall the representation
		\begin{equation*}
			S(\mathbf{0})=\frac{2^j\,\mathrm{i}}{\left(2\pi\right)^2}\sum_{k=1}^{M}\int\limits_{0}^{\infty} \int\limits_{0}^{2\pi}\Psi_{j,\ell}^{(i)}(2^j\rho,\theta)\,
\mathrm{e}^{2\pi \mathrm{i} 2^j\rho \boldsymbol{\Theta}^{\mathrm{T}}(\theta)\,\widetilde{\mathbf{y}}}\,\mathcal{I}_k(2^j\rho,\theta)\,\mathrm{d}\theta\,\mathrm{d}\rho,
		\end{equation*}
		and consider only the case $i=\mathrm{h}$ since the other case is similar. First we use \cref{lem:orientation_lemma} and the inverse triangle inequality to see that $\abs{S(\mathbf{0})}$ is bounded from below by
		\begin{equation*}
			\frac{2^j}{\left(2\pi\right)^2}\Biggl(\Bigl\lvert\sum_{k\in \mathcal{M}^{(\mathrm{v})}}\int\limits_{0}^{\infty} \int\limits_{0}^{2\pi}\Psi_{j,\ell}^{(\mathrm{h})}(2^j\rho,\theta)\,\mathrm{e}^{2\pi \mathrm{i} 2^j\rho\boldsymbol{\Theta}^{\mathrm{T}}(\theta)\,\widetilde{\mathbf{y}}}\mathcal{I}_k(2^j\rho,\theta)\,\mathrm{d}\theta\mathrm{d}\rho\Bigr\vert-\abs{\mathcal{M}^{(\mathrm{h})}}C(n)2^{-jn}\Biggr),
		\end{equation*}
		where the last term is negligible for large $j$ and $n\in \mathbb{N}$. Assume that there is $k^*\in \mathcal{M}^{(\mathrm{v})}$ such that $\left[2\pi y_2-\pi\,2^{-j/2},2\pi y_2+2^{-j/2}\right]\subseteq\left[a_{k^*},b_{k^*}\right]$. In the following we omit the index $k^*$ for simplicity and let $\varepsilon=2^{-j/2}$. \cref{lem:localization_lemma} and the inverse triangle inequality lead to
		\begin{align*}
			&\abs{\sum_{k\in \mathcal{M}^{(\mathrm{v})}}\int\limits_{0}^{\infty} \int\limits_{0}^{2\pi}\Psi_{j,\ell}^{(\mathrm{h})}(2^j\rho,\theta)\,
			\mathrm{e}^{2\pi \mathrm{i} 2^j\rho \boldsymbol{\Theta}^{\mathrm{T}}(\theta)\,\widetilde{\mathbf{y}}}\mathcal{I}(2^j\rho,\theta)\,\mathrm{d}\theta\,\mathrm{d}\rho}\\
			&\geq\abs{\int\limits_{0}^{\infty} \int\limits_{0}^{2\pi}\int\limits_{2\pi y_2-\varepsilon}^{2\pi y_2+\varepsilon}\hspace{-0.3cm}\Psi_{j,\ell}^{(\mathrm{h})}(2^j\rho,\theta)\,\mathrm{e}^{\mathrm{i} 2^j\rho \boldsymbol{\Theta}^{\mathrm{T}}(\theta)\left(2\pi\widetilde{\mathbf{y}}-(f(x),x)^{\mathrm{T}}\right)}\,\boldsymbol{\Theta}^{\mathrm{T}}(\theta)\,\boldsymbol{\beta}(x)\,\mathrm{d}x\,\mathrm{d}\theta\,\mathrm{d}\rho}-C_2\,2^{-jn}.
		\end{align*}
		From the previous observations we conclude that, if we want to show $\abs{S(\mathbf{0})}\geq C(q)$, it is enough to find a constant $C_2(q)>0$ such that
		\begin{equation*}
			\abs{\int\limits_{0}^{\infty} \int\limits_{0}^{2\pi}\int\limits_{2\pi y_2-\varepsilon}^{2\pi y_2+\varepsilon}\Psi_{j,\ell}^{(\mathrm{h})}(2^j\rho,\theta)\,\mathrm{e}^{\mathrm{i} 2^j\rho \boldsymbol{\Theta}^{\mathrm{T}}(\theta)\left(2\pi\widetilde{\mathbf{y}}-(f(x),x)^{\mathrm{T}}\right)}\,\boldsymbol{\Theta}^{\mathrm{T}}(\theta)\,\boldsymbol{\beta}(x)\,\mathrm{d}x\,\mathrm{d}\theta\,\mathrm{d}\rho}\geq C_2(q)\,2^{-j}.
		\end{equation*}
		We write the last integral as
		\begin{align*}
			I&\mathrel{\mathop:}=\int\limits_{0}^{\infty} \int\limits_{0}^{2\pi}\int\limits_{2\pi y_2-\varepsilon}^{2\pi y_2+\varepsilon}\Psi_{j,\ell}^{(\mathrm{h})}(2^j\rho,\theta)\,\mathrm{e}^{\mathrm{i} 2^j\rho \boldsymbol{\Theta}^{\mathrm{T}}(\theta)\left(2\pi\widetilde{\mathbf{y}}-(f(x),x)^{\mathrm{T}}\right)}\,\boldsymbol{\Theta}^{\mathrm{T}}(\theta)\,\boldsymbol{\beta}(x)\,\mathrm{d}x\,\mathrm{d}\theta\,\mathrm{d}\rho\\
			&=\int\limits_{0}^{\infty} \bigg(\int\limits_{-\frac{\pi}{2}}^{\frac{\pi}{2}}
		+\int\limits_{\frac{\pi}{2}}^{\frac{3\pi}{2}}\bigg)\int\limits_{2\pi y_2-\varepsilon}^{2\pi y_2+\varepsilon}\Psi_{j,\ell}^{(\mathrm{h})}(2^j\rho,\theta)\mathrm{e}^{\mathrm{i} 2^j\rho \boldsymbol{\Theta}^{\mathrm{T}}(\theta)\left(2\pi\widetilde{\mathbf{y}}-(f(x),x)^{\mathrm{T}}\right)}\boldsymbol{\Theta}^{\mathrm{T}}(\theta)\boldsymbol{\beta}(x)\mathrm{d}x\mathrm{d}\theta\mathrm{d}\rho\\
		&=\mathrel{\mathop:}I_1+I_2.
		\end{align*}
		In the integral $I_2$ we substitute $\theta=\tau+\pi$ and use the symmetry properties of the univariate window functions $\widetilde{g}$ and $g$ to see 
		\begin{equation*}
			\widetilde{g}\left(\rho\cos(\tau+\pi)\right)=\widetilde{g}\left(-\rho\cos\tau\right)=\widetilde{g}\left(\rho\cos\tau\right)
		\end{equation*}
		and
		\begin{equation*}
			g\left(\rho\cos(\tau+\pi)(2^{j/2}\tan(\tau+\pi)-\ell)\right)=g\left(\rho\cos\tau(2^{j/2}\tan\tau-\ell)\right),
		\end{equation*}
		which lead together with $\boldsymbol{\Theta}(\tau+\pi)=-\boldsymbol{\Theta}(\tau)$ to $I=2\,\mathrm{i}\,\mathrm{Im}(I_1)=2\,\mathrm{i}\,\mathrm{Im}(I_2)$ since
		\begin{align*}
			I_2&=\int\limits_{0}^{\infty}\int\limits_{-\frac{\pi}{2}}^{\frac{\pi}{2}}\int\limits_{2\pi y_2-\varepsilon}^{2\pi y_2+\varepsilon}\Psi_{j,\ell}^{(\mathrm{h})}(2^j\rho,\tau+\pi)\,\mathrm{e}^{\mathrm{i}2^j\rho\boldsymbol{\Theta}^{\mathrm{T}}(\tau+\pi)\left(2\pi\widetilde{\mathbf{y}}-(f(x),x)^{\mathrm{T}}\right)}\,\boldsymbol{\Theta}(\tau+\pi)\,\boldsymbol{\beta}(x)\,\mathrm{d}x\,\mathrm{d}\tau\,\mathrm{d}\rho\notag\\
		&=-\int\limits_{0}^{\infty}\int\limits_{-\frac{\pi}{2}}^{\frac{\pi}{2}}\int\limits_{2\pi y_2-\varepsilon}^{2\pi y_2+\varepsilon}\Psi_{j,\ell}^{(\mathrm{h})}(2^j\rho,\theta)\,\mathrm{e}^{-\mathrm{i}2^j\rho\boldsymbol{\Theta}^{\mathrm{T}}(\theta)\left(2\pi\widetilde{\mathbf{y}}-(f(x),x)^{\mathrm{T}}\right)}\,\boldsymbol{\Theta}(\theta)\,\boldsymbol{\beta}(x)\,\mathrm{d}x\,\mathrm{d}\theta\,\mathrm{d}\rho=-\overline{I_1}.
		\end{align*}
		
		Let us first assume $A=\frac{1}{2}f''(2\pi y_2)>0$. The case $A<0$ is similar and will be omitted and the case $A=0$ will be discussed separately in the end of the proof. With \cref{lem:T_tilde} we can replace the function $f(x)$ locally for $\abs{x-2\pi y_2}<2^{-j/2}$ by  
		\begin{equation*}
			T_\mathbf{y}(x)=f(2\pi y_2)+m(x-2\pi y_2)+A(x-2\pi y_2)^2.
		\end{equation*}
		For every $\mathbf{x}_0=(f(x_0),x_0)^{\mathrm{T}}\in\partial T$ with $\abs{\mathbf{x}_0-2\pi\widetilde{\mathbf{y}}}_2\leq C\,2^{-j/2}$ we can write
		\begin{equation*}
			T_\mathbf{y}(x)=C+B(x-x_0)+A(x-x_0)^2,
		\end{equation*}
		where $B\mathrel{\mathop:}=m+2(x_0-2\pi y_2)$, $C\mathrel{\mathop:}=f(2\pi y_2)+m(x_0-2\pi y_2)+(x_0-2\pi y_2)^2$ and $m=f'(2\pi y_2)\in[-1,1]$. We choose $\mathbf{x}_0\in\partial T$ such that there is $\ell\in\left\lbrace-2^{j/2},\hdots,2^{j/2}\right\rbrace$ with $\abs{2^{j/2}B+\ell}\leq \frac{1}{4}$ and $\widetilde{y}_1=2^{-j}(z_1-\frac{1}{2}),\,z_1=-2^{j-1},\dots,2^{j-1}-1$ such that $\abs{2^j(2\pi\,\widetilde{y}_1-C)}\leq \frac{3\pi}{4}$.\\
		
		 We follow the ideas  of \cite[Section 3.2]{labate:detection} and change the variable to $v=x-x_0$. Thus we can rewrite $I_1$ as
		 \begin{align}\label{eq:I1_start}
			 I_1&=\int\limits_{0}^{\infty} \int\limits_{-\frac{\pi}{2}}^{\frac{\pi}{2}}\int\limits_{-2\varepsilon}^{2\varepsilon}\Psi^{(\mathrm{h})}_{j,\ell}(2^j \rho,\theta)\,\mathrm{e}^{-\mathrm{i}\, 2^j \rho \,\boldsymbol{\Theta}^{\mathrm{T}}(\theta)(Av^2+Bv+C-2\pi\widetilde{y}_1,v)^{\mathrm{T}}}\,\varphi(v)\, \mathrm{d}v \, \mathrm{d}\theta \, \mathrm{d}\rho\\
			 &=\int\limits_{0}^{\infty} \int\limits_{-\frac{\pi}{2}}^{\frac{\pi}{2}}\Psi^{(\mathrm{h})}_{j,\ell}(2^j \rho,\theta)\,\mathrm{e}^{\mathrm{i}\, 2^j\rho\,\cos{\theta}\,2\pi\widetilde{y}_1}\int\limits_{-2\varepsilon}^{2\varepsilon}\mathrm{e}^{\mathrm{i}\, \Lambda \, R(v)}\, \varphi(v)\, \mathrm{d}v \, \mathrm{d}\theta \, \mathrm{d}\rho\notag,
		 \end{align}
		where $\Lambda\mathrel{\mathop:}=2^j\rho$, $\varphi(v)\mathrel{\mathop:}= \Bigl(-1,f'(v+x_0)\Bigr)\,\boldsymbol{\Theta}(\theta)\, \beta(v+x_0)$ and
		\begin{align*}
			R(v)&\mathrel{\mathop:}=-\boldsymbol{\Theta}^{\mathrm{T}}(\theta) (Av^2+Bv+C,v)^{\mathrm{T}}\\
			&=-\cos\theta\left(A\bigg(v+\frac{B+\tan\theta}{2A}\bigg)^2 +C - 
		 \frac{(B+\tan\theta)^2}{4A}\right).
		\end{align*}
		The equation $R'(v)=0$ gives $v_{\theta}=-\frac{B+\tan\theta}{2A}$. 
		 Let $\phi(v)=R(v)-R(v_{\theta})$. Then $\phi(v_{\theta})=\phi'(v_{\theta})=0$ and
		 $\phi''(v_{\theta})=R''(v_{\theta})=-2A\cos\theta\neq0$. Hence we can write $I_1$ as
		 \begin{equation}\label{eq:I1}
			 I_1=\int\limits_{0}^{\infty} \int\limits_{-\frac{\pi}{2}}^{\frac{\pi}{2}}\Psi^{(\mathrm{h})}_{j,\ell}(2^j \rho,\theta)\,\mathrm{e}^{\mathrm{i}\, 2^j\rho\,\cos{\theta}\,2\pi\widetilde{y}_1}\mathrm{e}^{\mathrm{i}\, \Lambda\, R(v_{\theta})}\int\limits_{-2\varepsilon}^{2\varepsilon}\mathrm{e}^{\mathrm{i}\, \Lambda\, \phi(v)}\, \varphi(v)\,\mathrm{d}v\,\mathrm{d}\theta\,\mathrm{d}\rho.
		 \end{equation}
		We apply \cref{lem:constant_phase} for $t_0=v_{\theta}$, which leads to
		 \begin{align}\label{eq:I11}
		\int\limits_{-2\varepsilon}^{2\varepsilon}\mathrm{e}^{\mathrm{i}\, \Lambda\, \phi(v)}\, \varphi(v)\,\mathrm{d}v&=\left(\frac{2\pi \mathrm{i}}{\abs{\phi''(v_{\theta})}} \right)^{1/2}\varphi(v_{\theta})\,\Lambda^{-\frac{1}{2}}+O(\Lambda^{-1})\notag\\
		&=C\,\sqrt{\pi\mathrm{i}}\, (2^j\rho\,\abs{A\,\cos\theta})^{-\frac{1}{2}}\,\varphi(v_{\theta})+ O((2^j\rho)^{-1}).
		 \end{align}
		 
		From \cref{lem:support_Psi} we have $\rho\in\left[\frac{1}{2},2\right]$ so that the notation $O((2^j\rho)^{-1})$ can be identified with a function $r(j)$ such that $\abs{r(j)}\leq C_2\,2^{-j}$ as $j\geq j_0$. As explained in \cite[p. 115]{labate:detection} the constant $C_2>0$ is independent of $\theta$, $\rho$, $j$, $\ell$, $\widetilde{\mathbf{y}}$. With \cref{eq:I11} we can write the integral in \cref{eq:I1} as $I_1=I_{11}+I_{12}$, where
		 \begin{align*}
			 I_{11}&=C\,2^{-j/2}\sqrt{\pi\mathrm{i}}\int\limits_{0}^{\infty} \int\limits_{-\frac{\pi}{2}}^{\frac{\pi}{2}}\Psi^{(\mathrm{h})}_{j,\ell}(2^j \rho,\theta)\mathrm{e}^{\mathrm{i}\, 2^j\rho\,\cos{\theta}\,2\pi\widetilde{y}_1}\,\mathrm{e}^{\mathrm{i}\, \Lambda\, R(v_{\theta})} (\rho \abs{A\,\cos{\theta}})^{-\frac{1}{2}}\varphi(v_{\theta}) \,\mathrm{d}\theta \,\mathrm{d}\rho,\\
			 I_{12}&=C_2\,2^{-j}\int\limits_{0}^{\infty} \int\limits_{-\frac{\pi}{2}}^{\frac{\pi}{2}}\Psi^{(\mathrm{h})}_{j,\ell}(2^j \rho,\theta)\mathrm{e}^{\mathrm{i}\, 2^j\rho\,\cos{\theta}\,2\pi\widetilde{y}_1}\,\mathrm{e}^{\mathrm{i}\, \Lambda\, R(v_{\theta})}\mathrm{d}\theta \, \mathrm{d}\rho.
		 \end{align*}
		In the integrals $I_{11}$ and $I_{12}$ we substitute $t=2^{j/2}\tan\theta-\ell$ or equivalently $\theta\mathrel{\mathop:}=\theta_t=\theta_{j,\ell+t}^{(\mathrm{h})}$ leading to $\mathrm{d}\theta=2^{-j/2}\cos^2{\theta_t}\,\mathrm{d}t$. We remind that by \cref{lem:support_Psi}
		\begin{equation*}
			\mathrm{supp}\,\Psi^{(\mathrm{h})}_{j,\ell}(2^j\rho,\theta_t)\subset\left\lbrace(\rho,\theta)\in \mathbb{R}\times\left[-\frac{\pi}{2},\frac{\pi}{2}\right]:\frac{1}{3}<\abs{\rho}< 2,\,\theta_{j,\ell-2}^{(\mathrm{h})}<\theta_t<\theta_{j,\ell+2}^{(\mathrm{h})}\right\rbrace
		\end{equation*} 
		implying that $I_{11}=I_{12}=0$ for $\abs{t}>2$. With the last change of variable we have
		 \begin{align*}
			 I_{11}&=C\,2^{-j}\sqrt{\pi\mathrm{i}}\int\limits_{\frac{1}{3}}^{2} \int\limits_{-2}^{2}\widetilde{g}\left(\rho\cos\theta_t\right)\, g\left(t\,\rho\cos\theta_t\right)\,\mathrm{e}^{-\mathrm{i}\rho\cos\theta_t\left( 2^j\,C-\frac{1}{4A}\left(2^{j/2}B+\ell+t\right)^2-2^j\,2\pi\widetilde{y}_1\right)}\\
			 &\qquad\times(\rho \abs{A\,\cos\theta_t})^{-\frac{1}{2}}\,\varphi(v_{\theta_t})\,\cos^2{\theta_t}\,\mathrm{d}t\,\mathrm{d}\rho,\\
			 I_{12}&=2^{-3j/2}\int\limits_{\frac{1}{3}}^{2} \int\limits_{-2}^{2}\Psi^{(\mathrm{h})}_{j,\ell}(2^j \rho,\theta)\mathrm{e}^{-\mathrm{i}\rho\cos\theta_t\left( 2^j\,C-\frac{1}{4A}\left(2^{j/2}B+\ell+t\right)^2-2^j\,2\pi\widetilde{y}_1\right)}\cos^2\theta_t\,\mathrm{d}t\,\mathrm{d}\rho.
		 \end{align*}
		It is straightforward to see that $I_{12}$ is negligible since
		\begin{equation}\label{eq:estimate_I12}
			\abs{I_{12}}\leq C\,2^{-3j/2},
		\end{equation}
		where $C$ is independent of $j$, $\ell$ and $\mathbf{y}$. We use the notation 
		\begin{equation}\label{eq:choice_p_D}
			p\mathrel{\mathop:}=2^{j/2}B+\ell,\qquad\qquad D\mathrel{\mathop:}=2^j(2\pi\,\widetilde{y}_1-C)
		\end{equation} 
		and from the choice of $\mathbf{x}_0\in\partial T$ we have $\abs{p}\leq \frac{1}{4}$ and $\abs{D}\leq \frac{3\pi}{4}$. We show that for this choice inequality \cref{eq:lower_bound} is fulfilled.
		
		In the following, we adapt some of the ideas from \cite{labate:detection}. Since $\abs{2^{-j}t}\leq 2^{-j+1}$ for $\abs{t}\leq 2$ we have $\cos\theta_t=\mu_{j,\ell}+O(2^{-j/2})$ and $\sin\theta_t=(2^{-j/2}\ell)\,\mu_{j,\ell}+O(2^{-j/2})$, where $\mu_{j,\ell}\mathrel{\mathop:}=(1+(2^{-j/2}\ell)^2)^{-1/2}$ fulfilling $2^{-1/2}\leq\abs{\mu_{j,\ell}}\leq 1$. There exists sufficiently small $q_{j,\ell}$ such that $\abs{\beta(q_{j,\ell})-\beta(v_{\theta_t})}=O(2^{-j/2})$ and $\beta(q_{j,\ell})\neq 0$ and similarly we can approximate $\abs{f'(\tilde{q}_{j,\ell})-f'(v_{\theta_t})}=O(2^{-j/2})$ and $f'(\tilde{q}_{j,\ell})\neq0$. To get the lower bound for $I_{11}$, after ignoring the higher order decay term we can replace $\beta(v_{\theta_t})$ by a constant $\beta(q_{j,\ell})$, $f'(v_{\theta_t})$ by a constant $f'(\tilde{q}_{j,\ell})$, and $\cos\theta_t$ by the constant $\mu_{j,\ell}$. Hence using the notation $\delta_{j,\ell}\mathrel{\mathop:}=\beta(q_{j,\ell})(-\mu_{j,\ell}+f'(\tilde{q}_{j,\ell})(2^{-j/2}\ell)\mu_{j,\ell})$ and the substitution $\lambda=\rho\,\mu_{j,\ell}$ we can express $I_{11}$ as
		 \begin{align}
		I_{11}(D,p)&=C\,2^{-j}\,\mu_{j,\ell}^{3/2}\sqrt{\frac{\pi\,\mathrm{i}}{A}}\int\limits_{\frac{1}{3}}^{2}\int\limits_{-\infty}^{\infty}\widetilde{g}\left(\rho\,\mu_{j,\ell}\right)\,\mathrm{e}^{\mathrm{i}\,\rho\,\mu_{j,\ell}\left(D+\frac{1}{4A}(p+t)^2\right)}\,g\left(t\,\rho\,\mu_{j,\ell}\right)\,\rho^{-1/2}\,\mathrm{d}t\,\mathrm{d}\rho\notag\\\label{eq:H_I11}
		&=C_2\,2^{-j}\,\sqrt{\frac{\mathrm{i}}{A}}\int\limits_{\frac{1}{3}}^{\frac{4}{3}}\widetilde{g}(\lambda)\,\mathrm{e}^{\mathrm{i}\,D\, \lambda}\,\lambda^{-1/2}\,H(\lambda,p,A)\,\mathrm{d}\lambda
		 \end{align}
		 with
		  \begin{equation*}
		 H(\lambda,p,A)\mathrel{\mathop:}=\int\limits_{-\infty}^{\infty}g\left(t\,\lambda\right)\,\mathrm{e}^{\mathrm{i}\,\lambda\,
		 \frac{1}{4A}(p+t)^2}\mathrm{d}t=\int\limits_{-\infty}^{\infty}g\left((u-p)\,\lambda\right)\,\mathrm{e}^{\mathrm{i}\,\lambda\,
		\frac{u^2}{4A}}\,\mathrm{d}u.
		 \end{equation*}
		We want to emphasize the dependency of the integral $I_{11}$ on the the parameters $p$ and $D$ defined in \cref{eq:choice_p_D}.
		A direct computation with the change of variable $v=\lambda\frac{u^2}{4A}$ shows that
		 \begin{align*}
		H(\lambda,p,A)&=\sqrt{\frac{A}{\lambda}}\int\limits_{0}^{\infty}\left[g\left(2\sqrt{A\lambda\,v}+p\,\lambda\right)+g\left(2\sqrt{A\lambda\,v}-p\,\lambda\right)\right]\,\frac{\mathrm{e}^{\mathrm{i}\,v}}{\sqrt{v}}\,\mathrm{d}v\\
		&=\sqrt{\frac{A}{\lambda}}\Bigl(a(\lambda,p,A)+\mathrm{i}\,b(\lambda,p,A)\Bigr),
		\end{align*}
		where $a(\lambda,p,A)$ and $b(\lambda,p,A)$ are defined in \cref{def:a} and \cref{def:b}. With the representation of $H(\lambda,p,A)$ and the positive solution $\sqrt{\mathrm{i}}=\frac{1+\mathrm{i}}{\sqrt{2}}$ we can write the integral \cref{eq:H_I11} as $I_{11}(D,p)=\mathrm{Re}(I_{11}(D,p))+\mathrm{i}\,\mathrm{Im}(I_{11}(D,p))$ with
		 \begin{equation*}  
			 \mathrm{Im}(I_{11}(D,p))=C\,2^{-j}\int\limits_{\frac{1}{3}}^{\frac{4}{3}}\widetilde{g}(\lambda)\,\lambda^{-1}\,\Bigl( \Bigl[a(\lambda,p,A)+b(\lambda,p,A)\Bigr]\cos(D\lambda)+\Bigl[a(\lambda,p,A)-b(\lambda,p,A)\Bigr]\sin(D\lambda) \Bigr)\mathrm{d}\lambda.
		 \end{equation*}
		Using the connection $I_2=-\overline{I_1}$, we can start at \cref{eq:I1_start}, use again \cref{lem:constant_phase} and repeat all the previous steps for $I_2$ instead of $I_1$, to get $I_2=I_{21}+I_{22}$ with $\abs{I_{22}}\leq C\,2^{-3j/2}$ and $I_{21}(D,p)=\mathrm{Re}(I_{21}(D,p))+\mathrm{i}\,\mathrm{Im}(I_{21}(D,p))$ with
		 \begin{equation*}  \mathrm{Im}(I_{21}(D,p))=C\,2^{-j}\int\limits_{\frac{1}{3}}^{\frac{4}{3}}\widetilde{g}(\lambda)\,\lambda^{-1}\,\Bigl(\Bigl[a(\lambda,p,A)+b(\lambda,p,A)\Bigr]\sin(D\lambda)-\Bigl[a(\lambda,p,A)-b(\lambda,p,A)\Bigr]\cos(D\lambda) \Bigr)\mathrm{d}\lambda.
		 \end{equation*} 
		As a consequence of the relation $I=2\,\mathrm{i}\,\mathrm{Im}(I_1)=2\,\mathrm{i}\,\mathrm{Im}(I_2)$ we see that $I=2\,\mathrm{i}\,\mathrm{Im}(I_{11}+I_{12})=2\,\mathrm{i}\,\mathrm{Im}(I_{21}+I_{22})$. By the inverse triangle inequality, \cref{eq:estimate_I12} and its analog for $I_{22}$ we can use \cref{lem:I11_I21_positive} in order to finish the proof of the \cref{thm:lower_estimate} for $A>0$.
		
In the case $A=0$ we see that \cref{eq:I1_start} simplifies to
		\begin{equation*}
			 I_1 =\int\limits_{-\frac{\pi}{2}}^{\frac{\pi}{2}}\int\limits_{-2\varepsilon}^{2\varepsilon}\Psi^{(\mathrm{h})}_{j,\ell}(2^j \rho,\theta)\,\mathrm{e}^{-\mathrm{i}\,\rho \,\cos\theta\left(2^{j/2}v\left(2^{j/2}B+2^{j/2}\tan\theta\right)+2^{j}\left(C-2\pi\widetilde{y}_1\right)\right)}\,\varphi(v)\, \mathrm{d}v \, \mathrm{d}\theta \, \mathrm{d}\rho.
		\end{equation*}
		Note that \cref{lem:constant_phase} can not be applied in this case. Instead we use the substitutions $u=2^{j/2}v$ and similar to the previous case $t=2^{j/2}\tan\theta-\ell$ and $\lambda=\rho\,\mu_{j,\ell}$ together with the simplifications and notations from the last pages to arrive at the analogous integral to \cref{eq:H_I11}, which in this case is given by
		\begin{align*}
		 I_{1}(D,p)&=C\,2^{-j}\int\limits_{0}^{\infty} \int\limits_{-2}^{2}\int\limits_{-\infty}^{\infty}\widetilde{g}\left(\rho\cos\theta_t\right)\, g\left(t\,\rho\cos\theta_t\right)\,\mathrm{e}^{-\mathrm{i}\,\rho \,\cos\theta_t\bigl((p+t)u-D\bigr)}\,\mathrm{d}u\,\mathrm{d}t\,\mathrm{d}\rho\\
		&=C\,2^{-j}\int\limits_{0}^{\infty} \int\limits_{-2}^{2}\int\limits_{-\infty}^{\infty}\widetilde{g}\left(\lambda\right)\, g\left(t\,\lambda\right)\,\mathrm{e}^{-\mathrm{i}\,\lambda\bigl((p+t)u-D\bigr)}\,\mathrm{d}u\,\mathrm{d}t\,\mathrm{d}\lambda.
		\end{align*}
		Some direct calculations after the change of variable $y=t\,\lambda$ show that
		\begin{align*}
		 I_{1}(D,p)&=C\,2^{-j}\int\limits_{0}^{\infty}\widetilde{g}\left(\lambda\right)\,\lambda^{-1}\,\mathrm{e}^{\mathrm{i}\,D\,\lambda} \int\limits_{-\infty}^{\infty}\left(\int\limits_{-\infty}^{\infty}g\left(y\right)\,\mathrm{e}^{-\mathrm{i}\,y\,u}\,\mathrm{d}y\right)\,\mathrm{e}^{-\mathrm{i}\,p\,\lambda\,u}\,\mathrm{d}u\,\mathrm{d}\lambda\\
		&=C\,2^{-j}\int\limits_{0}^{\infty}\widetilde{g}\left(\lambda\right)\,\lambda^{-1}\,\mathrm{e}^{\mathrm{i}\,D\,\lambda} \left(\int\limits_{-\infty}^{\infty}\mathcal{F}g(u)\,\mathrm{e}^{-\mathrm{i}\,p\,\lambda\,u}\,\mathrm{d}u\right)\mathrm{d}\lambda\\
		&=C\,2^{-j}\int\limits_{0}^{\infty}\widetilde{g}\left(\lambda\right)\,\lambda^{-1}\,\mathrm{e}^{\mathrm{i}\,D\,\lambda}\,g(-p\lambda)\,\mathrm{d}\lambda
		\end{align*}
		and since $g(-p\lambda)=1$ for $\lambda\in\left[\frac{1}{3},\frac{4}{3}\right]$ and $p\in\left[-\frac{1}{4},\frac{1}{4}\right]$ this implies
		\begin{equation*}
			\bigl\lvert\mathrm{Im}(I_{1}(D,p))\bigr\rvert=C\,2^{-j}\int\limits_{0}^{\infty}\widetilde{g}\left(\lambda\right)\,\lambda^{-1}\,\sin(D\lambda)\,\mathrm{d}\lambda>0
		\end{equation*}
		for $0<\abs{D}\leq\frac{3\pi}{4}$. For the case $D=0$, we slightly modify the function $\widetilde{g}$ to make it odd. Then with a similar argument as before we see that $I=2\mathrm{Re}(I_1(0,p))>0$.
		\qed 
		
\section{Generalizations and possible extensions} 
\label{sec:generalizations_and_possible_extentions}

In this paper we showed that trigonometric polynomial shearlets based on the construction of multivariate periodic de la Vall\'{e}e Poussin-type wavelets are able to detect step discontinuities along boundary curves of characteristic functions.

Since the constructions and results in \cite{bergmann:dlVP} are given in $d$ dimensions, there is a natural extension of the trigonometric polynomial shearlets to higher dimensions. If for example the dimension is $d=3$, the multivariate window functions become
$\Psi^{(1)}(\mathbf{x})\mathrel{\mathop:}=\widetilde{g}(x_1)\,g(x_2)\,g(x_3),\;\;\Psi^{(2)}(\mathbf{x})\mathrel{\mathop:}=g(x_1)\,\widetilde{g}(x_2)\,g(x_3),\;\;\Psi^{(3)}(\mathbf{x})\mathrel{\mathop:}=g(x_1)\,g(x_2)\,\widetilde{g}(x_3)$. 

For even $j\in \mathbb{N}_0$ and $\boldsymbol{\ell}=(\ell_1,\ell_2)^{\mathrm{T}}\in \mathbb{Z}^2$ with $\lvert\ell_1\rvert\leq 2^{j/2}$ and $\lvert\ell_2\rvert\leq 2^{j/2}$ the matrices analog to \cref{eq:N_jl} are given by
		\begin{equation*}
			\mathbf{N}_{j,\boldsymbol{\ell}}^{(1)}\mathrel{\mathop:}=\begin{pmatrix}
				2^j & \ell_1\, 2^{j/2} & \ell_2\, 2^{j/2}\\
				0 & 2^{j/2} & 0\\
				0 & 0 & 2^{j/2}
			\end{pmatrix},\qquad\qquad\mathbf{N}_{j,\boldsymbol{\ell}}^{(2)}\mathrel{\mathop:}=\begin{pmatrix}
				2^{j/2} & 0 & 0\\
				\ell_1\,2^{j/2} & 2^j & \ell_2\, 2^{j/2}\\
				0 & 0 & 2^{j/2}
			\end{pmatrix}
			\end{equation*}
			and
			\begin{equation*}
				\mathbf{N}_{j,\boldsymbol{\ell}}^{(3)}\mathrel{\mathop:}=\begin{pmatrix}
								2^{j/2} & 0 & 0\\
								0 & 2^{j/2} & 0\\
								\ell_1\,2^{j/2} & \ell_2\, 2^{j/2} & 2^j
							\end{pmatrix}.
			\end{equation*}
			We define the three-dimensional trigonometric polynomial shearlets by
			\begin{equation*}
				\psi_{j,\boldsymbol{\ell},\mathbf{y}}^{(i)}(\mathbf{x})\mathrel{\mathop:}=\sum\limits_{\mathbf{k}\in \mathbb{Z}^3}\Psi^{(i)}_{j,\boldsymbol{\ell}}(\mathbf{k})\,\mathrm{e}^{\mathrm{i}\mathbf{k}^{\mathrm{T}}(\mathbf{x}-2\pi\widetilde{\mathbf{y}})},\quad i\in \left\lbrace 1,2,3 \right\rbrace.
			\end{equation*}
As in the two-dimensional case, this construction is similar to the classical shearlets and its higher-dimensional generalizations.

The authors in \cite{labate:detection_3d} proved in detail that continuous shearlet systems in three dimensions are able to detect boundary curves of piecewise smooth surfaces. As remarked in \cite{labate:detection}, an analogous result holds for discrete shearlets in dimension $3$. We are convinced that it should be possible to derive a similar result for trigonometric polynomial shearlets, but a detailed proof is not in the focus of this paper.

Another interesting open question is the behavior of the shearlet coefficients near corner points. If $\boldsymbol{\gamma}:[0,2\pi)\rightarrow\partial T$ is a parametrization of the boundary $\partial T$, we call $\mathbf{x}_0=\boldsymbol{\gamma}(t_0)\in\partial T$ a corner point, if $\boldsymbol{\gamma}'(t_0^+)\neq\pm\boldsymbol{\gamma}'(t_0^-)$. For continuous shearlets, this question was answered in \cite{labate:detection_continuous} and in a more general setting in the context of parabolic molecules in \cite{grohs:molecules}. As far as we know, there is no result for corner points in the discrete setting until now. It would be very interesting to investigate in which way the techniques of the continuous setting can be combined with the ideas of this paper to prove similar results for discrete shearlets. We will leave this question as a topic for future research.

In many applications, such as image processing, the functions to be analyzed are piecewise smooth and not characteristic functions of sets as discussed in this paper. In \cite{grohs:molecules,labate:smooth} it was shown that the continuous shearlet coefficients of functions of the form $B(\mathbf{x})=f(\mathbf{x})\,\chi_T(\mathbf{x})$ with $f\in C^{\infty}(\mathbb{R}^2)$ exhibit the same decay rate as \cref{eq:lower_bound_cont} if $\mathbf{p}\notin\partial T$ or if $s = s_0$ does not correspond to the normal direction of $\partial T$ at $\mathbf{p}$. If $\mathbf{p}\in\partial T$ and $s = s_0$ corresponds to the normal direction of $\partial T$ at $\mathbf{p}$, then 
\begin{equation*}
	0<\lim\limits_{a\rightarrow 0^+}a^{-(n/2+3/4)}\mathcal{SH}_{\psi}B(a,s_0,\mathbf{p})<\infty,
\end{equation*} 
where $n$ denotes the number of vanishing derivatives of $f$ at $\mathbf{p}$. As in the case of corner points, there is no analogous result for discrete shearlet systems yet. To give a proof for the case of piecewise smooth functions is again beyond the scope of this paper and will be addressed in a forthcoming article.\\

\textbf{Acknowledgement.} We would like to thank the referees for their valuable comments and remarks. \\
The authors were supported by H2020-MSCA-RISE-2014 Project number 645672 (AMMODIT: Approximation Methods for Molecular Modelling and Diagnosis Tools).


\end{document}